\documentclass[11pt,oneside]{article}
\usepackage{amssymb,amsmath,latexsym,amsfonts,amscd,amsthm,multirow,ctable,colortbl,mathdots,caption,array}
\usepackage{wasysym}
\usepackage{fancyhdr,tabularx,cite,mathrsfs,graphicx}
\usepackage{authblk}
\usepackage[headings]{fullpage}
\usepackage{stmaryrd}
\usepackage[all]{xy}
\usepackage{stmaryrd,rotating}
\usepackage{pdflscape,longtable}

\usepackage[font=footnotesize]{caption}

\usepackage{setspace}

\usepackage{hyperref}

\newcommand*\cbottomrule[1]{\cmidrule[\heavyrulewidth]{#1}}
\newcommand*\ctoprule[1]{\addlinespace\cmidrule[\heavyrulewidth]{#1}}

\newcommand{\bea}{\begin{eqnarray}} 
\newcommand{\eea}{\end{eqnarray}} 
\newcommand{\bee}{\begin{eqnarray*}} 
\newcommand{\eee}{\end{eqnarray*}} 
\newcommand{\al}{\begin{align*}} 
\newcommand{\eal}{\end{align*}} 
\newcommand{\be}{\begin{equation}} 
\newcommand{\ee}{\end{equation}} 
\newcommand{\eq}[1]{(\ref{#1})} 
\newcommand{\bem}{\begin{pmatrix}} 
\newcommand{\eem}{\end{pmatrix}} 

\newcommand{\pwrm}[3]{#1\lvert[-\tfrac{#2}{#3}]} 
\newcommand{\pwr}[3]{#1\lvert[\tfrac{#2}{#3}]}

\def\a{\alpha} 
 
\def\c{\gamma} 
\def\d{\delta}

\def\f{\phi}

\def\inf{\infty}

\def\l{\lambda} 
\def\m{\mu}

\def\pa{\partial}        

\def\r{\rho}                  
\def\s{\sigma}            
\def\t{\tau} 
\def\th{\theta}

\def\D{\Delta}

\def\O{\Omega}

\newcolumntype{R}{ >{$}r <{$}}
\newcolumntype{C}{ >{$}c <{$}}
\newcolumntype{L}{ >{$}l <{$}}
\newcolumntype{F}{>{\centering\arraybackslash}m{1.5cm}}

\def\ll{\ell}

\newcommand{\gt}[1]{\mathfrak{#1}}

\newcommand{\comment}[1]{}

\newcommand{\RR}{{\mathbb R}}%Reals
\newcommand{\CC}{{\mathbb C}}%Complex
\newcommand{\ZZ}{{\mathbb Z}}%Integers
\newcommand{\ZZp}{\ZZ^+}%Positive integers
\newcommand{\QQ}{{\mathbb Q}}%Rationals
\newcommand{\HH}{{\mathbb H}}%quaternions
\newcommand{\JJ}{{\mathbb J}}%mock Jacobi forms

\newcommand{\lab}{{\langle}}    %Left angle brackets
\newcommand{\rab}{{\rangle}}    %Right angle brackets

\newcommand{\End}{\operatorname{End}}

\newcommand{\Id}{\operatorname{Id}}

\newcommand{\Supp}{\operatorname{Supp}}
\newcommand{\sgn}{\operatorname{sgn}}
\newcommand{\ex}{\operatorname{e}} %Number theory exp

\newcommand{\Ex}{\operatorname{Ex}}
\newcommand{\wk}{{\rm wk}}
\newcommand{\sk}{{\rm sk}}

\newcommand{\new}{{\rm new}}
\newcommand{\reg}{{\rm reg}}
\renewcommand{\top}{{\rm top}}
\newcommand{\opt}{{\rm opt}}
\newcommand{\Dp}{\mathbb{D}^+}

\newcommand{\E}{\mathcal{E}}
\newcommand{\W}{\mathcal{W}}
\newcommand{\xmod}{{\rm \;mod\;}}

\newcommand{\SZ}{\operatorname{\textsl{SZ}}} %Skoruppa--Zagier embedding
\newcommand{\Th}{\Theta}

\newcommand{\SL}{\operatorname{\textsl{SL}}}      %SL group
\newcommand{\mpt}{\widetilde{\SL}_2}      %Metaplectic double cover of \SL_2
\newcommand{\GL}{{\textsl{GL}}}      %GL group
\newcommand{\G}{\Gamma}	%Gamma
\newcommand{\g}{\gamma}	%gamma

\newtheorem{thm}{Theorem}[subsection]
\newtheorem*{thm*}{Theorem}
\newtheorem{cor}[thm]{Corollary}
\newtheorem{lem}[thm]{Lemma}
\newtheorem{prop}[thm]{Proposition}

\theoremstyle{definition}

\theoremstyle{remark}
\newtheorem*{rmk}{Remark}

\theoremstyle{theorem}
\newtheorem{quest}{Question}

\numberwithin{equation}{subsection}
\begin{document}

\setstretch{1.4}

\title{
\vspace{-35pt}
\textsc{\huge{ {O}ptimal {M}ock {J}acobi {T}heta {F}unctions}}
    }

\renewcommand{\thefootnote}{\fnsymbol{footnote}} 

%11F11:Holomorphic modular forms of integral weight
%11F27:Theta series; Weil representation; theta correspondences
%11F37:Forms of half-integer weight; nonholomorphic modular forms
%11F50:Jacobi forms

\renewcommand{\thefootnote}{\arabic{footnote}} 

\author[1]{Miranda C. N. Cheng\thanks{mcheng@uva.nl}}
\author[2]{John F. R. Duncan\thanks{john.duncan@emory.edu}}

\affil[1]{Institute of Physics and Korteweg-de Vries Institute for Mathematics\\
University of Amsterdam, Amsterdam, the Netherlands\footnote{On leave from CNRS, France.}}
\affil[2]{Department of Mathematics and Computer Science\\
Emory University, Atlanta, GA 30322, USA}

\date{} 

\maketitle

\abstract{
We classify the optimal mock Jacobi forms of weight one with rational coefficients. The space they span is thirty-four-dimensional, and admits a distinguished basis parameterized by genus zero groups of isometries of the hyperbolic plane. We show that their Fourier coefficients can be expressed explicitly in terms of singular moduli, and obtain positivity conditions which distinguish the optimal mock Jacobi forms that appear in umbral moonshine. We find that all of Ramanujan's mock theta functions can be expressed simply in terms of the optimal mock Jacobi forms with rational coefficients.
}

\clearpage

\tableofcontents

\clearpage

%------------------------------------------------------------------%
\section{Introduction}\label{sec:intro}
%------------------------------------------------------------------%

Modular forms are a cornerstone of modern mathematics, playing crucial roles in numerous subfields of number theory, geometry, topology and physics. 
Wiles' proof of Fermat's last theorem, 
Deligne's proof of the Weil conjectures, 
the construction of topological modular forms by Hopkins--Mahowald--Miller, 
and the mirror symmetry conjectures of Kontsevich and Strominger--Yau--Zaslow are each examples of 
profound developments that are intimately connected to the modular form theory.

In this century a broader theory of mock modular forms has been initiated. Among other things, this furnishes an abstract setting in which both classical modular forms and the mock theta functions of Ramanujan may be studied side by side.
In this work we obtain a ``genus zero'' classification of mock modular forms with rational coefficients satisfying a growth condition, and thereby give a new perspective on some of the most well-known and important examples.

\subsection{Questions}\label{sec:intro:que}

Ramanujan first communicated 
the notion of a mock theta function in his ``last letter
to Hardy,'' shortly before succumbing to illness in 1920, but 
left us little more 
than a collection of examples. Some further examples 
came to light in 1976, when Andrews recovered the ``lost notebook,'' 
but it was not until Zwegers' breakthrough doctoral thesis of 2002 that we were able to situate the mock theta functions within a structured theory. 

Thanks to Zwegers' results \cite{zwegers}---and the contributions of others, including Bringmann--Ono, Bruinier--Funke, and Zagier, to name a few---we can now say that Ramanujan's mock theta functions are, up to minor modifications, examples of (weakly holomorphic\footnote{There seem to be few examples of mock modular forms that are bounded near cusps, so the qualifier ``weakly holomorphic'' is usually omitted.}) mock modular forms of weight $\frac12$ whose shadows are unary theta functions. The new mock modular theory 
allows for easy proofs of some classical conjectures, and has already found deep 
applications in diverse areas of mathematics and physics. 
These developments, and some of the recent applications, are exposited in more detail in 
\cite{zagier_mock,Ono_unearthing,Folsom_what,MR3242661}.

As is typical, 
the new results have revealed new questions.
For example, why do Ramanujan's mock theta functions have integral coefficients? It is manifest that they do from the way he wrote them down, but results of Bruinier--Ono  
\cite{MR2726107} 
demonstrate that, for a generic mock modular form, infinitely many of its coefficients are transcendental.
\begin{quest}\label{que:intro:rational}
{What is the theoretical reason that Ramanujan's mock theta functions are 
rational?}
\end{quest}

A second question arises from the theory of umbral moonshine \cite{UM,MUM}, which attaches mock modular forms of weight $\frac12$  to the automorphisms of the Niemeier lattices. It was observed in \cite{MUM} that about half of Ramanujan's examples are recovered in this way. 
\begin{quest}\label{que:intro:moonshine}
{What abstract property, if any, distinguishes the mock theta functions that appear in moonshine?}
\end{quest}

The original motivation for the present work was yet another question, also originating in moonshine. It was noted in \cite{MUM} 
that the integers $m> 1$ such that the modular curve $X_0(m)$ has genus zero are exactly the $m$ for which there exists a self-dual even positive-definite lattice of rank $24$---i.e., a {\em Niemeier lattice}---containing the root system of type $A_{m-1}$. Extending this, 
the characteristic polynomials of Coxeter elements for irreducible simply laced root systems 
were used to attach
a genus zero subgroup of $\SL_2(\RR)$ 
to each of the Niemeier lattices (with roots) in \S2.3 of \cite{MUM}. 
Taking the vector-valued mock modular form attached to the identity automorphism of each such lattice in \S\S4-5 of \cite{MUM} we obtain a correspondence between mock modular forms (for the full modular group) and  subgroups of $\SL_2(\RR)$.
\begin{quest}\label{que:intro:genuszero}
Is there a conceptual explanation for this association of mock modular forms to genus zero groups?
\end{quest}

In this work we answer, in part, each of the above three questions, on rationality, moonshine, and genus zero groups. We achieve this through an analysis of 
vector-valued mock modular forms of weight $\frac12$ that 
have
minimal possible growth in their coefficients. 
We call this growth condition optimality.
By applying a Waldspurger type formula obtained by Skoruppa we are able to use modular $L$-functions to control the mock modular forms we consider. 
This leads to our first main result: a classification which attaches optimal mock modular forms to genus zero subgroups of $\SL_2(\RR)$. Subsequently we 
use Skoruppa's formula together with the Gross--Zagier formula and a theorem of Bruinier--Ono to prove our second main result: that any optimal mock modular form with rational coefficients must lie in the rational span of our ``genus zero'' examples.  

We also establish constructive relationships between the optimal mock modular forms in our classification and their corresponding genus zero groups, and formulate criteria that distinguish the mock modular forms of \cite{MUM} within the larger class that is obtained herein. Finally, we observe that all the mock theta functions of Ramanujan admit simple expressions in terms of the mock modular forms appearing in our classification.

Next, in \S\ref{sec:intro:res}, we present our main results in more detail.
As we elaborate in \S\ref{sec:intro:met}, our first main result answers Question \ref{que:intro:genuszero}, and our two main results taken together answer Question \ref{que:intro:rational}.
We discuss the relationship to moonshine, and our answer to Question \ref{que:intro:moonshine}, in \S\ref{sec:intro:moon}, and describe constructive relationships between genus zero groups and mock modular forms in \S\ref{sec:intro:sing}.

\subsection{Main Results}\label{sec:intro:res}

As our title indicates, we employ the language of Jacobi forms in this work. This reflects a restriction
to vector-valued forms of a certain kind. 
Precisely, a non-mock example of the mock modular forms we focus on would be a holomorphic function $h(\tau)=(h_r(\tau))_{r\xmod 2m}$, for $\tau$ in the upper half-plane, such that $\phi:= \sum_{r\xmod 2m} h_r\theta_{m,r}$ satisfies
\begin{gather}\label{eqn:intro:jacwt1indm}
\f(\t, z+\l \t +\m)
e^{2\pi i m (\l^2\t+2\l z)}
=	\f\left(\frac{a\t+b}{c\t+d},\frac{z}{c\t+d}\right)
\frac{e^{-2\pi i m \frac{cz^2}{c\tau+d}}}{c\t+d} 
=\f(\tau,z),
\end{gather}
for $\lambda,\mu\in \ZZ$ and $\left(\begin{smallmatrix}a&b\\c&d\end{smallmatrix}\right)\in \SL_2(\ZZ)$, when 
$\tau,z\in \CC$ with $\Im(\tau)>0$. The $\th_{m,r}$ here are the theta functions naturally attached to the even positive-definite lattices of rank $1$. 
(See \S\ref{sec:notn} for an explicit definition.)
Optimality 
is the condition that
\begin{gather}\label{eqn:intro:opt}
h_r(\tau)=O(q^{-\frac1{4m}})
\end{gather} 
as $\Im(\t)\to \infty$, for each $r\xmod 2m$, where $q=e^{2\pi i \tau}$. 

If we require both (\ref{eqn:intro:jacwt1indm}) and (\ref{eqn:intro:opt}), then $\f$ is an {\em optimal weak Jacobi form} of weight $1$ and index $m$, which is necessarily zero according to a result of Dabholkar--Murthy--Zagier (cf. \S\ref{sec:class:defn}). So we must relax one of these two conditions, in order to meet non-trivial examples.
In this work we leave 
(\ref{eqn:intro:opt}) unchanged, but relax the modularity condition (\ref{eqn:intro:jacwt1indm}) by requiring that 
it hold only after $\f$ is replaced by $\hat\f:=\sum_{r\xmod 2m} \hat{h}_r\th_{m,r}$, where 
\begin{gather}\label{eqn:intro:comph}
\hat{h}(\tau):=h(\tau)+e^{-\frac{\pi i}4}\frac{1}{\sqrt{4m}}
\int_{-\bar \tau}^\inf (\t'+\t)^{-\frac12} \overline{g(-\overline{\t'})} {\rm d}\t'
\end{gather}
for some cuspidal modular form $g(\tau)=(g_r(\tau))_{r\xmod 2m}$ with weight $\frac32$ (which is unique if it exists). 
In the language of mock modular forms, $\hat{h}$ is called the completion of $h$, and $g$ is called its shadow. 

With the above definition, a holomorphic function $\f=\sum_r h_r\th_{m,r}$ 
such that (\ref{eqn:intro:jacwt1indm}) is satisfied with $\hat\f:=\sum_r \hat{h}_r\th_{m,r}$ in place of $\f$ 
is called a {\em mock Jacobi form} of weight 1 and index $m$, following \cite{Dabholkar:2012nd}, except that we usually also require a growth condition on the {\em theta-coefficients} $h_r$. Say that $\f$ is a {\em weak} mock Jacobi form if 
$h_r(\tau)=O(q^{-\frac{r^2}{4m}})$ as $\Im(\tau)\to\infty$ for all $r$.
We say that $\f$ is {\em optimal} if the stronger\footnote{Our use of the term optimal is inspired by \cite{Dabholkar:2012nd}, but their definition is weaker than ours.} condition (\ref{eqn:intro:opt}) is satisfied.

We may, in principle, consider mock Jacobi forms $\f$ of weight 1 
such that the $h_r$ are bounded as $\Im(\tau)\to \infty$, but we will see in \S\ref{sec:class:defn} (cf. Proposition \ref{prop_uniqueness1}) that there are no non-zero examples. So we must allow some growth in the $h_r$ near cusps. 
It follows from (\ref{eqn:intro:jacwt1indm}) for $\hat\f$ that $\f$ has a Fourier expansion in integer powers of $q$. The theta functions $\th_{m,r}$ admit Fourier expansions in integer powers of $q^{\frac1{4m}}$, so the optimality condition (\ref{eqn:intro:opt}) is the minimal possible weakening of boundedness that we can consider.

If $\f=\sum_rh_r\th_{m,r}$ is a weak mock Jacobi form of weight 1 as above then $\s:=\sum_r \overline{g_r}\th_{m,r}$, called the {\em shadow} of $\f$, is a skew-holomorphic Jacobi form of weight 2 in the sense of Skoruppa, satisfying
\begin{gather}\label{eqn:intro:skjacwt2indm}
\s(\t, z+\l \t +\m)
e^{2\pi i m (\l^2\t+2\l z)}
=	\s\left(\frac{a\t+b}{c\t+d},\frac{z}{c\t+d}\right)
\frac{e^{-2\pi i m \frac{cz^2}{c\tau+d}}}{(c\bar\t+d)|c\t+d|} 
=\s(\tau,z),
\end{gather}
for $\lambda,\mu\in \ZZ$ and $\left(\begin{smallmatrix}a&b\\c&d\end{smallmatrix}\right)\in \SL_2(\ZZ)$.
It develops that there are non-zero skew-holomorphic Jacobi forms $\s=\sum_r \overline{g_r}\th_{m,r}$ 
such that the coefficient of $q^{\frac{D}{4m}}$ in the Fourier expansion of each $g_r$ vanishes unless $D$ is a perfect square.
We call these {\em theta type} skew-holomorphic Jacobi forms, and we say that a mock Jacobi form $\f$ is a {\em mock Jacobi theta function} if its shadow has theta type. This terminology is motivated, in part, 
by Zagier's modern definition (cf. \S5 of \cite{zagier_mock}) of a mock theta function: namely, a mock modular form of weight $\frac12$ (for some subgroup of the modular group) whose shadow is a unary theta series. It is also motivated 
by the empirical observation (cf. \S\ref{sec:mockthetafunctions}) that all of the mock theta functions of Ramanujan admit simple expressions in terms of theta-coefficients of mock Jacobi theta functions.

As we have mentioned, there are no non-vanishing 
weak Jacobi forms of weight 1 that are optimal. Moreover, a vanishing proportion of the Fourier coefficients of a theta type skew-holomorphic Jacobi form are non-zero. So heuristically we may regard the optimal mock Jacobi theta functions of weight 1 as the optimal mock Jacobi forms of weight 1 that are ``as close as possible'' to being weak Jacobi forms.

Our first main result 
classifies the mock Jacobi theta functions of weight 1 that are optimal by producing a basis for the space they comprise which is indexed by genus zero subgroups of $\SL_2(\RR)$. 
Our second main result identifies the optimal mock Jacobi theta functions as the span of the optimal mock Jacobi forms that have rational Fourier coefficients.

To formulate these results precisely define $\JJ^\opt_{1,m}$ to be the space of optimal mock Jacobi forms of weight $1$ and index $m$, 
and define $\JJ_{1,m}^\top$ to be the subspace composed of optimal mock Jacobi theta functions. 
The group $O_m:=\{a\xmod 2m\mid a^2=1\xmod 4m\}$ acts naturally on $\JJ^\opt_{1,m}$, 
according to the rule $\f\cdot a:=\sum_r h_r\th_{m,ra}$ for $\f=\sum_rh_r\th_{m,r}$, and this action preserves the subspace $\JJ^\top_{1,m}$. Thus it is natural to consider the eigenspace decomposition
\begin{gather}
\JJ_{1,m}^\top=\bigoplus_{\a\in \widehat{O}_m}\JJ_{1,m}^{\top,\a}
\end{gather}
where $\widehat{O}_m$ is the group of irreducible characters of $O_m$. 

The genus zero connection manifests a natural association of subgroups of $\SL_2(\RR)$ to elements of $\widehat{O}_m$. To describe this, note that $O_m$ is in natural correspondence with the {\em exact} divisors of $m$, for if $n$ divides $m$ and is coprime to $m/n$ then there is a unique $a=a(n)$ in $O_m$ such that $a=-1\xmod 2n$ and $a=1\xmod 2m/n$, and all $a\in O_m$ arise in this way. 
Thus, given $\a\in \widehat{O}_m$, we may consider 
the subgroup of $\SL_2(\RR)$ generated by the Atkin--Lehner involutions of $\Gamma_0(m)$ associated to exact divisors of $m$ corresponding to elements in the kernel of $\a$. 
Explicitly, we set
\begin{gather}\label{eqn:intro:Gamma0malpha}
	\Gamma_0(m)+\ker(\a):=
	\left\{
		\frac1{\sqrt{n}}\begin{pmatrix}
		an&b\\cm&dn
		\end{pmatrix}
		\mid
		adn-bcm/n=1,\;
		a(n)\in\ker(\a)
	\right\}
\end{gather}
where $a,b,c,d$ are assumed to be integers, $n$ is assumed to be an exact divisor of $m$, and $n\mapsto a(n)$ is the map just described. 
The orbits of $\Gamma_0(m)+\ker(\a)$ on the upper half-plane naturally constitute a Riemann surface, which admits a natural compactification. We say that $\Gamma_0(m)+\ker(\a)$ is {\em genus zero} if the corresponding compact Riemann surface is a sphere, and we say that $\Gamma_0(m)+\ker(\a)$ is {\em non-Fricke} if $a(m)=-1$ does not belong to $\ker(\a)$.

We are now ready to state our two main results.
\begin{thm}\label{thm:intro-firstmainthm}
Let $m$ be a positive integer and $\a\in \widehat{O}_m$. If 
$\Gamma_0(m)+\ker(\a)$ is non-Fricke and genus zero then $\dim \mathbb{J}_{1,m}^{\top,\a}=1$. Otherwise, $\dim \mathbb{J}_{1,m}^{\top,\a}=0$.
\end{thm}
\begin{thm}\label{thm:intro-rat}
Let $m$ be a positive integer, $\a\in \widehat{O}_m$ and $\f\in \JJ^{\opt,\a}_{1,m}$. If $\f\in\JJ^{\top,\a}_{1,m}$ then there exists $C\in \CC$ such that all the Fourier coefficients of $C\f$ are rational integers. Otherwise, $\f$ has transcendental Fourier coefficients. 
\end{thm}

Using Theorem \ref{thm:intro-firstmainthm} we may canonically (up to scale) attach a non-zero optimal mock Jacobi theta function to each genus zero group of the form $\Gamma_0(m)+\ker(\a)$. It develops (cf. Corollary \ref{cor:class:proof-dimJtop}) that there are $34$ such groups, and Theorem \ref{thm:intro-firstmainthm} confirms that the corresponding optimal mock Jacobi theta functions furnish a basis for $\JJ^\top_{1,*}:=\bigoplus_m\JJ^\top_{1,m}$.

From Theorem \ref{thm:intro-rat} it follows (cf. Corollay \ref{cor:class:rpm-theta}) that any optimal mock Jacobi form with rational coefficients is a rational linear combination of optimal mock Jacobi theta functions. So in particular, 
the optimal mock Jacobi theta functions may be characterised as the optimal mock Jacobi forms that are linear combinations of optimal mock Jacobi forms with integer coefficients.

Before proceeding to describe our methods we comment on interesting related work \cite{Bruinier201738} which appeared shortly after the first version of the present article was released. Part (1) of Theorem 4.7 in loc. cit. may be regarded as a rationality result for Fourier coefficients of mock Jacobi theta functions that is, in some sense, more general than our Theorem \ref{thm:intro-rat}. However, Theorem \ref{thm:intro-rat} does not follow from the results of \cite{Bruinier201738}. For example, by applying loc. cit. to a mock Jacobi theta function $\phi \in \JJ^{\top,\a}_{1,m}$ we can conclude (after verifying consistency of some definitions) that $C\phi$ has rational Fourier coefficients for some $C\in\CC$, but the methods of loc. cit. do not imply that $C$ can be chosen so that these coefficients are rational integers.
Another difference is that Theorem \ref{thm:intro-rat} implies the existence of transcendental Fourier coefficients for $\f\in \JJ^{\opt,\a}_{1,m}$ whenever its shadow is not of theta type. %Such forms are not considered in loc. cit.
As mentioned in loc. cit., it is a conjecture of Bruinier--Ono \cite{MR2726107} that if the shadow of some mock Jacobi form of weight $1$ is not of theta type then almost all of the Fourier coefficients of $\f$ are transcendental. 
Taken together, Theorems \ref{thm:intro-firstmainthm} and \ref{thm:intro-rat} enable the construction of an infinite family of mock modular forms that are inherently transcendental (without the requirement of prior choices of cusp forms).
%1) $\JJ^{\opt,\a}_{1,m}$ is always non zero because of Rademacher. 
%2) There are only finitely many $\Gamma_0(m)+\ker(\a)$ with genus zero. 
%3) Apply Theorem 1.2.1 to a non zero $\f\in \JJ^{\opt,\a}_{1,m}$ in one of the infinitely many $\Gamma_0(m)+\ker(\a)$ that are not genus zero to obtain that $\f$ is not in $\JJ^{\top,\a}_{1,m}$. 
%4) Apply Theorem 1.2.2 to conclude that $\f$ is transcendental.

\subsection{Methods}\label{sec:intro:met}

A key step in the proof of Theorem \ref{thm:intro-firstmainthm} is the application of the Waldspurger type formula (\ref{eqn:jac:szlifts-Lfns}) obtained by Skoruppa (following earlier 
work by Gross--Kohnen--Zagier), which connects Fourier coefficients of skew-holomorphic Jacobi forms to central critical values of modular $L$-functions. Using this formula we
show in Lemma \ref{vanishing_period_maxsym} that the existence of a non-zero optimal mock Jacobi theta function $\f$ implies the vanishing of $L(f,1)$ for all weight $2$ newforms $f$ in a certain family determined by $\f$. This vanishing condition turns out to be so stringent that it cannot be satisfied unless the corresponding family is empty. 
This is the origin of the genus zero coincidence, 
for an application 
of Abel's theorem to Lemma \ref{vanishing_period_maxsym} shows
(cf. Proposition \ref{prop:class:proof-genuszero}) 
that if $\JJ^{\top,\a}_{1,m}\neq \{0\}$ then $\Gamma_0(m)+\ker(\a)$ has genus zero. 

An upper bound $\dim\JJ^{\top,\a}_{1,m}\leq 1$ follows from the result mentioned above (cf. Proposition \ref{prop_uniqueness1}), that a mock Jacobi form of weight 1 whose theta-coefficients are bounded near cusps must vanish. Then the proof of Theorem \ref{thm:intro-firstmainthm} is completed by explicit construction of sufficiently many non-zero examples. 
We achieve this (cf. \S\ref{sec:desc:mjt}) by applying some methods considered earlier in the context of umbral moonshine \cite{UM,MUM,umrec} and string theory \cite{Dabholkar:2012nd}.

All the mock modular forms of umbral moonshine are optimal mock Jacobi theta functions, so Theorem \ref{thm:intro-firstmainthm} answers Question \ref{que:intro:genuszero} by demonstrating that optimal mock Jacobi theta functions are classified by genus zero groups. Our proof of Theorem \ref{thm:intro-firstmainthm} shows that, in number theoretic terms,
it is Skoruppa's Waldspurger type formula, and the particular manifestation of the Shimura correspondence that underlies it, which is responsible for this genus zero property.

Our answer to Question \ref{que:intro:rational} begins with the observation 
that all the mock Jacobi theta functions we employ for the proof of Theorem \ref{thm:intro-firstmainthm} have integral Fourier coefficients. This proves 
the first part of Theorem \ref{thm:intro-rat}. 
For the second part, and our answer to the rationality question, we employ a result of Bruinier--Ono (cf. Theorem 5.5. in \cite{MR2726107}), which connects the algebraicity of Fourier coefficients of mock Jacobi forms to properties of twisted Heegner divisors. Applying Skoruppa's 
formula to an optimal mock Jacobi form whose shadow is not of theta type we obtain a modular $L$-function with non-vanishing central critical value. Then the Gross--Zagier formula \cite{MR833192}, and a result of Bump--Friedburg--Hoffstein \cite{MR1074487}, leads to the existence of an associated twisted Heegner divisor that fails the conditions of the Bruinier--Ono theorem.

In light of 
the empirical observation (cf. \S\ref{sec:mockthetafunctions}) that Ramanujan's mock theta functions are, essentially, the theta-coefficients of optimal mock Jacobi theta functions, Theorem 
\ref{thm:intro-rat} answers Question \ref{que:intro:rational} by demonstrating that optimality and the theta type condition are exactly the properties that underly the rationality of Ramanujan's examples. 
Comparing with Theorem \ref{thm:intro-firstmainthm} we see that, although he did not explain this to us, Ramanujan was 
accessing special geometric features of modular curves when writing his last letter to Hardy.

\subsection{Moonshine}\label{sec:intro:moon}

We now address Question \ref{que:intro:moonshine} and the connection to moonshine. For this it is convenient to have a compact notation for groups like $\Gamma_0(m)+\ker(\a)$. Given $K<O_m$ define $\Gamma_0(m)+K<\SL_2(\RR)$ by replacing $\ker(\a)$ with $K$ in (\ref{eqn:intro:Gamma0malpha}).
Following a tradition initiated in \cite{MR554399}, we use $m+n,n',\dots$ as a shorthand for 
$\Gamma_0(m)+K$ 
when $K=\{1,a(n),a(n'),\dots\}$. As before, the action of such a $\Gamma_0(m)+K$ on the upper half-plane naturally determines a compact Riemann surface. 
We say that $\Gamma_0(m)+K$ is {\em genus zero} if this surface is a sphere, and we say that $\Gamma_0(m)+K$ is {\em non-Fricke} if $-1\notin K$. 

The 39 non-Fricke genus zero groups of the form $\Gamma_0(m)+K$ with $K<O_m$ are given in Table \ref{tab:intro-pm}, described in terms of their symbols $\ell=m+n,n',\dots$. We call these symbols {\em lambencies} following \cite{UM,MUM}, and write $\gt{L}_1$ for the set that they comprise. 
In a second corollary to Theorem \ref{thm:intro-firstmainthm} (cf. Corollary \ref{cor:class:proof-genK}) 
we naturally attach an optimal mock Jacobi theta function 
$\f^{(\ell)}$
with integer Fourier coefficients to each lambency $\ell\in \gt{L}_1$. If $\ell$ 
is one of the $34$ {\em maximal} lambencies, corresponding to a group of the form $\Gamma_0(m)+\ker(\a)$, then
$\f^{(\ell)}=\sum_r H^{(\ell)}_r\th_{m,r}$
is the unique (cf. Corollary \ref{cor:class:proof-dimJtop}) optimal mock Jacobi theta function 
in $\JJ^{\top,\a}_{1,m}$ such that 
\begin{gather}\label{eqn:intro-Hell1growth}
H^{(\ell)}_1(\t)=-2q^{-\frac1{4m}}+O(1)
\end{gather}
as $\Im(\t)\to\infty$. If $\ell$ corresponds to one of the remaining $5$ non-Fricke genus zero groups then 
$\f^{(\ell)}=\sum_r H^{(\ell)}_r\th_{m,r}$ 
also satisfies (\ref{eqn:intro-Hell1growth}), 
while the behavior of the $H^{(\ell)}_r$ for $r\neq \pm 1\xmod 2m$ is controlled by $K$ (cf. (\ref{eqn:class:proof-JJ1mtopK})). As a result, $\f^{(\ell)}$ is related to the optimal mock Jacobi forms of maximal lambencies by averaging. 
For example, $\f^{(\ell)}=\frac12\left(\f^{(\ell_1)}+\f^{(\ell_2)}\right)$ for $\ell_1=6+2$, $\ell_2=6+3$ and $\ell=6$.

\begin{table}[ht]
\begin{small}
\begin{center}

\begin{tabular}{ccccccccccc}
\toprule
\multicolumn{1}{c|}{$X$}&$A_1^{24}$&$A_2^{12}$&$A_3^8$&$A_4^6$&$A_5^4D_4$&&$D_4^6$\\
	\cmidrule{1-8}
\multicolumn{1}{c|}{$\ll$}&	2&	3&	4&	5&	6&	6+2&	6+3\\
	\cmidrule{1-8}
\multicolumn{1}{c|}{$T^{(\ell)}$}&			$\frac{1^{24}}{2^{24}}$&	$\frac{1^{12}}{3^{12}}$&	$\frac{1^{8}}{4^{8}}$&	$\frac{1^{6}}{5^{6}}$&	$\frac{1^{5}3^1}{2^16^{5}}$&	${\frac{1^42^4}{3^46^4}}$&$ \frac{1^6 3^6}{2^6 6^6 }$\vspace{2pt}\\
\bottomrule
\ctoprule{1-8}
\multicolumn{1}{c|}{$X$}&$A_6^4$&$A_7^2D_5^2$&$A_8^3$&$A_{9}^2D_6$&&$D_6^4$&$A_{11}D_7E_6$\\
	\cmidrule{1-8}
\multicolumn{1}{c|}{$\ll$}&	7&	8& 9&	10&	10+2&	10+5&	12\\
	\cmidrule{1-8}
\multicolumn{1}{c|}{$T^{(\ell)}$}&$\frac{1^{4}}{7^{4}}$&$\frac{1^{4}4^2}{2^{2} 8^4}$&$\frac{1^{3}}{9^{3}}$&$ \frac{1^{3} 5^1} {2^1 10^{3}}$& $\frac{1^2 2^2}{5^2 10^2}$&$\frac{1^4 5^4}{2^4 10^4 }$&$\frac{1^{3} 4^1 6^2}{2^{2} 3^112^3 }$\vspace{2pt}\\
\bottomrule
\ctoprule{1-8}
\multicolumn{1}{c|}{$X$}&&$E_6^4$&$A_{12}^2$&$D_{8}^3$&&$A_{15}D_9$&$A_{17}E_7$\\
	\cmidrule{1-8}
\multicolumn{1}{c|}{$\ll$}& 12+3&	12+4&	13&	14+7&	15+5&	16&	18\\
	\cmidrule{1-8}
\multicolumn{1}{c|}{$T^{(\ell)}$}&$\frac{1^2 3^2}{4^2 12^2}$&	$\frac{1^{4} 4^4 6^4}{2^{4} 3^4 12^4}$&	$\frac{1^2}{13^{2}}$&	$\frac{1^3 7^3}{2^3 14^3 }$&	$\frac{1^2 5^2}{3^2 15^2}$&$\frac{1^{2}8^1}{2^116^{2}}$ & $\frac{1^{2} 6^1 9^1}{2^1 3^1 18^{2}}$\vspace{2pt}\\
\bottomrule
\ctoprule{1-8}
\multicolumn{1}{c|}{$X$}&&$D_{10}E_7^2$&&&$D_{12}^2$&&$A_{24}$\\
	\cmidrule{1-8}
\multicolumn{1}{c|}{$\ll$}	&18+2&	18+9 &	20+4&	21+3&	22+11&	24+8&	25	\\
	\cmidrule{1-8}
\multicolumn{1}{c|}{$T^{(\ell)}$}&$\frac{1^12^1}{9^118^1}$&$\frac{1^3 6^2 9^3}{2^3 3^2 18^3 }$&$\frac{1^24^210^2}{2^25^220^2}$&$\frac{1^13^1}{7^121^1}$ &$\frac{1^2 11^2}{2^2 22^2 }$&$\frac{1^26^18^212^1}{2^13^24^124^2}$&$\frac{1^1}{25^1}$\vspace{2pt}\\
\bottomrule
\ctoprule{1-8}
\multicolumn{1}{c|}{$X$}&&$D_{16}E_8$&&$E_8^3$\\
	\cmidrule{1-8}
\multicolumn{1}{c|}{$\ll$}	&28+7&	30+15 &	30+3,5,15&	30+6,10,15&	33+11&	36+4&	42+6,14,21	\\
	\cmidrule{1-8}
\multicolumn{1}{c|}{$T^{(\ell)}$}&$\frac{1^17^1}{4^128^1}$&$\frac{1^2 6^1 10^1  15^2} {2^2 3^1 5^1  30^2 }$&$\frac{1^13^15^115^1}{2^16^110^130^1}$&$\frac{1^{3} 6^3 10^3 15^3}{2^{3} 3^3 5^3 30^3}$&$\frac{1^111^1}{3^133^1}$&$\frac{1^14^118^1}{2^19^136^1}$ &$\frac{1^26^214^221^2}{2^23^27^242^2}$\vspace{2pt}\\
\bottomrule
\ctoprule{1-5}
\multicolumn{1}{c|}{$X$}&$D_{24}$&\\
	\cmidrule{1-5}
\multicolumn{1}{c|}{$\ll$}	&46+23&	60+12,15,20&	70+10,14,35&	78+6,26,39	
\\
	\cmidrule{1-5}
\multicolumn{1}{c|}{$T^{(\ell)}$}&$\frac{1^1 23^1} {2^1 46^1 }$&$\frac{1^112^115^120^1}{3^14^15^160^1}$&$\frac{1^110^114^135^1}{2^15^17^170^1}$&$\frac{1^16^126^139^1}{2^13^113^178^1}$\vspace{2pt}\\
\cbottomrule{1-5}
\end{tabular}
\caption{Lambencies, root systems and principal moduli.}\label{tab:intro-pm}
\end{center}
\end{small}
\vspace{-20pt}
\end{table}

A key significance of this construction 
is that 
it recovers all the mock modular forms (for the full modular group) appearing in umbral moonshine \cite{UM,MUM}. Precisely, if $H^X=(H^X_r)$ is the vector-valued mock modular form attached to one of the Niemeier root systems $X$ in \cite{MUM}, then $H^X_r=H^{(\ell)}_r$ for all $r$, for some lambency 
$\ell\in\gt{L}_1$. 
The 23 Niemeier root systems and their corresponding lambencies are also given in Table \ref{tab:intro-pm}. 

Write $\gt{L}_1^+$ for the set of 23 lambencies that have root systems attached in Table \ref{tab:intro-pm}. 
Notably, $\gt{L}_1^+$ does not exhaust $\gt{L}_1$. 
In \S\ref{sec:class:rpm} we verify (cf. Proposition \ref{prop:class:rpm-posphi}) that a lambency $\ell\in\gt{L}_1$ 
belongs to $\gt{L}_1^+$ 
if and only if the coefficient of $q^{-\frac{D}{4m}}$ in $H^{(\ell)}_r$ is positive when $D$ is negative, for $0<r<m$.
This positivity condition is natural from the point of view of moonshine, where the coefficients of the $H^{(\ell)}_r$, for $\ell\in \gt{L}_1^+$, are interpreted as dimensions of representations of a certain finite group $G^{(\ell)}$. Is there any such interpretation for the lambencies not in $\gt{L}_1^+$? 

The first few Fourier coefficients of the $H^{(\ell)}_r$ for $\ell$ not in $\gt{L}_1^+$ are displayed in 
Tables \ref{tab:desc:mjt-6+2}-\ref{tab:desc:mjt-78+6,26,39}, in \S\ref{sec:desc:mjt}. 
Although we do not pursue this in detail here, inspection of these tables indicates that there is regularity in the failure of the corresponding lambencies 
to satisfy the positivity condition formulated above, and suggests that an extension of the umbral moonshine conjectures 
to all the lambencies in $\gt{L}_1$ probably exists. 

Viewed from this perspective, 
the answer to Question \ref{que:intro:moonshine} is: probably yes. 
Since essentially all of Ramanujan's mock theta functions are visible in the $H^{(\ell)}_r$ for $\ell\in \gt{L}_1$ according to \S\ref{sec:mockthetafunctions}, 
essentially all of Ramanujan's mock theta functions appear in moonshine, subject to a suitable extension of the umbral moonshine conjectures to the 16 new lambencies appearing here. 
We save further consideration of this extension problem for future work.

\subsection{Singular Moduli}\label{sec:intro:sing}

As mentioned in \S\ref{sec:intro:met}, the proof of Theorem \ref{thm:intro-firstmainthm} involves the identification of a non-zero mock Jacobi form in $\JJ^{\top,\a}_{1,m}$ for each $\a\in \widehat{O}_m$ such that $\Gamma_0(m)+\ker(\a)$ is non-Fricke and genus zero. While there are enough ad hoc methods available to obtain all the forms we need (cf. \ref{sec:desc:mjt}), it is natural to ask for a more uniform construction. 

The result (cf. Corollary \ref{cor:class:rpm-theta}) that any optimal mock Jacobi form that is not a mock Jacobi theta function has transcendental coefficients 
exemplifies the fact that mock modular forms are, generally speaking, 
difficult to construct.
Nonetheless, 
using the generalized Borcherds product construction introduced by Bruinier--Ono in \cite{MR2726107}, we are able to show (cf. Theorem \ref{thm:hauptmodul_Psi}) that if $\ell\in \gt{L}_1$ then the Fourier coefficients of 
$\f^{(\ell)}$ can be expressed explicitly in terms of singular moduli for the corresponding genus zero group. 
As we will explain momentarily, these singular moduli are directly computable. In a sense, Theorem \ref{thm:hauptmodul_Psi} enables us to reconstruct the mock Jacobi form $\f^{(\ell)}$ from a suitably chosen eta product.

To explain this, choose a lambency $\ell\in \gt{L}_1$ and let $\Gamma_0(m)+K$ be the corresponding genus zero group. The aforementioned 
eta product $T^{(\ell)}$ for $\ell$ is given in Table \ref{tab:intro-pm}, where we use a formal product $n_1^{d_1}\cdots n_l^{d_l}$ as shorthand for $\eta(n_1\tau)^{d_1}\cdots \eta(n_l\tau)^{d_l}$. 
One significance of $T^{(\ell)}$ is that it is a {\em principal modulus} (a.k.a. {\em Hauptmodul}) for $\Gamma_0(m)+K$ (cf. Lemma \ref{lem:class:pm-chrTell}), meaning that $T^{(\ell)}$ is a 
$\Gamma_0(m)+K$-invariant holomorphic function on the upper half-plane that
descends to generator for the function field of 
the corresponding Riemann surface.

A second significance of $T^{(\ell)}$ is that it directly recovers the shadow of $\f^{(\ell)}$. Explicitly, 
if $\s^{(\ell)}=\sum_r\overline{g_r}\th_{m,r}$ is the shadow of $\f^{(\ell)}$, and $C_{\s^{(\ell)}}(D,r)$ is the coefficient of $q^{\frac{D}{4m}}$ in $g_r$, then from Lemma \ref{lem:class:pm-sigmaellprops} and Proposition \ref{prop:class:pm-xiphisig} it follows that $\sum_{d|n}C_{\s^{(\ell)}}\left(\frac{n^2}{d^2},\frac{n}{d}\right)$ is the coefficient of $q^n$ in the logarithmic derivative of $T^{(\ell)}$.

As we have claimed, a more direct relationship holds between $T^{(\ell)}$ and $\f^{(\ell)}$. To describe this write $C^{(\ell)}(D,r)$ for the coefficient of $q^{-\frac{D}{4m}}$ in $H^{(\ell)}_r$, and say that $D\in\ZZ$ is a {\em fundamental discriminant} if it is the discriminant of a number field of degree at most $2$.
Then for $D$ a negative fundamental discriminant and $r\xmod 2m$ such that $D=r^2\xmod 4m$ we consider the formal product 
\begin{gather}\label{eqn:intro:sing-Psiell}
	\Psi^{(\ell)}_{D,r}(\tau):=\prod_{n>0}\prod_{b\xmod D}
	\left(
	1-e^{ \frac{2\pi  i b}{D}}
	q^n\right)^{\left(\frac D b\right)C^{(\ell)}(Dn^2,rn)},
\end{gather}
where 
$\left(\frac D b\right)$ is the Kronecker symbol (cf. \S\ref{sec:notn}). 

Applying Proposition 5.2 and Theorem 5.3 of \cite{MR2726107} we obtain that the product (\ref{eqn:intro:sing-Psiell}) converges for $\Im(\t)$ sufficiently large, admits an analytic continuation to a meromorphic function on the upper half-plane, and transforms with a character under the action of $\Gamma_0(m)$. Assuming that $D\neq-3$ when $m\in\{7,13,21\}$ we verify that the 
action extends to $\Gamma_0(m)+K$, and that the 
extended character is trivial. Thus we obtain a meromorphic function on the compact Riemann surface defined by $\Gamma_0(m)+K$. The results of \cite{MR2726107} also lead to an explicit expression (\ref{eqn:class:pm-ZellDr}) for the divisor of this function in terms of roots of quadratic forms (i.e. Heegner points), and from this we obtain the result of Theorem \ref{thm:hauptmodul_Psi}: that $\Psi^{(\ell)}_{D,r}$, with the aforementioned restrictions on $D$, is an explicitly computable rational function in $T^{(\ell)}$.

Theorem \ref{thm:hauptmodul_Psi} is a generalization of the main result of \cite{MR3357517}, which covers the cases that $\ell\in\{2,3,4,5,7,9,13,25\}$. Examples of expressions for the $\Psi^{(\ell)}_{D,r}$ as rational functions in $T^{(\ell)}$, and expressions for coefficients of the $\f^{(\ell)}$ as weighted sums of singular moduli, can be found therein. 

\subsection{Overview}\label{sec:intro:struc}

We now describe the structure of the remainder of the article. In \S\ref{sec:notn} we present a guide to the most commonly used notation. 
The discussion begins 
in \S\ref{sec:jac} with recall of the main facts from the theory of Jacobi forms that we utilize. We discuss holomorphic and skew-holomorphic Jacobi forms in \S\ref{sec:jac:hol}, mock Jacobi forms in \S\ref{sec:Mock Jacobi Forms}, certain theta lifts for Jacobi forms in \S\ref{sec:jac:szlifts}, and Eichler--Zagier operators on Jacobi forms in \S\ref{sec:jac:ez}.

The proofs of our results are presented in \S\ref{sec:class}. After discussing optimality and the notion of mock Jacobi theta function in \S\ref{sec:class:defn}, we prove certain technical results on theta type skew-holomorphic Jacobi forms in \S\ref{sec:class:theta}. These results are used in \S\ref{sec:class:proof} to prove our first main result, Theorem \ref{thm:intro-firstmainthm}. Then, we prove our second main result, Theorem \ref{thm:intro-rat}, in \S\ref{sec:class:rat}. We develop the relationship between optimal mock Jacobi theta functions and genus zero groups in \S\ref{sec:class:pm}, and discuss the connection to umbral moonshine, including positivity conditions, in \S\ref{sec:class:rpm}.

The article concludes with an appendix. In \S\ref{sec:desc:mjt} we describe how to construct the optimal mock Jacobi theta functions that have not already appeared in umbral moonshine, and tabulate their low order Fourier coefficients. In \S\ref{sec:mockthetafunctions} we give explicit expressions for Ramanujan's mock theta functions, and those found later by Andrews and Gordon--McIntosh, in terms of theta-coefficients of the optimal mock Jacobi theta functions that are classified in this work.

\break

%-------------------------------------------------------------------------------%
\section{Notation Guide}\label{sec:notn}
%-------------------------------------------------------------------------------%

\begin{footnotesize}

\begin{list}{}{
	\itemsep -1pt
	\labelwidth 23ex
	\leftmargin 13ex	
	}

\item
[$\sqrt{\cdot}$]
The principal branch of the square root function, satisfying $\sqrt{e^{i\theta}}=e^{i\theta/2}$ when $-\pi<\theta\leq\pi$.

\item
[$\XBox(\,\cdot\,)$]
For $n$ a non-zero integer, set $\XBox(n):=n/f^2$ where $f$ is the largest integer such that $f^2|n$.

\item
[$\left(\frac{\,\cdot\,}{\,\cdot\,}\right)$]
The Kronecker symbol. Let $D$ be a non-zero integer congruent to $0$ or $1$ modulo $4$. Then $\left(\frac{D}{\cdot}\right)$ 
is totally multiplicative, satisfies $\left(\frac{D}{-1}\right)=\sgn(D)$, and vanishes on primes dividing $D$. If $p$ is prime and $p\not |D$ then $\left(\frac{D}{p}\right)= 1$ if $D$ is a square modulo $4p$, and $\left(\frac{D}{p}\right)=-1$ otherwise.

\item
[$\langle\cdot\,,\cdot\rangle$]
The Petersson inner product on Jacobi forms or modular forms. See (\ref{eqn:jac:hol-PetIP}).

\item
[$\langle\cdot\,,\cdot\rangle^\reg$]
The regularized Petersson inner product. See (\ref{eqn:jac:szlifts-Fs}).

\item
[$\{\cdot\,,\cdot\}$]
The Bruinier--Funke pairing on $\JJ^\wk_{k,m}\times J^\sk_{3-k,m}$. See (\ref{eqn:jac:mock-BFpairing}).

\item
[$(n,n',\dots)$]
The greatest common divisor of integers $n,n',\dots$.

\item
[$n*n'$]
We set $n*n':=nn'/(n,n')^2$ for integers $n,n'$, not both zero.

\item
[$m+n,n',\dots$]
A shorthand for $\Gamma_0(m)+K$ where $K=\{1,a(n),a(n'),\dots\}$.

\item
[$|_{k}$]
The weight $k$ action of $\mpt(\RR)$, for $k\in\frac12\ZZ$. Given a holomorphic function $f:\HH\to\CC$ and $(\gamma,\upsilon)\in\mpt(\RR)$, define $(f|_{k}(\gamma,\upsilon))(\tau):=f(\gamma\tau)\upsilon(\tau)^{-2k}$. If $k\in \ZZ$ the action factors through $\SL_2(\RR)$ and we simply write $f|_k\gamma$.

\item
[$a(n)$]
For $n\in \Ex_m$ write $a(n)$ for the unique $a\in O_m$ such that $a=-1\xmod 2n$ and $a=1\xmod 2m/n$. The assignment $n\mapsto a(n)$ defines an isomorphism of groups $\Ex_m\xrightarrow{\sim} O_m$.

\item[$\a_Q$]
A Heegner point. The image under the natural map $\HH\to X_0(m)$ of the unique solution to $Q(x,1)$ that lies in $\HH$, for $Q\in \mathcal{Q}(m,D,r)$ with $D<0$.

\item
[$c_f(n)$]
The coefficient of $q^n$ in the Fourier expansion of $f$, for $f$ in $M_k(\Gamma)$.

\item
[$C^{(\ell)}(D,r)$]
A shorthand for $C_{\f^{(\ell)}}(D,r)$.

\item
[$C_\f(D,r)$]
A Fourier coefficient of a 
Jacobi form $\f$. See (\ref{eqn:jac:hol-DFoucffhol}), (\ref{eqn:jac:hol-DFoucffskw}) and (\ref{eqn:jac:mock-Fou}).

\item
[$\chi_D(\,\cdot\,)$]
The generalized genus character, as defined in Proposition 1 of \cite{MR909238}. Let $m$ be a positive integer, let $D$ be a fundamental discriminant (i.e. $1$ or the discriminant of some quadratic number field) that is a square modulo $4m$, and let $D_1$ be a multiple of $D$ such that $D_1/D$ is also a square modulo $4m$. Then for $Q\in \mathcal{Q}(m,D_1,r_1)$ with $D_1=r_1^2\xmod 4m$ and $Q(x,y)=Ax^2+Bxy+Cy^2$, define $\chi_{D}(Q)=0$ unless $(A/m,B,C,D)=1$. If this condition holds then set $\chi_{D}(Q)=\left(\frac{D}{d}\right)$ where $d$ is any integer coprime to $D$ that is represented by a quadratic form $(A/n)x^2+Bxy+Cny^2$ for some $n|m$.

\item
[${\cal D}^{k-1}$]
The linear operator on smooth functions $\th:\HH\times\CC\to \CC$
defined by setting ${\cal D}^{k-1}\th(\tau):=(2\pi i)^{1-k}\partial_z^{k-1}\th(\tau,z)|_{z=0}$.
This operator maps $\Th_m$ to $M_{k-\frac12}(\Gamma(4m))$ when $k\in \{1,2\}$. 

\item
[$\Dp(\f)$] The positive discriminant part of a 
Jacobi form $\f$. See (\ref{eqn:jac:hol-polarpart}) and (\ref{eqn:jac:mock-polarpart}).

\item
[$\ex(\,\cdot\,)$]
We set $\ex(x):=e^{2\pi i x}$.

\item
[$\E_m$]
The space of elliptic forms of index $m$. See \S\ref{sec:jac:hol}.

\item
[$\E_m^\a$]
The $\a$-eigenspace for the action of $O_m$ on $\E_m$. Given $\a\in \widehat{O}_m$ define $\E^\a_m$ to be the span of the $\f\in \E_m$ such that $\f\cdot a=\a(a)\f$ for all $a\in O_m$. 
Given $S<\E_m$ define $S^\a:=S\cap\E_m^\a$.

\item
[$\Ex_m$]
The group of exact divisors of $m$. That is, the positive integers $n$ such that 
$n|m$ and $(n,m/n)=1$, with group operation $n*n'$. 

\item
[$\eta(\tau)$]
The Dedekind eta function, $\eta(\t):=q^{\frac1{24}}\prod_{n>0}(1-q^n)$ for $\t\in\HH$.

\item
[$f\otimes D$]
The {$D$-th twist} of a cusp form $f\in S_k(\Gamma_0(m))$, defined by setting $c_{f\otimes D}(n):=c_f(n)\left(\frac{D}{n}\right)$.

\item
[$\mathcal{F}$]
We set 
$\mathcal{F}:=\{\tau\in \HH\mid |\Re(\tau)|\leq \frac 12,\,|\tau|\geq 1\}$.

\item
[$\mathcal{F}_t$]
We set $\mathcal{F}_t:=\{\tau\in \mathcal{F}\mid \Im(\t)\leq t\}$.

\item
[$g_r$, $g$]
The complex conjugates of the theta-coefficients of a skew-holomorphic Jacobi form $\phi\in J^\sk_{k,m}$, satisfying $\phi=\sum_{r\xmod 2m}\overline{g_r}\theta_{m,r}$ and $g=(g_r)_{r\xmod 2m}$. See \S\ref{sec:jac:hol}. 
We regard $g(\tau)=(g_r(\tau))$ as a column vector, so that $\phi=\overline{g}^t\theta_m$.

\item
[$g_r^*$, $g^*$]
Non-holomoprhic Eichler integrals of $g_r$, $g$. See (\ref{def:Eichler_integral}).

\item
[$\Gamma(N)$]
The kernel of the natural map $\SL_2(\ZZ)\to \SL_2(\ZZ/N\ZZ)$.

\item
[$\Gamma_0(m)$]
The preimage of upper triangular matrices under the natural map $\SL_2(\ZZ)\to \SL_2(\ZZ/m\ZZ)$.

\item
[$\Gamma_0(m)+K$] 
Given $K<O_m$ let $\Gamma_0(m)+K$ be the group 
generated by $\Gamma_0(m)$ and the 
$W_n$ for $a(n)\in K$.

\item[$\Gamma_0(m)_Q$]
The stabilizer in $\Gamma_0(m)$ of some $Q\in \mathcal{Q}(m,D,r)$.

\item
[$\Gamma_1(N)$]
The preimage of the matrices $\left(\begin{smallmatrix}1&*\\0&1\end{smallmatrix}\right)$ under the natural map $\SL_2(\ZZ)\to \SL_2(\ZZ/N\ZZ)$.

\item
[$h_r$, $h$] 
The theta-coefficients of an elliptic form $\phi\in \E_m$, satisfying $\phi=
\sum_{r\xmod 2m}h_r\theta_{m,r}$ and $h=(h_r)_{r\xmod 2m}$.
See (\ref{eqn:jac:thtdec}). 
Regard 
$h(\tau)$ as a column vector, so that $\phi=h^t\theta_m$.

\item
[$\hat{h}_r$, $\hat{h}$]
Completions of $h_r$ and $h=(h_r)$. See (\ref{eqn:jac:mck-hhat}).

\item
[$H^{(\ell)}_r$, $H^{(\ell)}$]
The theta-coefficients of $\f^{(\ell)}$. See \S\ref{sec:desc}.

\item
[$\HH$]
The upper half-plane, composed of the $\tau\in \CC$ such that $\Im(\tau)>0$.

\item
[$\iota_{m',t}^m$]
The degeneracy map $S_2(\Gamma_0(m'))\to S_2(\Gamma_0(m))$ defined by setting $(\iota_{m',t}^mf)(\tau):=f(t\tau)$, for $m'$ a divisor of $m$, and $t$ a divisor of $\tfrac{m}{m'}$.

\item
[$j(\gamma,\tau)$]
Given $\gamma\in\Gamma_1(4)$ define $j(\gamma,\tau):=\theta_{1,0}^0(\gamma\tau)/\theta_{1,0}^0(\tau)$. 
Then $j(\gamma,\tau)^2=(c\tau+d)$ when 
$\gamma=\left(\begin{smallmatrix}a&b\\c&d\end{smallmatrix}\right)$ (cf. e.g. \cite{MR0332663}). The assignment $\gamma\mapsto (\gamma,j(\gamma,\tau))$ defines an embedding of groups $\Gamma_1(4)\to\mpt(\ZZ)$, and we use this to identify any $\Gamma<\Gamma_1(4)$ as a subgroup of $\mpt(\ZZ)$.

\item
[$J_0(m)$]
The Jacobian of $X_0(m)$.

\item
[$J_{k,m}$]
{The space of holomorphic Jacobi forms of weight $k$ and index $m$.} See \S\ref{sec:jac:hol}.

\item
[$J^\sk_{k,m}$]
{The space of skew-holomorphic Jacobi forms of weight $k$ and index $m$.} See \S\ref{sec:jac:hol}.

\item
[$J^\wk_{k,m}$]
{The space of weak holomorphic Jacobi forms of weight $k$ and index $m$.} See \S\ref{sec:jac:hol}.

\item
[$\JJ_{k,m}^\opt$]
The space of optimal mock Jacobi forms of weight $k$ and index $m$. See (\ref{eqn:class:defn-JJopt}). 

\item
[$\JJ_{k,m}^{\opt,\a}$]
For $\a\in \widehat{O}_m$ we set $\JJ_{k,m}^{\opt,\a}:=\JJ_{k,m}^\opt\cap\E_m^\a$.

\item
[$\JJ_{k,m}^\top$]
The space of optimal mock Jacobi theta functions of weight $k$ and index $m$. See (\ref{eqn:class:defn-JJtop}). 

\item
[$\JJ_{k,m}^{\top,\a}$]
For $\a\in \widehat{O}_m$ we set $\JJ_{k,m}^{\top,\a}:=\JJ_{k,m}^\top\cap\E_m^\a$.

\item
[$\JJ^{\top|K}_{1,m}$]
A certain subspace of $\JJ^\top_{k,m}$ determined by a choice of $K<O_m$. See (\ref{eqn:class:proof-JJ1mtopK}).

\item
[$\JJ_{k,m}^\wk$]
The space of weak mock Jacobi forms of weight $1$ and index $m$. See \S\ref{sec:Mock Jacobi Forms}.

\item
[$\JJ_{k,m}^{\wk,\a}$]
For $\a\in \widehat{O}_m$ we set $\JJ_{k,m}^{\wk,\a}:=\JJ_{k,m}^\wk\cap\E_m^\a$.

\item
[$\ell$]
A lambency. That is, a symbol $\ell=m+n,n',\dots$ belonging to $\gt{L}_1$. See \S\ref{sec:class:proof}.

\item
[$L(f,s)$]
The $L$-function naturally attached to a cusp form $f\in S_k(\Gamma_0(m))$, defined by analytic continuation of the Dirichlet series $\sum_{n>0}c_f(n)n^{-s}$. 

\item
[$\gt{L}_1$]
The set of symbols $\ell=m+n,n',\dots$ such that $\JJ^{\top|K}_{1,m}$ is non-vanishing, where $K=\{1,a(n),a(n'),\dots\}$. See \S\ref{sec:class:proof}. 

\item
[$\gt{L}_1^+$]
The subset of $\gt{L}_1$ consisting of lambencies that have Niemeier root systems attached to them in Table \ref{tab:intro-pm}. See \S\ref{sec:class:proof}, and Propositions \ref{prop:class:rpm-possig} and \ref{prop:class:rpm-posphi}.

\item
[$M_k(\Gamma)$]
The space of modular forms of weight $k$ for $\Gamma$, for $k\in\frac12\ZZ$ and $\Gamma$ a discrete subgroup of $\SL_2(\RR)$. 
For $k\in \ZZ$ assume $\Gamma$ is commensurable with $\SL_2(\ZZ)$ and define $M_k(\Gamma)$ to be the space of holomorphic functions $f:\HH\to\CC$ such that $f|_{k}\gamma=f$ for all $\gamma\in \Gamma$, and $f|_{k}\gamma=O(1)$ as $\Im(\tau)\to\infty$ for all $\gamma\in\SL_2(\ZZ)$. 
For $k\in \ZZ+\frac12$ assume $\Gamma<\Gamma_1(4)$ and define $M_k(\Gamma)$ by requiring $f|_{k}(\gamma,j(\gamma,\tau))=f$ for all $\gamma\in \Gamma$, and $f|_{k}(\gamma,\upsilon)=O(1)$ as $\Im(\tau)\to\infty$ for all $(\gamma,\upsilon)\in\mpt(\ZZ)$.

\item
[$M_k(\Gamma_0(m))^\a$]
The space of $f\in M_k(\Gamma_0(m))$ such that $f|W_n= \a(a(n))f$ for $n\in\Ex_m$. Cf. (\ref{eqn:jac:szlifts-SZalpha}).

\item
[$O_m$]
The automorphisms of $\ZZ/2m\ZZ$ that preserve the $\QQ/\ZZ$-valued quadratic form $x\mapsto \frac{x^2}{4m}$. 

\item
[$\widehat{O}_m$]
The Pontryagin dual of $O_m$, defined by
$\widehat{O}_m:=\hom(O_m,\CC^\times)$.

\item
[$\Omega_m(n)$]
The Omega matrix associated to $n|m$. See (\ref{def:OmegaMatrices}).

\item
[$P^\a$]
A projection operator on modular forms or weak mock Jacobi forms. See (\ref{eqn:class:proof-Palpha}) and (\ref{eqn:class:proof-Palphaphi}).

\item
[$P^\sk_{k,m}$]
{The subspace of $S^\sk_{k,m}$ spanned by Hecke eigenforms that do not belong to $T^\sk_{k,m}$. See \S\ref{sec:jac:szlifts}.}

\item
[$\f^{(\ell)}$]
For $\ell=m+n,n',\ldots \in \gt{L}_1$, the unique element of $\JJ_{1,m}^{\top|K}$ such that $C_{\f^{(\ell)}}(1,1)=-2$, where $K=\{1,a(n),a(n'),\dots\}$.

\item
[$\Psi^{(\ell)}_{D,r}$]
A Borcherds product attached to $\ell=m+n,n',\ldots\in \gt{L}_1$, for $D$ a negative fundamental discriminant and $r\xmod 2m$ such that $D=r^2\xmod 4m$. See (\ref{eqn:class:pm-Psiell}). 

\item
[$q$]
We set $q:=\ex(\tau)$ for $\tau\in \HH$.

\item[$\mathcal{Q}(m,D,r)$]
The set of integral binary quadratic forms $Q(x,y)=Ax^2+Bxy+Cy^2$ such that $A=0\xmod m$ and $D=B^2-4AC$, and $B=r\xmod 2m$. The group $\Gamma_0(m)$ acts naturally on $\mathcal{Q}(m,D,r)$, according to $\left(Q|\left(\begin{smallmatrix} a&b\\c&d\end{smallmatrix}\right)\right)(x,y):=Q(ax+by,cx+dy)$. 

\item[$\widehat{Q}(\tau)$]
Given $Q(x,y)=Ax^2+Bxy+Cy^2$ set $\widehat{Q}(\tau):=A|\tau|^2+B\Re(\tau)+C$. 

\item
[$\varrho_m$]
The unitary representation $\varrho_m:\mpt(\ZZ)\to\GL_{2m}(\CC)$ generated by the assignments $\varrho_m(\widetilde{S})={\cal S}$ and $\varrho_m(\widetilde{T})={\cal T}$. (This is a Weil representation. Cf. e.g. \S3 of \cite{MR2512363}.)

\item
[$S_k(\Gamma)$]
The space of cuspidal modular forms of weight $k$ for $\Gamma$, for $k\in\frac12\ZZ$ and $\Gamma<\SL_2(\RR)$. For $k\in \ZZ$ these are the $f\in M_k(\Gamma)$ such that $f|_k\gamma=O(q)$ as $\Im(\tau)\to \infty$, for all $\gamma\in \SL_2(\ZZ)$. For $k\in \ZZ+\frac12$ replace $\gamma\in \SL_2(\ZZ)$ with $(\gamma,\upsilon)\in \mpt(\ZZ)$.

\item
[$S_{k,m}$]
{The space of cuspidal holomorphic Jacobi forms of weight $k$ and index $m$.} See \S\ref{sec:jac:hol}

\item
[$S^\sk_{k,m}$]
{The space of cuspidal skew-holomorphic Jacobi forms of weight $k$ and index $m$.} See \S\ref{sec:jac:hol}.

\item
[${\cal S}$, ${\cal T}$]
The unitary matrices ${\cal S}=({\cal S}_{rr'})_{r,r'\xmod 2m}$ and ${\cal T}=({\cal T}_{rr'})_{r,r'\xmod 2m}$, for a fixed positive integer $m$, given by ${\cal S}_{rr'}:=\frac1{\sqrt{2m}}\ex\left(-\frac{1}{8}-\frac{rr'}{2m}\right)$ and ${\cal T}_{rr'}:=\ex\left(\frac{r^2}{4m}\right)\delta_{r,r'}$.

\item
[$\mathcal{S}_{D,r}$]
A theta lift for holomorphic or skew-holomorphic Jacobi forms.
See (\ref{eqn:jac:szlifts-SDr}).

\item
[$\mathcal{S}^\reg_{D,r}$]
A regularised theta lift for weak mock Jacobi forms. 
See (\ref{eqn:jac:szlifts-SDrreg}).

\item
[$\mpt(\RR)$]
The metaplectic double cover of $\SL_2(\RR)$, whose elements are the pairs $(\gamma,\upsilon)$, where $\gamma=\left(\begin{smallmatrix}a&b\\c&d\end{smallmatrix}\right)\in\SL_2(\RR)$, and $\upsilon:\HH\to\CC$ is a holomorphic function satisfying $\upsilon(\tau)^2=(c\tau+d)$. The multiplication is $(\gamma,\upsilon)(\gamma',\upsilon')=(\gamma\gamma',(\upsilon\circ\gamma')\upsilon')$.

\item
[$\mpt(\ZZ)$]
The preimage of $\SL_2(\ZZ)$ under the natural projection $\mpt(\RR)\to \SL_2(\RR)$. The elements $\widetilde{T}:=\left(\left(\begin{smallmatrix}1&1\\0&1\end{smallmatrix}\right),1\right)$ and $\widetilde{S}:=\left(\left(\begin{smallmatrix}0&-1\\1&0\end{smallmatrix}\right),\sqrt{\tau}\right)$ comprise a generating set.

\item
[$\SZ$]
The Skoruppa--Zagier map $S_{k,m}\oplus S^\sk_{k,m}\to M_{2k-2}(\Gamma_0(m))$. 
See \S\ref{sec:jac:szlifts}.

\item
[$\s^{(\ell)}$]
A theta type skew-holomorphic Jacobi form naturally attached to $T^{(\ell)}$, for $\ell\in \gt{L}_1$. 
See (\ref{eqn:class:pm-sigmaell}), Lemma \ref{lem:class:pm-sigmaellprops} and Proposition \ref{prop:class:pm-xiphisig}.

\item
[$t_{k,m}$]
Set $t_{k,m}(\tau,z):=\sum_{r\xmod 2m} \overline{\theta^{k-1}_{m,r}(\tau)}\theta_{m,r}(\tau,z)$ for $k\in \{1,2\}$. We have $t_{k,m}\in T^\sk_{k,m}$.

\item
[$T^{(\ell)}$]
For $\ell=m+n,n',\ldots\in \gt{L}_1$, a certain principal modulus for the group $\Gamma_0(m)+K$, where $K=\{1,a(n),a(n'),\dots\}$. See Table \ref{tab:intro-pm} and 
Lemma \ref{lem:class:pm-chrTell}.

\item
[$T_n$]
Hecke operators, for modular forms and Jacobi forms. See (\ref{eqn:jac:szlifts-Heckemod}) and (\ref{eqn:jac:hol-phiTn}).

\item
[$T^\prime_{k-\frac12}(\Gamma)$]
A certain subspace of $M_{k-\frac12}(\Gamma)$, for $\Gamma=\Gamma_1(4m)$ or $\Gamma=\Gamma(4m)$. See (\ref{eqn:class:theta-TGamma14m}) and (\ref{eqn:class:theta-TGamma4m}).

\item
[$T^\sk_{k,m}$]
The space of theta type skew-holomorphic Jacobi forms of weight $k$ and index $m$. See (\ref{eqn:jac:hol-Tskkm}).

\item
[$T^{\prime\sk}_{k,m}$]
A certain subspace of $J^\sk_{k,m}$. See (\ref{eqn:class:theta-Tprimeskkm}).

\item
[$\theta_{m,r}$, $\theta_m$]
The theta series $\theta_{m,r}(\tau,z):=\sum_{\ell=r\xmod 2m}q^{\ll^2/4m}y^\ell$, for $\tau\in \HH$ and $z\in \CC$. We regard $\theta_m(\tau,z):=(\theta_{m,r}(\tau,z))_{r\xmod 2m}$ as a column vector.

\item
[$\theta^{k-1}_{m,r}$]
The thetanullwert $\theta_{m,r}^{k-1}:=
\sum_{\ell=r\xmod 2m}\ell^{k-1}q^{\ll^2/4m}$. An element of $M_{k-\frac12}(\Gamma(4m))$ when $k=1$ or $k=2$, cuspidal if $k=2$.

\item
[$\Theta_m$]
The $2m$-dimensional $\mpt(\ZZ)$-module spanned by the theta series $\theta_{m,r}$. See \S\ref{sec:jac:hol}.

\item
[$\Theta_m^{\new,\a}$]
A certain sub $\mpt(\ZZ)$-module of $\Theta_m$. See Theorem \ref{thm:jac:EZ-Thmdec}.

\item
[$\Theta_{D,r}(\cdot\,,\cdot\,,\cdot)$]
A kernel function for $\mathcal{S}_{D,r}$ or $\mathcal{S}_{D,r}^\reg$. See (\ref{eqn:jac:szlifts-ThetaC}), (\ref{eqn:jac:szlifts-ThetaCreg}) and (\ref{eqn:class:proof-Theta}).

\item
[$U_d$, $V_{\ell}$]
Hecke-like operators for Jacobi forms. See (\ref{eqn:jac:szlifts-Ud}) and (\ref{eqn:jac:szlifts-Vell}).

\item
[$W_n$]
An Atkin--Lehner involution. 
For $n\in \Ex_m$, the coset of $\Gamma_0(m)$ in $\SL_2(\RR)$ composed of the matrices $\frac1{\sqrt{n}}\left(\begin{smallmatrix}an&b\\cm&dn\end{smallmatrix}\right)$ such that $a,b,c,d\in\ZZ$ and $adn-bcm/n=1$. Given $f\in M_k(\Gamma_0(m))$ define $f|W_n:=f|_kw_n$ for any $w_n\in W_n$. We have $W_nW_{n'}=W_{n*n'}$ in $\SL_2(\RR)$.

\item
[$\W_m(n)$] 
An Eichler--Zagier operator on $\E_m$, for $n|m$. See (\ref{eqn:jac:ez-phiWmn}).

\item[$X_\Gamma$]
The Riemann surface $X_\Gamma:=\Gamma\backslash\HH\cup\QQ\cup\{\infty\}$, for $\Gamma<\SL_2(\RR)$ commensurable with $\SL_2(\ZZ)$. 

\item[$X_0(m)$]
Shorthand for $X_\Gamma$ when $\Gamma=\Gamma_0(m)$.

\item
[$\xi$]
The shadow map $\JJ_{k,m}^\wk\to S^\sk_{3-k,m}$. See (\ref{eqn:jac:mock-xi}).

\item
[$y$]
We set $y:=\ex(z)$ for $z\in\CC$.

\item
[$Z_{D,r}(\,\cdots)$]
A twisted Heegner divisor. See (\ref{eqn:jac:szlifts-ZDrDprimerprime}). 

\end{list}

\end{footnotesize}

\break
 
%-----------------------------------------------------------------------------------%
\section{Jacobi Forms}\label{sec:jac}
%-----------------------------------------------------------------------------------%

In this
section we recall 
foundational
properties of the various kinds of Jacobi forms which 
are utilized in this work. 
In \S\ref{sec:jac:hol}
we recall the 
holomorphic Jacobi forms of Eichler--Zagier \cite{eichler_zagier}, and skew-holomorphic Jacobi forms of Skoruppa \cite{MR1096975,MR1074485}. In \S\ref{sec:Mock Jacobi Forms} we discuss weak mock Jacobi forms, following \cite{Dabholkar:2012nd}. 
In \S\ref{sec:jac:szlifts} we recall the lifting maps from holomorphic and skew-holomorphic Jacobi forms to modular forms with level, which were analysed by Skoruppa--Zagier \cite{MR958592} in the holomorphic case (cf. also \cite{MR909238}), and by Skoruppa \cite{MR1072974,MR1074485,MR1116103} in the skew-holomorphic case (cf. also \cite{MR1096975}). We also discuss lifting maps for weak mock Jacobi forms, which were studied by Bruinier--Ono in \cite{MR2726107}.
Finally, 
in \S\ref{sec:jac:ez} we review 
the Eichler--Zagier operators on Jacobi forms, 
and their relationship to Atkin--Lehner involutions. 

\subsection{Holomorphic and Skew-Holomorphic Jacobi Forms}\label{sec:jac:hol}

We first define {elliptic} forms, generalizing slightly the treatment in \cite{Dabholkar:2012nd}. For $m$ an integer define the index $m$ {\em elliptic action} of the group $\ZZ^2$ on functions $\phi:\HH\times\CC\to \CC$ by setting
\begin{gather} \label{elliptic}
(\f\lvert_{m} (\l,\m) )(\t,z) := \ex( m\l^2 \t + 2m\l z) \, \f(\t, z+\l \t +\m)
\end{gather}
for $(\l,\m)\in \ZZ^2$. 
Say that a smooth function $\phi:\HH\times\CC\to\CC$ is an {\em elliptic form} of index $m$ if $z\mapsto\phi(\t,z)$ is holomorphic\footnote{The elliptic forms of \cite{Dabholkar:2012nd} are assumed to be holomorphic also in $\tau$, but we wish to allow for the possibility that $\tau\mapsto \phi(\tau,z)$ is real-analytic on $\HH$, for fixed $z\in \CC$.} for every $\tau\in \HH$, and if $\phi\lvert_m(\l,\m)=\phi$ for all $(\l,\m)\in\ZZ^2$. 
Write $\E_m$ for the space of elliptic forms of index $m$. 

Observe that any elliptic form $\phi\in\E_m$ admits a {\em theta-decomposition}
\begin{gather}\label{eqn:jac:thtdec}
\phi(\tau,z)=\sum_{r\xmod 2m} h_r(\tau)\theta_{m,r}(\tau,z),
\end{gather}
for some $2m$ smooth functions $h_r:\HH\to \CC$, 
called the {\em theta-coefficients} of $\phi$. 
Indeed, since $(\phi|_m(0,1))(\tau,z)=\phi(\tau,z+1)$ we have $\phi(\tau,z)=\sum_{\ll\in\ZZ} c_\ll(\tau)y^\ll$ for some $c_\ll:\HH\to\CC$. Then the identity $\phi|_m(1,0)=\phi$ (cf. \S\ref{sec:notn})
implies that $c_r(\tau)q^{-r^2/4m}$ depends only on $r\xmod 2m$. So 
the $2m$ functions $h_r(\tau)=c_r(\tau)q^{-r^2/4m}$ are the theta-coefficients of $\phi$.

It will be convenient to regard the $h_r$ and $\th_{m,r}$ in (\ref{eqn:jac:thtdec}) as defining $2m$-vector-valued functions $h:=(h_r)_{r\xmod 2m}$ and $\th_m:=(\th_{m,r})_{r\xmod 2m}$. 
Then the theta-decomposition (\ref{eqn:jac:thtdec}) may be more succinctly written 
$\phi=h^t\th_m$,
where the superscript $t$ denotes matrix transpose.

It follows from the Poisson summation formula that the vector-valued function $\th_{m}(\tau,z)=(\th_{m,r}(\tau,z))$ satisfies
\begin{gather}\label{transf_theta}
\theta_m\left(-\frac1\tau,\frac{z}\tau\right)\frac{1}{\sqrt{\tau}}\ex\left(-\frac{mz^2}{\tau}\right)
=
{\cal S}\theta_m(\tau,z),\quad
\theta_{m}(\tau+1,z)
={\cal T}\theta_m(\tau,z),
\end{gather}
where 
${\cal S}=({\cal S}_{rr'})$ and ${\cal T}=({\cal T}_{rr'})$ are as in \S\ref{sec:notn}. 
(Cf. e.g. \S5 of \cite{eichler_zagier}.) This suggests that we will obtain elliptic forms $\f=h^t\theta_m\in\E_m$ with good transformation properties under the action of the modular group $\SL_2(\ZZ)$ by requiring suitable conditions on $h(-1/\tau)$ and $h(\tau+1)$. 
Roughly speaking, holomorphic and skew-holomorphic Jacobi forms correspond to the cases that $h$ is a (vector-valued) holomorphic or anti-holomorphic modular form, respectively. 

To formulate these notions precisely, define the weight $k$ {\em modular}, and {\em skew-modular} actions of $\SL_2(\ZZ)$ on $\E_m$, for $k$ and $m$ integers, by setting 
\begin{align} 
\label{modular}
(\f\lvert_{k,m}\g )(\t,z) &:= 
\f\left(\frac{a\t+b}{c\t+d},\frac{z}{c\t+d}\right)
\frac1{(c\t+d)^{{k}}} 
\ex\left(- \frac{c mz^2}{c\t+d}\right), 
\\
\label{skewmodular}
(\f\lvert^\sk_{k,m}\g )(\t,z) &:= 
\f\left(\frac{a\t+b}{c\t+d},\frac{z}{c\t+d}\right)
\frac1{(c\bar\t+d)^{{k}}}\frac{c\bar\t+d}{|c\t+d|}
\ex\left(- \frac{c mz^2}{c\t+d}\right),
\end{align}
for $\f\in\E_m$ and $\g=\left(\begin{smallmatrix} a&b\\c&d\end{smallmatrix}\right) \in \SL_2(\ZZ)$. 
Say that an elliptic form $\phi\in\E_m$ is a {\em weak holomorphic Jacobi form} of weight $k$ and index $m$ for
$\SL_2(\ZZ)\ltimes \ZZ^2$ if it satisfies the following conditions.  
First, its theta-coefficients are holomorphic functions on $\HH$; second, it
is invariant for the weight $k$ modular action (\ref{modular}), so that
$\f\lvert_{k,m}\g=\f$ for all $\g \in \SL_2(\ZZ)$; 
finally, the function $\tau\mapsto \phi(\tau,z)$ remains bounded as $\Im(\tau)\to \infty$, for every $z\in \CC$.

It follows from the holomorphicity condition and the {translation invariance} $\phi|_{k,m}\left(\begin{smallmatrix}1&1\\0&1\end{smallmatrix}\right)=\phi$ that a weak holomorphic Jacobi form admits a {\em Fourier expansion} of the form
\begin{gather}\label{eqn:jac:hol-DFoucffhol}
\f(\t,z) = \sum_{\substack{D,\ll \in \ZZ\\D=\ll^2\xmod 4m}} C_\f(D,\ll) q^{-D/4m}q^{\ll^2/4m} y^\ll,
\end{gather}
for some  
$2m$ functions $D\mapsto C_{\f}(D,r)$. 
Indeed, these Fourier coefficients are related to the theta-coefficients by 
\begin{gather}\label{eqn:jac:hol-thtcoeffFoucoeff}
h_r(\tau)=\sum_{\substack{D\in\ZZ\\D=r^2\xmod 4m}}C_\phi(D,r)q^{-D/4m}.
\end{gather}
Note that $C_\f(D,\ell)=0$ when $\ell^2-D<0$ according to the growth condition on $\t\mapsto\f(\t,z)$.

With (\ref{eqn:jac:hol-thtcoeffFoucoeff}) in mind, a weak holomorphic Jacobi form $\f$ is called a 
{\em holomorphic Jacobi form}, or a {\em cuspidal holomorphic Jacobi form},  when the Fourier coefficients satisfy 
$C_\f(D,r)=0$ for $D > 0$, or $C_{\f}(D,r)=0$ for $D \ge0$, respectively. 
We denote the space of weak holomorphic Jacobi forms of weight $k$ and index $m$ by $J^\wk_{k,m}$, write $J_{k,m}$ for the subspace of holomorphic Jacobi forms, and write $S_{k,m}$ for the subspace of cuspidal holomorphic Jacobi forms.

If $\phi\in J^\wk_{k,m}$ 
then the $C_\f(D,\ll)$ with $D>0$ measure the failure of $\phi$ to lie in $J_{k,m}$. 
We call these $C_\f(D,r)$ the {\em positive (discriminant) coefficients} of $\phi$, and we define the {\em positive (discriminant) part} of $\phi$ by setting
\begin{gather}\label{eqn:jac:hol-polarpart}
	\Dp(\phi)(\t,z):=
	\sum_{\substack{D,\ll \in \ZZ\\D=\ll^2\xmod 4m\\D>0}} C_\f(D,\ll) q^{-D/4m}q^{\ll^2/4m} y^\ll. 
\end{gather}

We now turn to the closely related skew-holomorphic Jacobi forms. 
An elliptic form $\f\in\E_m$ is called 
a {\em weak skew-holomorphic Jacobi form}
if it meets the following conditions. 
First, its theta-coefficients are anti-holomorphic functions on $\HH$; second, it is invariant for the weight $k$ 
skew-modular action (\ref{skewmodular}), so that $\f\lvert_{k,m}^\sk\g =\f$ for all $\g \in \SL_2(\ZZ)$; finally, $\tau \mapsto \f(\t,z)$ remains bounded as $\Im(\tau)\to \infty$ for fixed $z\in \CC$. Thus a weak skew-holomorphic Jacobi form admits a Fourier expansion of the form
\begin{gather}\label{eqn:jac:hol-DFoucffskw}
 \f(\t,z) = 
 \sum_{\substack{D,\ll\in \ZZ\\  D=\ll^2  \xmod 4m} }C_\f(D,\ll)\, 
 \bar q^{D/4m} q^{\ll^2/4m} y^\ll,
\end{gather}
for some $2m$ functions $D\mapsto C_\f(D,r)$, and we recover its theta-coefficients by writing
\begin{gather}\label{eqn:jac:hol-thtcoeffFoucoeff-skew}
h_r(\tau)=\sum_{\substack{D\in\ZZ\\D=r^2\xmod 4m}}C_\f(D,r)\bar q^{D/4m}.
\end{gather}

A weak skew-holomorphic Jacobi form $\f$ is called a {\em skew-holomorphic Jacobi form}, or a {\em cuspidal skew-holomorphic Jacobi form}, when the Fourier coefficients satisfy $C_\f(D,r)=0$ for $D < 0$, or $C_{\f}(D,r)=0$ for $D \le0$, respectively. 
We denote the space of skew-holomorphic Jacobi forms of weight $k$ and index $m$ by $J^\sk_{k,m}$, and we write $S^\sk_{k,m}$ for the subspace of cuspidal skew-holomorphic Jacobi forms. 

As we have alluded to above, the identities (\ref{transf_theta}) indicate a close relationship between Jacobi forms and modular forms of half integral weight. To make this explicit, 
define the weight $1/2$ action of $\mpt(\ZZ)$ 
on $\E_m$
by setting
\begin{gather}\label{eqn:jac:hol-wthlfindmaction}
	\left(\f|_{\frac12,m}(\gamma,\upsilon)\right)(\tau,z):=\f\left(\frac{a\tau+b}{c\tau+d},\frac{z}{c\tau+d}\right)\frac1{\upsilon(\tau)}\ex\left(-\frac{cmz^2}{c\tau+d}\right)
\end{gather}
for $\f\in \E_m$ when $\gamma=\left(\begin{smallmatrix}a&b\\c&d\end{smallmatrix}\right)$.
Then (\ref{transf_theta}) implies that the $\theta_{m,r}$ span a module $\Theta_m$ 
for $\mpt(\ZZ)$ under this action. 

Recall (cf. \S\ref{sec:notn}) that the assignment $\gamma\mapsto (\gamma,j(\gamma,\tau))$ defines an embedding of groups $\Gamma_1(4)\to \mpt(\ZZ)$, which we may use 
to identify any $\Gamma<\Gamma_1(4)$ as a subgroup of $\mpt(\ZZ)$. 
With this understanding, we may conclude from, e.g., Lemma 1.2 in \cite{Sko_Thesis}, that the $\mpt(\ZZ)$-module 
$\Theta_m$
becomes trivial when restricted to the normal subgroup $\Gamma(4m)<\mpt(\ZZ)$. 
This implies that 
the theta-coefficients of a holomorphic 
Jacobi form 
of weight $k$ and index $m$ belong to $M_{k-\frac12}(\Gamma(4m))$, and similarly in the skew-holomorphic case.
(It also implies that $\th^{k-1}_{m,r}\in M_{k-\frac12}(\Gamma(4m))$ for $k\in \{1,2\}$.) 
Observe that $M_{k-\frac12}(\Gamma(4m))$ is naturally an $\mpt(\ZZ)$-module by virtue of the fact that $\Gamma(4m)$ is normal in $\mpt(\ZZ)$. Comparing (\ref{modular}) and (\ref{skewmodular}) with (\ref{eqn:jac:hol-wthlfindmaction}) 
we see that the assignment $\phi=\sum_r h_r\theta_{m,r}\mapsto \sum_r h_r\otimes \theta_{m,r}$ defines an isomorphism
from $J_{k,m}$ onto the space of $\mpt(\ZZ)$-invariants in $M_{k-\frac12}(\Gamma(4m))\otimes \Theta_m$, and similarly for $J^\sk_{k,m}$, 
\begin{gather}
	\label{eqn:jac:hol-jacmetinv}
	J_{k,m}\oplus J^\sk_{k,m} \xrightarrow{\sim} \left( \left(M_{k-\frac12}(\Gamma(4m))\oplus \overline{M_{k-\frac12}(\Gamma(4m))}\right)\otimes \Theta_m \right)^{\mpt(\ZZ)}. 
\end{gather}
Equivalently, if $\f\in J_{k,m}$ and $\f=h^t\th_m$, then $h$ is a vector-valued modular form of weight $1/2$, in the sense that we have $h|_{k-\frac12}(\gamma,\upsilon)=\overline{\varrho_m((\gamma,\upsilon))}h$ for $(\gamma,\upsilon)\in \mpt(\ZZ)$, whereas for $\f\in J^\sk_{k,m}$ with $\f=h^t\th_m$, upon setting $g:=\overline{h}$ we obtain $g|_{k-\frac12}(\gamma,\upsilon)=\varrho_m((\gamma,\upsilon))g$.

Define the {\em Petersson inner product} of a pair $\f,\f'\in J_{k,m}$, or $\f,\f'\in J^\sk_{k,m}$, with at least one of $\f$ or $\f'$ cuspidal, by setting 
\begin{gather}\label{eqn:jac:hol-PetIP}
	\langle\f,\f'\rangle:=\frac{1}{\sqrt{2m}}\sum_{r\xmod 2m}\int_{\mathcal{F}}h_{r}(\tau)\overline{h'_{r}(\tau)}
	v^{k-\frac52}{\rm d}u{\rm d}v
\end{gather}
where the $h_{r}$ and $h'_r$ are the theta-coefficients of $\f$ and $\f'$, respectively, 
and $\t=u+iv$.
(Cf. Theorem 5.3 in \cite{eichler_zagier}, and p.75 of \cite{MR1072974} or p.511 of \cite{MR1074485}.)
The Petersson inner product is non-degenerate on $S_{k,m}\oplus S^\sk_{k,m}$.

It will be important in what follows to distinguish the skew-holomorphic Jacobi forms $\f\in J^\sk_{k,m}$ whose Fourier coefficients $C_\f(D,r)$ are supported on the $D$ 
that are perfect squares. 
These are the skew-holomorphic Jacobi forms that are called {trivial} in \S3 of \cite{MR1116103} (cf. also \S1 of \cite{MR1074485}). In this work, we say that such skew-holomorphic Jacobi forms are of {\em theta type}, and write $T^{\sk}_{k,m}$ for the subspace of $J^\sk_{k,m}$ that they comprise. 
\begin{gather}\label{eqn:jac:hol-Tskkm}
	T^{\sk}_{k,m}:=\left\{\phi\in J^\sk_{k,m}\mid C_\phi(D,r)\neq 0\Rightarrow \XBox(D)=1 \right\} 
\end{gather}
The prototypical example of a skew-holomorphic Jacobi form of theta type in $T^{\sk}_{k,m}$, for $k=1$ or $k=2$, is 
the function $t_{k,m}$ defined in \S\ref{sec:notn}.

We will 
present a characterization of the theta type skew-holomorphic Jacobi forms in \S\ref{sec:class:defn}, where
it will develop (cf. Proposition \ref{prop:class:tht-tht}) that any such form has theta-coefficients that are linear combinations of thetanullwerte. (In particular, $T^{\sk}_{k,m}=\{0\}$ unless $k\in \{1,2\}$.) This result motivates our choice of terminology.

\subsection{Mock Jacobi Forms}\label{sec:Mock Jacobi Forms}

We now discuss 
weak mock Jacobi forms, which are our main object of study. 
We 
follow the treatment of \cite{Dabholkar:2012nd} quite closely. The closely related notion of harmonic Maass--Jacobi form was introduced by Bringmann--Richter in \cite{MR2680205}.

To begin, let $w\in \frac12\ZZ$, and suppose to be given a holomorphic function $g:\HH\to \CC$ such that $g(\tau)=O(1)$ as $\Im(\t)\to \infty$, or $g(\t)=0$ as $\Im(\t)\to \infty$ if $w\leq 1$. 
Define the {\em non-holomorphic weight $w$ Eichler integral} of $g$, denoted $g^*(\t)$, by setting\footnote{Note that our $g^*$ is $-(4\pi)^{w-1}$ times the non-holomorphic Eichler integral defined in \S7 of \cite{Dabholkar:2012nd}. Our choice of scaling is motivated by Proposition \ref{prop:jac:hol-lampol}.}
\be\label{def:Eichler_integral}
g^\ast (\t):= -2^{w-1}\ex(\tfrac{w-1}{4}) \int_{-\bar \tau}^\inf (\t'+\t)^{-w} \overline{g(-\overline{\t'})} {\rm d}\t',
\ee
so that
\begin{gather}\label{eqn:jac:mck-partialgstargbar}
-2i\Im(\t)^w\frac{\pa}{\pa \bar \tau} g^\ast(\t) =  
\overline{g(\t)}. 
\end{gather}
We follow tradition by suppressing the weight $w$ from notation. In applications, $g$ will be a 
cusp form of weight $5/2-k$ for $\Gamma(4m)$, for some $k$ and $m$, and then $w=k-\frac12$. 

Next, for integers $k$ and $m$, say that 
an elliptic form $\phi\in\E_m$ is a {\em weak mock 
Jacobi form} of weight $k$ and index $m$ if the following is true. First, $\tau\mapsto \phi(\t,z)$ 
remains bounded as $\Im(\tau)\to \infty$ for every fixed $z\in\CC$; second, the theta-coefficients of $\phi$ are holomorphic; finally, there exists 
a cuspidal skew-holomorphic Jacobi form $\s=\sum_r \overline{g_r}\theta_{m,r}\in S^\sk_{3-k,m}$,
such that if we let
\begin{gather}\label{eqn:jac:mck-hhat}
\hat{h}_r(\t):=
	h_r(\t)+\frac1{\sqrt{2m}}g_r^*(\t)
\end{gather} 
for $w=k-\frac12$ (cf. (\ref{def:Eichler_integral})),
where 
$\f=\sum_r h_r\theta_{m,r}$, 
then 
$\hat\phi:=\sum_r\hat{h}_r\th_{m,r}$ is invariant for the weight $k$ 
modular action $\lvert_{k,m}$ of $\SL_2(\ZZ)$ on $\E_m$ (cf. (\ref{modular})), so that $\hat\phi\lvert_{k,m}\gamma=\hat\phi$ for all $\g\in\SL_2(\ZZ)$. 

It follows from this that $\phi$ itself 
must satisfy translation invariance $\phi(\tau+1,z)=\phi(\tau,z)$, so a weak mock Jacobi form $\phi$ admits a Fourier expansion 
\begin{gather}\label{eqn:jac:mock-Fou}
	\phi(\tau,z)=\sum_{\substack{D,\ll\in\ZZ\\D=\ll^2\xmod 4m}}C_\phi(D,\ll)q^{-D/4m}q^{\ll^2/4m}y^\ll,
\end{gather}
with $C_\f(D,\ell)=0$ whenever $\ell^2-D<0$,
just as in the case of weak holomorphic Jacobi forms (cf. (\ref{eqn:jac:hol-DFoucffhol})).
Using the above notation we call $\s$ 
the {\em shadow} of $\phi$ (it is uniquely determined by the conditions it must satisfy), and we call $\hat{\phi}$ 
the {\em completion} of $\phi$. Note that the completion of a weak mock Jacobi form is a harmonic Maass--Jacobi form in the sense of Bringmann--Ricther \cite{MR2680205}.

We remark that one may consider mock Jacobi forms with non-cuspidal, or even weakly-holomorphic shadow. Cf. \S3 of \cite{BruFun_TwoGmtThtLfts} and \S7.1 of \cite{Dabholkar:2012nd}. However, our interest in this work is in mock Jacobi forms whose shadows are ``as close as possible'' to being trivial, or more precisely, are of theta type (cf. (\ref{eqn:jac:hol-Tskkm}), \S\ref{sec:class:defn}). 
Moreover, our focus is on mock Jacobi forms of weight $1$. 
It will develop in \S\ref{sec:class:theta} that any theta type skew-holomorphic Jacobi form of weight $2$ is cuspidal (cf. Lemma \ref{lem:class:tht-TS}), so it is general enough for our purposes to consider mock Jacobi forms with cuspidal shadow.

We denote the space of weak mock Jacobi forms of weight $k$ and index $m$ by ${\mathbb J}^\wk_{k,m}$, although it is typical elsewhere in the literature to omit the superscript $^\wk$ since there seem to be very few non-zero examples of $\f\in\JJ_{k,m}^\wk$ with $C_\f(D,r)=0$ for $D>0$. (Cf. Proposition \ref{prop_uniqueness1} for a concrete result along these lines.) With this notation, the 
{\em shadow map} is
\begin{gather}\label{eqn:jac:mock-xi}
\begin{split}
\xi :{\mathbb J}^\wk_{k,m} &\to  S^\sk_{3-k,m}\\
\phi&\mapsto\s
\end{split}
\end{gather}
for $\phi \in {\mathbb J}^\wk_{k,m}$ with $\s\in S^\sk_{3-k,m}$ as above. This is closely related to the differential operator (15) of \cite{MR2680205}.

Just as in the case that $\phi$ is a weak holomorphic Jacobi form, we call the $C_\f(D,\ll)$ in (\ref{eqn:jac:mock-Fou}) with $D>0$ the {\em positive (discriminant) coefficients} of $\f$, for $\f\in\JJ^\wk_{k,m}$ (cf. (\ref{eqn:jac:hol-thtcoeffFoucoeff})), and we define the {\em positive (discriminant) part} of $\phi$ by setting
\begin{gather}\label{eqn:jac:mock-polarpart}
	\Dp(\phi)(\t,z):=
	\sum_{\substack{D,\ll \in \ZZ\\D=\ll^2\xmod 4m\\D>0}} C_\f(D,\ll) q^{-D/4m}q^{\ll^2/4m} y^\ll.
\end{gather}

There is no direct analogue of the Petersson inner product for mock Jacobi forms. However, following Brunier--Funke (cf. (3.9) of \cite{BruFun_TwoGmtThtLfts}) we may consider the pairing\footnote{Note that our $\{\cdot\,,\cdot\}$ is sequilinear, antilinear in the right-hand slot, whereas the similarly denoted pairing in \cite{BruFun_TwoGmtThtLfts} is $\CC$-bilinear.}
\begin{gather}\label{eqn:jac:mock-BFpairing}
\{\cdot\,,\cdot\}:\JJ^\wk_{k,m}\times J_{3-k,m}^\sk\to\CC
\end{gather} 
defined by setting
$\{\f,\f'\}:=\langle\xi(\f),\f'\rangle$, where $\langle\cdot\,,\cdot\rangle$ is the Petersson inner product (\ref{eqn:jac:hol-PetIP}) on $S^\sk_{3-k,m}\times J^\sk_{3-k,m}$. 
With this definition, Proposition 3.5 of \cite{BruFun_TwoGmtThtLfts} translates into the following statement. 
\begin{prop}[Bruinier--Funke]\label{prop:jac:hol-lampol}
Let $\f\in \JJ^\wk_{k,m}$ and $\f'\in J^\sk_{3-k,m}$. Then we have
\begin{gather}\label{pairing_Petersson2}
	\{\f,\f'\}=
	\sum_{r\xmod 2m}\sum_{D\geq 0}C_\phi(D,r)\overline{C_{\f'}(D,r)}.
\end{gather}
\end{prop}
\begin{proof}
If $\xi(\f)=\s=\sum_r\overline{g_r}\th_{m,r}$ and $\f'=\sum_r\overline{g'_r}\th_{m,r}$ then by definition (\ref{eqn:jac:hol-PetIP}) we have 
\begin{gather}
\{\f,\f'\}=\frac1{\sqrt{2m}}\sum_{r\xmod 2m}\int_{\mathcal{F}}\overline{g}_rg_r' v^{\frac12-k}{\rm d}u{\rm d}v.
\end{gather}
Following the proof of Proposition 3.5 in \cite{BruFun_TwoGmtThtLfts} we consider the holomorphic $1$-form $\a=\sum_r\hat{h}_rg_r'{\rm d}\t$ where $\hat{h}_r=h_r+\frac{1}{\sqrt{2m}}g_r^*$ (cf. (\ref{eqn:jac:mck-hhat})). We compute ${\rm d}\a=-\frac{1}{\sqrt{2m}}\sum_r\overline{g_r}g_r'v^{\frac12-k}{\rm d}u{\rm d}v$ using (\ref{eqn:jac:mck-partialgstargbar}) and conclude that $\{\f,\f'\}=-\int_{\mathcal{F}}{\rm d}\a$. 
So $\{\f,\f'\}=\lim_{t\to\infty}\int_{\partial\mathcal{F}_t}(-\a)$ by Stokes' theorem. Noting that $\alpha$ is $\SL_2(\ZZ)$-invariant we obtain 
\begin{gather}
\begin{split}
\int_{\partial\mathcal{F}_t}(-\a)&=\sum_{r\xmod 2m}\int_{-\frac12}^{\frac12} \hat{h}_r(u+it)g'_r(u+it){\rm d}u\\
	 &= \sum_{r\xmod 2m}\sum_{D\geq 0}C_\f(D,r)\overline{C_{\f'}(D,r)}+O(e^{-\varepsilon t})
\end{split}
\end{gather}
for some $\varepsilon>0$. The required identity (\ref{pairing_Petersson2}) follows.
\end{proof}
See Proposition 2 
of \cite{MR2805582} for a closely related result, where certain harmonic Maass--Jacobi forms take on the role played by weak mock Jacobi forms here.

\subsection{Theta Lifts}\label{sec:jac:szlifts}

In this section we recall certain lifting maps from Jacobi forms to modular forms 
which will play important roles in the proofs of our main results. Such maps first appeared in \S5 of \cite{eichler_zagier} for the case of holomorphic Jacobi forms of index one. 
Shimura's correspondence \cite{MR0332663}, and related works of Niwa \cite{MR0364106,MR0562506} and Kohnen \cite{MR575942,MR660784} serve as important antecedents.
Later, lifts for holomorphic Jacobi forms of arbitrary index were analyzed by Skoruppa--Zagier in \cite{MR958592}, and Skorrupa studied the skew-holomorphic case in \cite{MR1072974,MR1074485,MR1116103} (cf. also \cite{MR1096975}). 
As we will describe, the result of these analyses is a series of injective maps of Hecke algebra modules
\begin{gather}\label{eqn:jac:szlifts-injmap}
	\SZ:S_{k,m}\oplus S^\sk_{k,m} \to M_{2k-2}(\Gamma_0(m))
\end{gather}
(cf. \S\ref{sec:notn})
for $k\geq 2$ and $m\geq 1$, whose image contains all cuspidal newforms (and is entirely cuspidal if $k>2$). 
We will also explain how results of Bruinier--Ono \cite{MR2726107} furnish an extension of (\ref{eqn:jac:szlifts-injmap}) to $k=1$.

The embeddings (\ref{eqn:jac:szlifts-injmap}) are described somewhat indirectly, as linear combinations of linear maps $\mathcal{S}_{D,r}$, where $D=r^2\xmod 4m$. 
To 
motivate the $\mathcal{S}_{D,r}$, suppose that $f\in S_{2k-2}(\Gamma_0(m))$ is a newform, $f(\tau)=\sum_{n>0}c_f(n)q^n$, normalized so that $c_f(1)=1$. 
Then $f$ is an eigenform for the Hecke operators 
\begin{gather}\label{eqn:jac:szlifts-Heckemod}
	c_{f|T_n}(n')=\sum_{d|(n,n')}d^{2k-3}c_f\left(\frac{nn'}{d^2}\right),
\end{gather}
and in particular, its coefficients are its Hecke eigenvalues, $c_{f|T_n}(1)=c_f(n)$.
If we assume that $\SZ$ is an embedding of Hecke algebra modules then there is a corresponding $\phi\in S_{k,m}\oplus S^\sk_{k,m}$ having the same Hecke eigenvalues as $f$, so if $D$ and $r$ are chosen so that $C_\phi(D,r)$ is non-zero, then $\sum_{n>0}C_{\phi|T_n}(D,r)q^n$ is proportional to $f$. This reasoning suggests that the map $\SZ$---assuming it exists---should be expressible as a linear combination of the {\em Skoruppa--Zagier lifts}
\begin{gather}\label{eqn:jac:szlifts-SDr}
	\begin{split}
	\mathcal{S}_{D,r}:S_{k,m}\oplus S^\sk_{k,m}&\to M_{2k-2}(\Gamma_0(m))\\
		\phi&\mapsto \sum_{n\geq 0}C_{\phi|T_n}(D,r)q^n.
	\end{split}
\end{gather}

The constant term of $\mathcal{S}_{D,r}\phi$, here denoted $C_{\phi|T_0}(D,r)$, requires some explanation. It turns out to be zero unless $k=2$ and $\phi$ has a non-trivial component in $S^\sk_{2,m}$. Certainly $\mathcal{S}_{D,r}$ is identically zero unless $D=r^2\xmod 4m$. In order to recover the embeddings (\ref{eqn:jac:szlifts-injmap}) it suffices to consider the $\mathcal{S}_{D,r}$ for which $D$ is a fundamental discriminant (i.e., $1$ or the discriminant of some quadratic number field).
With this restriction on $D$ in place define the constant term $C_{\phi|T_0}(D,r)$ to be zero unless $k=2$ and $D=1$. If $\phi\in S^\sk_{2,m}$ and $a^2=1\xmod 4m$ 
then set
\begin{gather}\label{eqn:jac:szlifts-SDrcnst}
	C_{\phi|T_0}(1,a):=\lab\phi\cdot a,t_{2,m}\rab
\end{gather}
(cf. the Corollary of \cite{MR1074485}) where $\phi\cdot a:=\sum_rh_r\th_{m,ra}$ for $\phi=\sum_r h_r\th_{m,r}$ (cf. (\ref{eqn:jac:EZ-phidota})). 
The pairing in (\ref{eqn:jac:szlifts-SDrcnst}) is the Petersson inner product (\ref{eqn:jac:hol-PetIP}) on $S^\sk_{2,m}$, and $t_{2,m}$ is
as in \S\ref{sec:notn}.

The Hecke operators $T_n$ for holomorphic Jacobi forms were introduced in \S4 of \cite{eichler_zagier}. See \S3 of \cite{MR1116103} for the skew-holomorphic case. We obtain a completely explicit description of the $\mathcal{S}_{D,r}$, for $D$ fundamental, by noting that 
\begin{gather}\label{eqn:jac:szlifts-SDrfund}
	C_{\phi|T_n}(D,r)=\sum_{d|n}
		d^{k-2}\left(\frac{D}{d}\right)C_\phi\left(\frac{n^2}{d^2}D,\frac{n}{d}r\right)
\end{gather}
(so long as $D$ is fundamental). 
The formula (\ref{eqn:jac:szlifts-SDrfund}) makes it clear that $\mathcal{S}_{D,r}$ vanishes on $S_{k,m}$ if $D>0$, and vanishes on $S^\sk_{k,m}$ if $D<0$. Cf. \S\ref{sec:notn} for the Kronecker symbol $\left(\frac{\,\cdot\,}{\,\cdot\,}\right)$.

For completeness we give a full description of the Hecke operators $T_n$ on $J_{k,m}\oplus J^\sk_{k,m}$, but only for $(n,m)=1$; a restriction which is natural in light of (\ref{eqn:jac:szlifts-injmap}). We define $\phi|T_n$ 
for $\phi\in J_{k,m}\oplus J^\sk_{k,m}$ 
by requiring that its Fourier coefficients are given by
\begin{gather}\label{eqn:jac:hol-phiTn}
	C_{\phi|T_n}(D,r)=
	\sum_{d}
	d^{k-2}\varepsilon_D(d)C_\phi\left(\frac{n^2}{d^2}D,r'\right),
\end{gather}
where 
the sum is over divisors $d$ of $n^2$ such that $d^2|n^2D$, and such that there exists an $r'\xmod 2m$ satisfying $nr=dr'\xmod 2m(n,d)$ and $(nr)^2=(dr')^2\xmod 4m$. Note that such an $r'$ is unique mod $2m$ if it exists. The symbol $\varepsilon_D(d)$ is 
defined by setting 
$\varepsilon_D(d)=g\left(\frac{D/g^2}{d/g^2}\right)$ 
if $(d,D)=g^2$ and $D/g^2$ is a square mod $4$, and $\varepsilon_D(d)=0$ otherwise.
Cf. (5) in \S0 of \cite{MR958592}, and \S3 of \cite{MR1116103}.

These operators $T_n$ for $(n,m)=1$ 
generate commutative subalgebras of $\End(J_{k,m})$ and $\End(J^\sk_{k,m})$ which preserve the cuspidal subspaces, are self-adjoint with respect to the 
{Petersson inner product} (\ref{eqn:jac:hol-PetIP}), and satisfy the same defining relations 
\begin{gather}\label{eqn:jac:szlifts-Heckerels}
	T_n T_{n'}=\sum_{d|(n,n')}d^{2k-3}T_{nn'/d^2}
\end{gather}
(cf. Corollary 1 in \S4 of \cite{eichler_zagier})
as the usual Hecke operators (\ref{eqn:jac:szlifts-Heckemod}) acting on modular forms of weight $2k-2$ and level $m$. 
We refer to the abstract algebra generated by symbols $T_n$ for $n$ positive and coprime to $m$, subject to the relations (\ref{eqn:jac:szlifts-Heckerels}), as the {\em Hecke algebra} at weight $2k-2$ and level $m$.

We now state a theorem which summarizes the main properties of the maps $\SZ$, as established in \cite{MR958592,MR1072974,MR1074485,MR1116103}. In preparation for this, 
define $P^\sk_{2,m}$ following \cite{MR1116103} (see p.104) to be the subspace of $S^\sk_{2,m}$ spanned by Hecke eigenforms that do not belong to $T^\sk_{2,m}$. (We will verify in the sequel that $S^\sk_{2,m}=T^\sk_{2,m}\oplus P^\sk_{2,m}$ as inner product spaces. In particular, $T^\sk_{2,m}$ is a subspace of $S^\sk_{2,m}$. Cf. Proposition \ref{prop:class:tht-orthog}.)

\begin{thm}[Skoruppa--Zagier, Skoruppa]\label{thm:jac:szlifts-sz}
Let $m$ be a positive integer and let $k$ be an integer greater than $1$. Then there is some linear combination 
of the operators $\mathcal{S}_{D,r}$, for $D$ fundamental and $r^2=D \xmod 4m$, that defines an injective map 
\begin{gather}
	\SZ:S_{k,m}\oplus S^\sk_{k,m} \to M_{2k-2}(\Gamma_0(m))
\end{gather}
of modules for the Hecke algebra at weight $2k-2$ and level $m$. 
If $\phi\in S_{k,m}$ and $f=\SZ(\phi)$ then $f|W_m=(-1)^kf$, whereas if $\phi\in S^\sk_{k,m}$ then $f|W_m=(-1)^{k-1}f$.
The image of $\SZ$ contains all cuspidal newforms. 
A cuspidal Hecke eigenform $f\in M_2(\Gamma_0(m))$ such that $f|W_m=-f$ is the image under $\SZ$ of an element of $P^\sk_{2,m}$.
\end{thm}
Note that the image of $\SZ$ can be described explicitly. See \cite{MR958592}.

Two 
families of {\em Hecke-like} operators on Jacobi forms were introduced in \cite{eichler_zagier}. These are the $U_d$ and 
$V_\ell$, for $d,\ell\geq 1$, mapping $J_{k,m}$ to $J_{k, md^2}$ and 
$J_{k, m\ell}$, respectively, 
and similarly for skew-holomorphic forms. They may be defined by requiring that
\begin{gather}\label{eqn:jac:szlifts-Ud}
	C_{\phi|U_d}(D,r)=\begin{cases} C_\phi(D,\frac{r}{d})&\text{ if $r=0\xmod d$}\\ 0&\text{ if $r\neq 0\xmod d$}\end{cases},\\
	\label{eqn:jac:szlifts-Vell}
	C_{\phi|V_\ell}(D,r)=\sum_{d|\left(\frac{r^2-D}{4m\ell},r,\ell\right)}d^{k-1}C_\phi\left(\frac{D}{d^2},\frac{r}{d}\right),
\end{gather}
for $\phi\in J_{k,m}\oplus J^\sk_{k,m}$. In terms of the Skoruppa--Zagier correspondence $\SZ$, the operator 
$U_d\circ V_\ell$ 
corresponds to $B_{\ell,d}$, mapping $M_k(\Gamma_0(m))$ to $M_k(\Gamma_0(m\ell d^2))$ according to $(f|B_{\ell,d})(\tau):=\sum_{t|\ell}t^{k/2}f(td\tau)$. That is, we have 
\begin{gather}\label{eqn:jac:szlifts-Belld}
\SZ(\phi|U_d\circ V_\ell)=\SZ(\phi)|B_{\ell,d}
\end{gather}
in $M_k(\Gamma_0(m\ell d^2))$, for $\phi\in S_{k,m}\oplus S^\sk_{k,m}$. (Cf. the proof of Theorem 5 in \S3 of \cite{MR958592}.)

It will be useful for us in \S\ref{sec:class:proof} that the $\mathcal{S}_{D,r}$ can be described as theta lifts. For example, if $k=2$ and $D>0$ then we have
\begin{gather}\label{eqn:jac:szlifts-SDrTheta}
	(\mathcal{S}_{D,r}\phi)(\tau')=\frac{i}{\sqrt{mD}}\lab \phi(\cdot\,,\cdot),\Theta_{D,r}(\cdot\,,\cdot\,,\tau')\rab
\end{gather}
according to the Corollary of \cite{MR1074485}, where $\lab\cdot,\cdot\rab$ is the Petersson inner product (\ref{eqn:jac:hol-PetIP}) on $S^\sk_{2,m}$, and $\Theta_{D,r}(\tau,z,\tau')$ is defined by setting
\begin{gather}\label{eqn:jac:szlifts-ThetaC}
\begin{split}	
	\Theta_{D,r}(\tau,z,\tau')
	&:=\sum_{\substack{D',r'\in \ZZ\\D'=(r')^2\xmod 4m}}C_{D,r}(D',r';\tau,\tau')\bar{q}^{D'/4m}q^{(r')^2/4m}y^{r'},
	\\
	C_{D,r}(D',r';\tau,\tau')
	&:=\sqrt{\Im(\tau)}\sum_{Q\in \mathcal{Q}(m,DD',rr')}\chi_{D}(Q)
	\frac{Q(\tau',1)}{\Im(\tau')^2}\exp\left(-\frac{\pi\Im(\tau)}{mD}\frac{\widehat{Q}(\tau')^2}{\Im(\tau')^2}\right).
\end{split}
\end{gather}

An analogous construction introduced by Bruinier--Ono \cite{MR2726107}---following earlier work of Borcherds \cite{Bor_AutFmsSngGrs}, Bruinier \cite{MR1903920}, and Bruinier--Funke \cite{BruFun_TwoGmtThtLfts}---can be applied to weak mock 
Jacobi forms of weight $1$. Specifically, for $D$ a negative fundamental discriminant such that $D=r^2\xmod 4m$ for some $r$, and for $\f\in \JJ^\wk_{1,m}$, we may consider the {\em regularized theta lift} 
\begin{gather}\label{eqn:jac:szlifts-SDrreg}
	(\mathcal{S}^\reg_{D,r}\phi)(\tau'):=\lab \f(\cdot\,,\cdot),\Theta_{D,r}(\cdot\,,\cdot\,,\tau')\rab^\reg,
\end{gather} 
where $\Theta_{D,r}(\tau,z,\tau')$ is now defined by
\begin{gather}
\begin{split}\label{eqn:jac:szlifts-ThetaCreg}
	\Theta_{D,r}(\tau,z,\tau')&:=\sum_{\substack{D',r'\in \ZZ\\D'=(r')^2\xmod 4m}}C_{D,r}(D',r';\tau,\tau'){q}^{-D'/4m}q^{(r')^2/4m}y^{r'},
	\\
	C_{D,r}(D',r';\tau,\tau')&:=\sqrt{\Im(\tau)}\sum_{Q\in \mathcal{Q}(m,DD',rr')}\chi_{D}(Q)
	\exp\left(\frac{\pi\Im(\tau)}{mD}\frac{\widehat{Q}(\tau')^2}{\Im(\tau')^2}\right).
\end{split}
\end{gather}
In (\ref{eqn:jac:szlifts-SDrreg}) we write $\lab \cdot\,,\cdot\rab^\reg$ for the {\em regularized Petersson inner product}, which is defined as follows (cf. \S6 of \cite{Bor_AutFmsSngGrs}). Given that $\f(\tau,z)=\sum_{r'\xmod 2m} h_{r'}(\tau)\th_{m,r'}(\tau,z)$ and $\Theta_{D,r}(\tau,z,\tau')=\sum_{r'\xmod 2m} \theta_{D,r,r'}(\tau,\tau')\th_{m,r'}(\tau,z)$, we consider the function 
\begin{gather}\label{eqn:jac:szlifts-Fs}
	F(s):=\lim_{t\to \infty}\int_{\mathcal{F}_t}\sum_{r'\xmod 2m} h_{r'}(\tau)\overline{\theta_{D,r,r'}(\tau,\tau')}v^{-\frac32 - s}{\rm d}u{\rm d}v,
\end{gather}
where $\tau=u+iv$. 
The limit in (\ref{eqn:jac:szlifts-Fs}) is well-defined for $\Re(s)$ sufficiently large. Analytically continue $F(s)$ to a function that is meromorphic in some domain containing the origin, and define (\ref{eqn:jac:szlifts-SDrreg}) to be the constant term in the Laurent series expansion at $s=0$. 

Note the similarity of (\ref{eqn:jac:szlifts-ThetaCreg}) to (\ref{eqn:jac:szlifts-ThetaC}). We have obtained (\ref{eqn:jac:szlifts-ThetaCreg}) by translating (5.5) of \cite{MR2726107} into the language of mock Jacobi forms. 

To describe the modularity of $\mathcal{S}_{D,r}^\reg\f$ define a (generally non-integral) divisor $Z_{D,r}(\f)$ on $X_0(m)$ by setting $Z_{D,r}(\f):=\sum_{r'\xmod 2m}\sum_{D'\geq 0} C_\f(D',r')Z_{D,r}(D',r')$ where
\begin{gather}\label{eqn:jac:szlifts-ZDrDprimerprime}
	Z_{D,r}(D',r'):=\sum_{Q\in \mathcal{Q}(m,DD',rr')/\Gamma_0(m)}\frac{\chi_D(Q)}{\#\Gamma_0(m)_Q}\a_Q.
\end{gather}
The next result follows from Proposition 5.2 and Theorem 5.3 in \cite{MR2726107}.
\begin{thm}[Bruinier--Ono]\label{thm:jac:szlifts-SDrreg}
Let $m$ be a positive integer, let $D$ be a negative fundamental discriminant and assume that $D=r^2\xmod 4m$ for some $r\xmod 2m$. If $\f\in \JJ^\wk_{1,m}$ then $\mathcal{S}_{D,r}^\reg\f$ is a $\Gamma_0(m)$-invariant function on $\HH\setminus Z_{D,r}(\f)$ with a logarithmic singularity on $-4Z_{D,r}(\f)$. For $\t\in \HH$ with $\Im(\t)$ sufficiently large we have
\begin{gather}
(\mathcal{S}_{D,r}^\reg\f)(\tau)=-4\sum_{n>0}\sum_{b\xmod D}\left(\frac{D}{b}\right)C_\f(Dn^2,rn)\log\left|1-\ex\left(n\tau+\frac b D\right)\right|.
\end{gather}
\end{thm}

Waldspurger established very general results \cite{MR577010,MR646366,MR1103429} relating properties of automorphic forms to $L$-functions of their theta lifts. 
We conclude this section with 
a Waldspurger type formula which relates Fourier coefficients of cuspidal holomorphic or skew-holomorphic Jacobi forms of weight $2$ 
to central critical values of twisted $L$-functions for the 
corresponding modular forms under the map $\SZ$ of Theorem \ref{thm:jac:szlifts-sz}. For holomorphic Jacobi forms this is the $k=1$ case of Corollary 1 in \S II.4 of \cite{MR909238}. As observed in \S1 of \cite{MR1074485}, the statement for skew-holomorphic Jacobi forms is obtained in a directly similar way, using the Proposition and Corollary in \S2 of \cite{MR1074485}.
\begin{thm}[Gross--Kohnen--Zagier, Skoruppa]\label{thm:jac:szlifts-Lfns}
Let $f$ be a cuspidal newform of weight $2$ for $\Gamma_0(m)$, and let $\phi\in S_{2,m}\oplus P^\sk_{2,m}$ be the preimage of $f$ under $\SZ$. If $D$ is a fundamental discriminant that is coprime to $m$, and $D=r^2\xmod 4m$ for some $r\xmod 2m$, then 
\begin{gather}\label{eqn:jac:szlifts-Lfns}
	\frac{|C_\phi(D,r)|^2}{\lab\phi,\phi\rab}=\frac{\sqrt{|D|}}{2\pi}\frac{L(f\otimes D,1)}{\lab f,f\rab}.
\end{gather}
\end{thm}
On the left-hand side of (\ref{eqn:jac:szlifts-Lfns}) we write $\lab\cdot\,,\cdot\rab$ for the Petersson inner product on $S_{2,m}\oplus S_{2,m}^\sk$, whereas $\lab\cdot\,,\cdot\rab$ on the right-hand side of (\ref{eqn:jac:szlifts-Lfns}) denotes the Petersson inner product on $S_2(\Gamma_0(m))$. 

\subsection{Eichler--Zagier Operators}
\label{sec:jac:ez}

For $m$ a positive integer and $n$ a divisor of $m$, define the {\em Eichler--Zagier operator} $\W_m(n)$ on 
$\E_m$,
following \cite{MR958592}, by setting 
\be\label{eqn:jac:ez-phiWmn}
( \phi\lvert {\cal W}_m{(n)}) \,(\t,z) := \frac{1}{n} \sum_{a,b = 0}^{n-1} \ex\left(m\left(\tfrac{a^2}{n^2} \t + 2 \tfrac{a}{n}z +\tfrac{ab}{n^2}\right)\right) \phi\left(\t,z+\tfrac{a}{n}\t+\tfrac{b}{n}\right)
\ee
for $\phi\in\E_m$. It is elementary to check that if 
$\phi=h^t\th_m$ 
then we have
$\phi|\W_m(n)=h^t\O_m(n)\th_m$ 
where $\O_m(n)=(\O_m(n)_{r,r'})$ is the $2m\times 2m$ {\em Omega matrix} defined by setting
\begin{align}\label{def:OmegaMatrices}
\Omega_{m}(n)_{r,r'} := \begin{cases} 1 &\text{if $r=-r'\xmod 2 n$ and $r=r'\xmod {2m}/{n}$,} \\ 
0 &{\rm otherwise}.
\end{cases}
\end{align}

Observe that $\O_m(1)=\Id$. 
More generally, $\O_m(n)$ is invertible (with order at most $2$) 
so long as $n$ is an {exact divisor} of $m$. 
Indeed, 
we have 
$\O_m(n)\O_m(n')=\O_m(n\ast n')$
whenever $n,n'|m$, and at least one of $n$ or $n'$ belongs to $\Ex_m$. 
Thus $\W_m(n)\W_m(n')=\W_m(n\ast n')$ as operators on $\E_m$, when $n\in \Ex_m$ and $n'|m$. In particular, $\W_m(n)$ is an involution on $\E_m$ when $n$ is a non-trivial exact divisor of $m$.

There is another natural source of operators on $\E_m$, coming from integers that are coprime to $m$.
Indeed, the group $(\ZZ/2m\ZZ)^*$ acts naturally on $\E_m$, according to the rule 
\begin{gather}\label{eqn:jac:EZ-phidota}
	\phi\cdot a = \sum_{r\xmod 2m}h_r\th_{m,ra},
\end{gather}
where $\phi=h^t\th_m\in \E_m$ for $h=(h_r)$, and $a$ is an invertible element of $\ZZ/2m\ZZ$. If $a$ is chosen to lie in $O_m<(\ZZ/2m\ZZ)^*$, so that 
it preserves the $\QQ/\ZZ$-valued quadratic form $x\mapsto \frac{x^2}{4m}$ on $\ZZ/2m\ZZ$ (cf. \S\ref{sec:notn}), then the action (\ref{eqn:jac:EZ-phidota}) actually preserves the subspaces $J_{k,m}$ and $J^\sk_{k,m}$, as can be seen by noting the relation between this quadratic form and the entries of ${\cal S}$ and ${\cal T}$ (cf. \S\ref{sec:notn}, (\ref{transf_theta})). 

It turns out that the groups $\Ex_m$ and $O_m$, and their actions on $\E_m$ coincide. 
To be precise, 
we have
\begin{gather}\label{eqn:jac:ez-Wmndota}
	\phi|\W_m(n)=\phi\cdot a(n)
\end{gather}
for $\phi\in \E_m$, when $n\in\Ex_m$ and $a(n)$ is as in \S\ref{sec:notn}.  
In particular, the action of $\Ex_m$ on $\E_m$ by Eichler--Zagier operators preserves the spaces $J_{k,m}$ and $J^\sk_{k,m}$. Note, however, that if $\f$ is a weak holomorphic (or skew-holomorphic) Jacobi form then $\f|\W_m(n)=\phi\cdot a(n)$ also transforms like a Jacobi form, but generally fails the growth condition (as $\Im(\tau)\to \infty$) which characterizes weak Jacobi forms.

The Eichler--Zagier operators $\W_m(n)$ for $n\in \Ex_m$ 
correspond to Atkin--Lehner involutions $W_n$ under the Skoruppa--Zagier map $\SZ$ of Theorem \ref{thm:jac:szlifts-sz}. More precisely, we have\footnote{The identity (\ref{eqn:jac:ez-ALeqv}) is proven for $\phi\in S_{k,m}$ in \cite{MR958592}, but the case that $\phi\in S^\sk_{k,m}$ does not seem to have been treated elsewhere in the literature. We verify (\ref{eqn:jac:ez-ALeqv}) for $\phi\in  S^\sk_{k,m}$ in \S\ref{sec:class:proof}.}

\begin{gather}\label{eqn:jac:ez-ALeqv}
	\SZ(\phi|\W_m(n))=\SZ(\phi\cdot a(n))=\SZ(\phi)|W_n
\end{gather}
for $\phi\in S_{k,m}\oplus S^\sk_{k,m}$, and $n\in \Ex_m$. 
(Cf. Lemma \ref{lem:class:proof-ALinv}.)

Given $\a\in\widehat{O}_m$ 
define $\E_m^\a$ to be the $\a$-eigenspace for the action of $O_m$ on $\E_m$, composed of the $\f\in \E_m$ such that $\f\cdot a=\a(a)\f$ for all $a\in O_m$. Set $J_{k,m}^\a:=J_{k,m}\cap\E_m^\a$, 
and apply the same notational convention to other $O_m$-stable subspaces of $\E_m$, such as $S_{k,m}$, $J_{k,m}^\sk$, $S^\sk_{k,m}$ and $\Th_m$.
Given (\ref{eqn:jac:ez-ALeqv}) it is natural to extend this convention to $M_k(\Gamma_0(m))$, by defining $M_k(\Gamma_0(m))^\a$ to be the subspace composed of the $f\in M_k(\Gamma_0(m))$ such that $f|W_n=\alpha(a(n))f$ for $n\in \Ex_m$. Define $S_k(\Gamma_0(m))^\a$ similarly by restricting to cusp forms. 
Then Theorem \ref{thm:jac:szlifts-sz} yields injections
\begin{gather}\label{eqn:jac:szlifts-SZalpha}
	\SZ: S^\a_{k,m}\oplus S^{\sk,\a}_{k,m}\to M_{2k-2}(\Gamma_0(m))^\a
\end{gather}
for $\a\in \widehat{O}_m$.

Observe that $\theta_{m,r}(\tau,-z)=\theta_{m,-r}(\tau,z)$. Taking this together with the modular (\ref{modular}) and skew-modular (\ref{skewmodular}) actions of $-I\in \SL_2(\ZZ)$, and the action of $-1\in O_m$ on $J_{k,m}\oplus J^\sk_{k,m}$, we see that 
$J^\a_{k,m}=\{0\}$ unless $\a(-1)=(-1)^k$, and $J^{\sk,\a}_{k,m}=\{0\}$ unless $\a(-1)=(-1)^{k+1}$. 
In particular, $S^\a_{k,m}=\{0\}$ unless $\a(-1)=(-1)^k$, and $S^{\sk,\a}_{k,m}=\{0\}$ unless $\a(-1)=(-1)^{k-1}$. Even though $\JJ_{k,m}^\wk$ is not an $O_m$-stable subspace of $\E_m$, this result nonetheless extends to weak mock Jacobi forms, in the sense that we have
\begin{gather}\label{eqn:jac:ez-Foucoeffsym}
	C_\f(D,r)=(-1)^kC_\f(D,-r)
\end{gather}
for $\f\in \JJ^\wk_{k,m}$.

Skoruppa has classified the irreducible sub $\mpt(\ZZ)$-modules of $\Th_m$, and we can describe his result using the action of $O_m$ on $\Th_m$. Notice that this action 
is unitary with respect to the inner product $\lab\cdot\,,\cdot\rab$ on $\Th_m$ which is defined by requiring that $\{\th_{m,r}\}_{r\xmod 2m}$ be an orthonormal basis. Note also that the Hecke-like operator $U_d$ (cf. (\ref{eqn:jac:szlifts-Ud})) defines $\mpt(\ZZ)$-module morphisms from $\Th_m$ to $\Th_{md^2}$. Define $\Th_m^\new$ to be the $U$-new part of $\Th_m$. That is, take $\Th_m^\new$ to be the orthogonal complement in $\Th_m$, with respect to $\lab\cdot\,,\cdot\rab$, of the sum of the subspaces $\Th_{m/d^2}|U_d$, where $d$ ranges over the integers $d>1$ such that $d^2|m$. For $\a\in \widehat{O}_m$ define\footnote{Note that $\Th_m^{\new,\a}$ is denoted ${\rm Th}_m^{1,f}$ in \cite{Sko_Thesis}, where the relationship between $\a$ and $f$ is as follows. If $m=p_1^{m_1}\cdots p_k^{m_k}$ is the decomposition of $m$ into powers of distinct primes, then $f$ is the product of the $p_i$ such that $\a(p_i^{m_i})=-1$.} $\Th_m^{\new,\a}:=\Th_m^\new\cap\Th_m^\a$. The next result is Satz 1.8 in \cite{Sko_Thesis}.

\begin{thm}[Skoruppa]\label{thm:jac:EZ-Thmdec}
For each $\a\in \widehat{O}_m$ the $\mpt(\ZZ)$-module $\Th_m^{\new,\a}$ is irreducible. Two such irreducible modules $\Th_m^{\new,\a}$ and $\Th_{m'}^{\new,\a'}$ are isomorphic if and only if $m=m'$ and $\a=\a'$. The decomposition of $\Th_m$ into irreducible $\mpt(\ZZ)$-modules is given by
\begin{gather}\label{eqn:jac:EZ-Thmdec}
	\Th_m=\bigoplus_{\substack{d>0\\d^2|m}}\bigoplus_{\a\in\widehat{O}_{m/d^2}}\left.\Th_{m/d^2}^{\new,\a}\right|U_d.
\end{gather}
\end{thm}

In \cite{Sko_Thesis} Skoruppa has used (\ref{eqn:jac:hol-jacmetinv}) and Theorem \ref{thm:jac:EZ-Thmdec}, together with the Serre--Stark theorem \cite{MR0472707} on modular forms of weight $1/2$, 
to show that there are no non-zero holomorphic Jacobi forms of weight $1$. This result will 
be important 
for us in \S\ref{sec:class:defn}.
\begin{thm}[Skoruppa]\label{thm:jac:hol-jacwt1}
We have $J_{1,m}=\{0\}$ for all $m$.
\end{thm}

%---------------------------------------------------------------------------------------%
\section{Mock Jacobi Theta Functions}\label{sec:class}
%---------------------------------------------------------------------------------------%

We introduce the notion of mock Jacobi theta function in \S\ref{sec:class:defn}. The definition involves skew-holomorphic Jacobi forms of theta type, which we consider in more detail in \S\ref{sec:class:theta}. 
Our first main result is a classification of optimal mock Jacobi theta functions in terms of genus zero subgroups of $\SL_2(\RR)$, and is established in \S\ref{sec:class:proof}. We prove our second main result in \S\ref{sec:class:rat}, and thereby classify the optimal mock Jacobi forms of weight one with rational coefficients. In \S\ref{sec:class:pm} we establish constructive relationships between optimal mock Jacobi theta functions, their shadows, and the principal moduli of the genus zero groups appearing in \S\ref{sec:class:proof}. Finally, we discuss the relationship to umbral moonshine in \S\ref{sec:class:rpm}.

\subsection{Definition}\label{sec:class:defn}

Our focus in this paper is on Jacobi forms of weight 1. 
Theorem \ref{thm:jac:hol-jacwt1} states that $J_{1,m}=\{0\}$ for all $m$, so we must consider a more general notion in order to meet non-trivial examples. Since mock Jacobi forms generalize holomorphic Jacobi forms, it is natural to ask 
if there are any non-zero (non-weak) mock Jacobi forms of weight 1.
Using Proposition \ref{prop:jac:hol-lampol} together with the aforementioned result of Skoruppa, we now verify 
that 
the answer to this question is negative.

\begin{prop}\label{prop_uniqueness1}
There are no non-zero mock Jacobi forms of weight 1. That is, if $\phi \in \JJ^\wk_{1,m}$ and $\Dp(\phi)=0$ then $\f=0$. 
\end{prop}
\begin{proof}
By hypothesis $C_\phi(D,r)=0$ for all $D>0$. So 
\eq{pairing_Petersson2} shows that $\langle\xi(\phi),\varphi\rangle= 0$ for all $\varphi \in S^\sk_{2,m}$. From the non-degeneracy of the Petersson inner product on $S^\sk_{2,m}$ we conclude that 
$\xi(\phi)=0$. In other words, $\phi$ is a holomorphic Jacobi form of weight 1 and index $m$, but such a function must vanish according to 
Theorem \ref{thm:jac:hol-jacwt1}.
\end{proof}

Proposition \ref{prop_uniqueness1} confirms that we must consider weak mock Jacobi forms of weight 1 in order to have non-zero examples. Inspired by 
\cite{Dabholkar:2012nd} and \cite{UM,MUM}, we restrict our attention to {\em optimal} mock Jacobi forms, which we define to be the $\phi\in \JJ^\wk_{k,m}$ such that $C_\phi(D,r)=0$ when $D>1$. 
\begin{gather}\label{eqn:class:defn-JJopt}
	\JJ^\opt_{k,m}:=\left\{\f\in\JJ^\wk_{k,m}\mid C_\f(D,r)=0\Leftarrow D>1\right\}.
\end{gather}
In words, the optimal mock Jacobi forms are those weak mock Jacobi forms that are ``as close as possible'' to lying in $\JJ_{k,m}$. 

Having expanded our consideration to weak mock Jacobi forms that are optimal, it is natural to look for non-zero examples with vanishing shadow. But it turns out that there are none,
for it is proven in Theorem 9.4 of \cite{Dabholkar:2012nd} that if $\phi\in J^\wk_{1,m}$ is optimal then $\phi\in J_{1,m}$, and hence $\phi=0$ according to 
Theorem \ref{thm:jac:hol-jacwt1}.
The weak mock Jacobi forms $\phi^X$ of umbral moonshine, in addition to being optimal, also have the special property that their shadows are of theta type, in the sense of \S\ref{sec:jac:hol} (cf. (\ref{eqn:jac:hol-Tskkm})). At a heuristic level we may understand this
as saying that the $\phi^X$ are ``as close as possible'' to being non-mock, 
for the theta type condition restricts so many (all, asymptotically) of the coefficients of the shadow $\xi(\phi^X)$ to be zero. 

Motivated by these considerations we say that a weak mock Jacobi form $\phi\in \JJ^\wk_{k,m}$ is a {\em mock Jacobi theta function} of weight $k$ and index $m$ if $\xi(\phi)$ belongs to $T^{\sk}_{k,m}$ (cf. (\ref{eqn:jac:hol-Tskkm})). We define $\JJ^\top_{k,m}$ to be the space of mock Jacobi theta functions of weight $k$ and index $m$ that are optimal. 
\begin{gather}\label{eqn:class:defn-JJtop}
	\JJ^\top_{k,m}:=\left\{\phi\in \JJ^\opt_{k,m}\mid \xi(\phi)\in T^{\sk}_{k,m} 
	\right\}
\end{gather}
It will develop in \S\ref{sec:mockthetafunctions} that all of Ramanujan's mock theta functions admit simple expressions in terms of the theta-coefficients of mock Jacobi theta functions. This result further motivates the terminology.

As we have observed in \S\ref{sec:jac:ez}, if $\phi\in \JJ^\wk_{k,m}$ 
then $\phi|\W_m(n)$ might not be a weak mock Jacobi form, since it may fail to remain bounded as $\Im(\tau)\to \infty$, for all fixed $z\in \CC$. However, if $\phi$ is an optimal mock Jacobi form 
then no such poles can arise, so the Eichler--Zagier operators $\W_m(n)$ 
act naturally on $\JJ^\opt_{k,m}$.
Moreover, if $\phi\in \JJ^\opt_{k,m}$ 
then it follows from (\ref{eqn:jac:ez-Wmndota}), and the definition of the shadow map (\ref{eqn:jac:mock-xi}), that 
\begin{gather}
	\xi(\f|\W_m(n))=\xi(\f)|\W_m(n)
\end{gather}
when 
$n$ is an exact divisor of $m$.
It is apparent from (\ref{eqn:jac:ez-Wmndota}) that the Fourier coefficients of $\xi(\f)|\W_m(n)$ are supported on perfect square values of $D$ if and only if the same is true of $\xi(\f)$, so the 
Eichler--Zagier operators $\W_m(n)$ 
preserve the spaces 
of optimal mock Jacobi theta functions. Thus we have direct sum decompositions
\begin{gather}\label{eqn:class:defn-JJtopkmOmdec}
	\JJ^\opt_{k,m}=\bigoplus_{\a\in\widehat{O}_m}\JJ^{\opt,\a}_{k,m},\quad
	\JJ^\top_{k,m}=\bigoplus_{\a\in\widehat{O}_m}\JJ^{\top,\a}_{k,m},
\end{gather}
for each $k$ and $m$, and the shadow map (\ref{eqn:jac:mock-xi}) restricts to 
\begin{gather}\label{eqn:class:defn-xiJJtopa}
\xi:\JJ^{\top,\a}_{k,m}\to T^{\sk,\a}_{3-k,m}
\end{gather} 
for each $\a\in \widehat{O}_m$.
Note that we have $\JJ^{\opt,\a}_{k,m}=\JJ^{\top,\a}_{k,m}=\{0\}$ unless $\a(-1)=(-1)^k$ by force of (\ref{eqn:jac:ez-Foucoeffsym}).

As mentioned in \S\ref{sec:intro}, all the mock Jacobi forms $\phi^X$ of umbral moonshine belong to $\JJ^\top_{1,*}:=\bigoplus_m \JJ^\top_{1,m}$. 
In this work we will furnish a complete description of the space $\JJ^\top_{1,*}$ (cf. \S\ref{sec:class:proof}), and characterize it as a subspace of $\JJ^\opt_{1,*}:=\bigoplus_m \JJ^\opt_{1,m}$ (cf. \S\ref{sec:class:rpm}).
As part of the preparation for this we conclude this section with a non-vanishing result for the Fourier coefficient $C_\f(1,1)$ of a non-zero 
$\f\in \JJ_{1,m}^{\opt,\a}$. 

\begin{lem}\label{lem:class:proof-Cphi11nonzero}
Suppose that $\f$ is a non-zero 
element of $\mathbb{J}_{1,m}^{\opt,\a}$ for some $m\in \ZZp$ and $\a\in\widehat{O}_m$. Then $C_\f(1,1)\neq 0$.
\end{lem}
\begin{proof}
Let $\f$ be as in the statement of the lemma. Since $\f$ is optimal we have $C_\f(D,r)\neq 0$ only for $D\leq 1$. If $C_\f(1,a)\neq 0$ then we must have $a^2=1\xmod 4m$, so the possible non-zero positive discriminant Fourier coefficients $C_\f(1,a)$ are indexed  by $O_m$. 
We have $C_\f(1,a)=\a(a)C_\f(1,1)$ for $a\in O_m$, so either $C_\f(1,1)\neq 0$, or $\f$ has vanishing positive discriminant part, $\Dp(\f)=0$. In the latter case $\f=0$ according to Proposition \ref{prop_uniqueness1}. So if $\f\neq 0$ then $C_\f(1,1)\neq 0$, as we required to show.
\end{proof}

\subsection{Theta Type Jacobi Forms}\label{sec:class:theta}

Recall (\ref{eqn:jac:hol-Tskkm}) that $T^{\sk}_{k,m}$ is defined to be the subspace of $J^\sk_{k,m}$ spanned by skew-holomorphic Jacobi forms $\phi\in J^\sk_{k,m}$ such that $C_\phi(D,r)=0$ unless $D$ is a perfect square. 
We will describe the spaces $T^{\sk}_{k,m}$ explicitly in this section. For technical reasons it is convenient to introduce the (a priori) larger spaces ${T}^{\prime\sk}_{k,m}$, which allow $C_\f(D,r)\neq 0$ for those $D$ whose square-free parts are restricted to lie in some finite set (as opposed to the particular finite set $\{1\}$, cf. (\ref{eqn:jac:hol-Tskkm})). To be precise, write $\Supp^+ (C_\f(\,\cdot\,,r))$ for the set of $D\in \ZZ^+$ such that $C_\f(D,r)\neq 0$, and write $\XBox(\Supp^+ (C_\f(\,\cdot\,,r)))$ for the set of square-free parts of the positive integers in $\Supp^+ (C_\f(\,\cdot\,,r))$. Then we define
\begin{gather}\label{eqn:class:theta-Tprimeskkm}
	{T}^{\prime\sk}_{k,m}:=\left\{
		\phi\in J^{\sk}_{k,m}\mid 
		\#\XBox(\Supp^+ (C_\f(\,\cdot\,,r)))<\infty\,\forall r
		\right\}.
\end{gather}

Evidently ${T}^{\prime\sk}_{k,m}$ contains $T^{\sk}_{k,m}$. 
It turns out that ${T}^{\prime\sk}_{k,m}$ and $T^\sk_{k,m}$ actually coincide for all $k$ and $m$. 
We will show this 
by exploiting the relationship (\ref{eqn:jac:hol-jacmetinv}) between Jacobi forms and modular forms of half integral weight, 
so we introduce half integral weight counterparts $T^\prime_{k-\frac12}(\Gamma_1(4m))$ and $T^\prime_{k-\frac12}(\Gamma(4m))$ for $T^{\prime\sk}_{k,m}$, for integers $k$ and $m$, by setting
\begin{gather}
		\label{eqn:class:theta-TGamma14m}
	T^\prime_{k-\frac12}(\Gamma_1(4m)):=\left\{
		f\in M_{k-\frac12}(\Gamma_1(4m))\mid 
		\#\XBox(\Supp^+(c_f(\,\cdot\,))) <\infty
		\right\},\\
		\label{eqn:class:theta-TGamma4m}
	T^\prime_{k-\frac12}(\Gamma(4m)):=\left\{
		f\in M_{k-\frac12}(\Gamma(4m))\mid f(4m\tau)\in T^\prime_{k-\frac12}(\Gamma_1(16m^2))
		\right\}.
\end{gather}
To make the definition (\ref{eqn:class:theta-TGamma4m}) more transparent: note that it follows from the explicit description of $j(\gamma,\tau)$ in \cite{MR0332663} or \cite{MR0472707}, for example, that the assignment $f(\tau)\mapsto f(4m\tau)$ defines an injection from $M_{k-\frac12}(\Gamma(4m))$ to $M_{k-\frac12}(\Gamma_1(16m^2))$.

Recall the theta type skew-holomorphic Jacobi form $t_{k,m}\in T^\sk_{k,m}$, defined in \S\ref{sec:notn} for $k\in \{1,2\}$. Note that $T^\sk_{k,m}$ contains the image of $t_{k,m'}\cdot a'$ under $U_{d}$ whenever $m=m'd^2$, and $a'$ is an element of $O_{m'}$. 

\begin{prop}\label{prop:class:tht-tht}
If $k=1$ or $k=2$ then $T^{\prime\sk}_{k,m}$ is spanned by the $(t_{k,m'}\cdot a')|U_d$ 
such that $m=m'd^2$ and $a'\in O_{m'}$. 
If $k>2$ then $T^{\prime\sk}_{k,m}=\{0\}$. In particular, $T^{\prime\sk}_{k,m}=T^\sk_{k,m}$ for all $k$ and $m$. 
\end{prop}
\begin{proof}
We begin by noting
that if $\phi=\overline{g}^t\theta_m\in J^\sk_{k,m}$, with $g=(g_r)$, then $\phi\in T^{\prime\sk}_{k,m}$ if and only if $g_r\in T^\prime_{k-\frac12}(\Gamma(4m))$ for each $r$. We can put this another way if we assume that $T^\prime_{k-\frac12}(\Gamma(4m))$ is invariant for the action of $\mpt(\ZZ)$. Namely, if $T^\prime_{k-\frac12}(\Gamma(4m))$ is a sub $\mpt(\ZZ)$-module of $M_{k-\frac12}(\Gamma(4m))$ then the natural map (\ref{eqn:jac:hol-jacmetinv}) defines an isomorphism from $T^{\prime\sk}_{k,m}$ to the space of $\mpt(\ZZ)$-invariants in $\overline{T^\prime_{k-\frac12}(\Gamma(4m))}\otimes \Th_m$ upon restriction. This latter space is then, in turn, naturally isomorphic to the space of maps of $\mpt(\ZZ)$-modules from $T^\prime_{k-\frac12}(\Gamma(4m))$ to $\Th_m$, so we have
\begin{gather}\label{eqn:class:theta-Tprimeskhom}
	T^{\prime\sk}_{k,m} 
	\simeq
	\hom_{\mpt(\ZZ)}\left( T^\prime_{k-\frac12}(\Gamma(4m)), \Theta_m \right).
\end{gather}
Thus the problem of understanding $T^{\prime\sk}_{k,m}$ reduces to that of understanding the action of $\mpt(\ZZ)$ on $T^\prime_{k-\frac12}(\Gamma(4m))$.

Using the fact that $\th^{k-1}_{m,r}\in M_{k-\frac12}(\Gamma(4m))$ for $k\in \{1,2\}$ (cf. \S\S\ref{sec:notn},\ref{sec:jac:hol}), and observing that $\XBox(\Supp^+(c_f(\cdot)))=\{\XBox(t)\}$ for $f(\t)=\th^{k-1}_{m,r}(4mt\tau)$, we see that $\th^{k-1}_{m,r}(t\tau)\in T^\prime_{k-\frac12}(\Gamma(4mt))$ for $k\in \{1,2\}$ and $t\in \ZZp$. 
Actually, the thetanullwerte span the spaces $T^\prime_{k-\frac12}(\Gamma(4m))$. For the Serre--Stark theorem on modular forms of weight $1/2$ implies (cf. Corollary 3 of  \cite{MR0472707}) that $\sum_{m>0}M_{\frac12}(\Gamma_1(4m))$ is spanned by the rescaled thetanullwerte $\th^0_{m,r}(4mt\tau)$, 
and it is elementary to see from this that 
\begin{gather}\label{eqn:class:theta-TprimethetaGamma}
	\sum_{m>0}T^\prime_{k-\frac12}(\Gamma(4m))
	=
	\sum_{m>0}\Th_m^{k-1}
\end{gather}
for $k=1$. For $k>1$ we apply the result of Vign\'eras\footnote{See Theorem 2 of \cite{MR1658385} for a slight generalization of the Vign\'eras theorem, proved in an alternative way.} (cf. Th\'eor\`eme 3 of \cite{MR0485739}) that if $f$ belongs to $\sum_{m>0} M_{k-\frac12}(\Gamma_1(4m))$ and $k>1$, then either $k=2$ and $f(\t)$ is a linear combination of rescaled thetanullwerte $\th^1_{m,r}(4mt\tau)$, or the set of square-free parts of integers $n$ such that $c_f(n)$ does not vanish is infinite (i.e., $\#\XBox(\Supp^+(c_f(\,\cdot\,)))=\infty$). From this we conclude that (\ref{eqn:class:theta-TprimethetaGamma}) also holds for $k=2$, and $\sum_{m>0}T^\prime_{k-\frac12}(\Gamma(4m))=\{0\}$ for $k>2$. Thus, in particular, $T^{\prime\sk}_{k,m}=T^\sk_{k,m}=\{0\}$ for all $m$, when $k>2$.

Assume henceforth that $k\in \{1,2\}$. We will show that $T^\prime_{k-\frac12}(\Gamma(4m))$ is isomorphic to a direct sum of irreducible $\mpt(\ZZ)$-modules $\Th^{\new,\a'}_{m'}$ (cf. Theorem \ref{thm:jac:EZ-Thmdec}), by applying an argument directly similar to that of the proof of Satz 5.2 in \cite{Sko_Thesis}. This will justify the formula (\ref{eqn:class:theta-Tprimeskhom}), and ultimately enable us to identify the spaces $T^{\prime\sk}_{k,m}$ explicitly.

Define $\Th_m^{k-1,\new,\a}$ to be the image of $\Th_m^{\new,\a}$ under the operator ${\cal D}^{k-1}$ (cf. \S\ref{sec:notn}), for $\a\in \widehat{O}_m$,
and observe that ${\cal D}^{k-1}$ is a morphism of $\mpt(\ZZ)$-modules whose image is non-trivial if and only if $\a(-1)=(-1)^{k-1}$. Applying this observation and Theorem \ref{thm:jac:EZ-Thmdec} to the right hand side of (\ref{eqn:class:theta-TprimethetaGamma}) we see that 
\begin{gather}
	\sum_{m>0}T^\prime_{k-\frac12}(\Gamma(4m))
	=
	\bigoplus_{m>0}\bigoplus_{\substack{\a\in \widehat{O}_m\\\a(-1)=(-1)^{k-1}}}\Th_m^{k-1,\new,\a},
\end{gather}
where the sums are direct because the non-trivial summands $\Th_m^{k-1,\new,\a}$ are pairwise non-isomorphic, according to Theorem \ref{thm:jac:EZ-Thmdec}.
Since $\Th_m$, and hence also $\Th_m^{k-1,\new,\a}$, is trivial for the action of $\Gamma(4m)$, we actually have
\begin{gather}\label{eqn:class:theta-Tprime4moplus}
	T^\prime_{k-\frac12}(\Gamma(4m))
	=
	\bigoplus_{m'|m}\bigoplus_{\substack{\a'\in \widehat{O}_{m'}\\\a'(-1)=(-1)^{k-1}}}\Th_{m'}^{k-1,\new,\a'}.
\end{gather}
In particular, $T^\prime_{k-\frac12}(\Gamma(4m))$ is stable for the action of $\mpt(\ZZ)$, and the formula (\ref{eqn:class:theta-Tprimeskhom}) is valid.

Now we apply (\ref{eqn:class:theta-Tprime4moplus}), and Skoruppa's decomposition (\ref{eqn:jac:EZ-Thmdec}) to (\ref{eqn:class:theta-Tprimeskhom}), 
in order to deduce that
\begin{gather}\label{eqn:class:theta-Tprimeskkmbigsum}
	T^{\prime\sk}_{k,m}\simeq \bigoplus_{m'|m}\bigoplus_{\substack{\a'\in \widehat{O}_{m'}\\\a'(-1)=(-1)^{k-1}}}
		\bigoplus_{d^2|m}\bigoplus_{\substack{\a\in\widehat{O}_{m/d^2}}}
		\hom_{\mpt(\ZZ)}\left(\Th_{m'}^{\new,\a'},\Th_{m/d^2}^{\new,\a}\right).
\end{gather} 
By applying the second part of Theorem \ref{thm:jac:EZ-Thmdec} to (\ref{eqn:class:theta-Tprimeskkmbigsum}) we see that the only components $\Th_{m'}^{k-1,\new,\a'}$ of $T^\prime_{k-\frac12}(\Gamma(4m))$ that contribute to $T^{\prime\sk}_{k,m}$ are those where $\XBox(m/m')=1$ (i.e. $m/m'$ is a perfect square). So 
if we define 
\begin{gather}
T_{k-\frac12}(\Gamma(4m)):=\bigoplus_{d^2|m}\bigoplus_{\substack{\a\in\widehat{O}_{m/d^2}\\\a(-1)=(-1)^{k-1}}}
	\Th^{k-1,\new,\a}_{m/d^2},
\end{gather}
then $T^{\prime\sk}_{k,m}$ is naturally isomorphic to the space of $\mpt(\ZZ)$-module maps from $T_{k-\frac12}(\Gamma(4m))$ to $\Th_m$, 
and the dimension of $T^{\prime\sk}_{k,m}$ is the number of pairs $(d,\a)$, where $d^2|m$ and $\a\in \widehat{O}_{m/d^2}$ satisfies $\a(-1)=(-1)^{k-1}$. Finally we observe that the $(t_{k,m'}\cdot a')|U_d$ with $m=m'd^2$ and $a'\in O_{m'}$ span a subspace of $T^\sk_{k,m}$ with precisely this dimension. So $T^{\prime\sk}_{k,m}=T^{\sk}_{k,m}$ 
also for $k\in \{1,2\}$, and the proof of the proposition is complete.
\end{proof}

We have $M_{\frac12}(\Gamma(4m))=T^\prime_{\frac12}(\Gamma(4m))$ for all $m$ according to the Serre--Stark theorem, so the proof of Proposition \ref{prop:class:tht-tht} verifies that $T^\sk_{1,m}=J^\sk_{1,m}$, for all positive integers $m$. Since $T^\sk_{k,m}=\{0\}$ for $k>2$ it remains to understand the relationship between $T^\sk_{2,m}$ and $J^\sk_{2,m}$. 
\begin{lem}\label{lem:class:tht-TS}
We have $T^\sk_{2,m}<S^\sk_{2,m}$.
\end{lem}
\begin{proof}
Proposition \ref{prop:class:tht-tht} implies that the (complex conjugates of the) theta-coefficients of a $\f$ in $T^\sk_{2,m}$ are linear combinations of the thetanullwerte $\theta_{m,r}^1$. The $\theta_{m,r}^1$ for fixed $m$ span the image of the $\mpt(\ZZ)$-module map $\mathcal{D}^1:\Theta_m\to M_\frac32(\Gamma(4m))$, 
so cuspidality follows from the observation that $c_f(0)=0$ when $f=\theta_{m,r}^1$. 
\end{proof}
Recall from \S\ref{sec:jac:szlifts} that $P^\sk_{2,m}$ is defined to be the subspace of $S^\sk_{2,m}$ spanned by Hecke eigenforms that do not belong to $T^\sk_{2,m}$. We will show momentarily that $S^\sk_{2,m}$ is the orthogonal direct sum of $T^\sk_{2,m}$ and $P^\sk_{2,m}$, but first we require another lemma.
\begin{lem}\label{lem:class:tht-PtoS}
The image of $P^\sk_{2,m}$ under $\SZ$ is contained in $S_2(\Gamma_0(m))$.
\end{lem}
\begin{proof}
Let $\f\in P^\sk_{2,m}$, set $f=\SZ(\f)$, and suppose that $f$ does not belong to $S_2(\Gamma_0(m))$. We may assume that $\f$ is a Hecke eigenform. Then $f$ is a Hecke eigenform, and an Eisenstein series. If we write $\f=\overline{g}^t\th_m$ then it follows from the definition (\ref{eqn:jac:hol-phiTn}) of the Hecke operators on Jacobi forms that the functions $g_r(4m\t)\in S_\frac32(\Gamma_1(16m^2))$ are {\em associated} to $f$ in the sense of \S2 of \cite{MR660380}. Then Theorem 1 of \cite{MR660380} implies that $g_r\in T'_\frac32(\Gamma(4m))$ for all $r$, since $f$ is not a cusp form. Thus $\f\in T^\sk_{2,m}$ according to Proposition \ref{prop:class:tht-tht}, but this contradicts the hypothesis that $\f\in P^\sk_{2,m}$. 
\end{proof}

\begin{prop}\label{prop:class:tht-orthog}
We have $S^\sk_{2,m}=T^\sk_{2,m}\oplus P^\sk_{2,m}$, where the direct sum is orthogonal with respect to the Petersson inner product.
\end{prop}
\begin{proof}
It follows from the prescription (\ref{eqn:jac:hol-phiTn}) that the Hecke operators $T_n$ preserve the subspace $T^\sk_{2,m}<S^\sk_{2,m}$ of theta type forms. So $T^\sk_{2,m}$ and $P^\sk_{2,m}$ are both spanned by Hecke eigenforms. 
Two Hecke eigenforms are orthogonal with respect to the Petersson inner product unless they have the same eigenvalues, but Lemma \ref{lem:class:tht-PtoS} implies that the eigenvalues of any $\phi\in P^\sk_{2,m}$ are those of a cuspidal modular form of weight $2$ for $\Gamma_0(m)$. By contrast, an explicit calculation reveals that the image of 
$t_{2,m}\cdot a(n)$ under $\SZ$ for $n\in \Ex_m$ is proportional to the Eisenstein series $nE_2(n'\tau)-n'E_2(n\tau)\in M_2(\Gamma_0(m))$, where $n'=m/n$, and $E_2(\tau):=1-24\sum_{k>0}kq^k(1-q^k)^{-1}$. The Hecke eigenvalues of a cusp form cannot coincide with those of an Eisenstein series (cf. e.g. \cite{MR927162}), so the desired result follows from 
the explicit description of $T^\sk_{2,m}$ given in
Proposition \ref{prop:class:tht-tht}. 
\end{proof}

\subsection{Classification}
\label{sec:class:proof}

In this 
section we present the proof of our first main theorem, 
after establishing some preparatory results.  
The first 
of these is a verification of the claim (\ref{eqn:jac:ez-ALeqv}) that the map $\SZ$ intertwines the actions of $\Ex_m$ on $S^\sk_{k,m}$ and $M_{2k-2}(\Gamma_0(m))$. 
\begin{lem}\label{lem:class:proof-ALinv}
Let $m$ be a positive integer and let $n\in \Ex_m$. 
We have 
\begin{gather}\label{eqn:class:proof-ALinv}
\SZ(\phi|\W_m(n))
=
\SZ(\phi)|W_n
\end{gather}
in $M_{2k-2}(\Gamma_0(m))$, for all $\phi\in S^\sk_{k,m}$.
\end{lem}
\begin{proof}
Assume first that $k=2$, which is the case of most importance for this work. 
When restricted to $S^\sk_{k,m}$ the map $\SZ$ is a linear combination of the operators $\mathcal{S}_{D,r}$ with $D$ positive fundamental and $r\xmod 2m$ such that $D=r^2\xmod 4m$, so it suffices to check (\ref{eqn:class:proof-ALinv}) with such $\mathcal{S}_{D,r}$ in place of $\SZ$. To verify this we use the theta lift description (\ref{eqn:jac:szlifts-SDrTheta}) of $\mathcal{S}_{D,r}$.

To begin, note from Proposition 1 of \cite{MR909238} that $\chi_{D}$ is Atkin--Lehner invariant, 
in the sense that if $n\in \Ex_m$ and $w_n=\frac1{\sqrt{n}}\left(\begin{smallmatrix} an&b\\cm&dn\end{smallmatrix}\right)\in W_n$, 
then we have $\chi_{D}(Q|w_n)=\chi_{D}(Q)$ for $Q\in \mathcal{Q}(m,DD',rr')$, where $D$, $D'$, $r$ and $r'$ are as in (\ref{eqn:jac:szlifts-ThetaC}), and
\begin{gather}\label{eqn:class:proof-Qwe}
	(Q|w_n)(x,y):=\frac1nQ(anx+by,cmx+dny)
\end{gather}
for $w_n$ as above. Note also that, for $w_n$ as above, and for $D$ and $r$ such that $D=r^2\xmod 4m$, the map $Q\mapsto Q|w_n$ defines an isomorphism of sets $\mathcal{Q}(m,D,r)\xrightarrow{\sim}\mathcal{Q}(m,D,ra(n))$. 

To prove (\ref{eqn:class:proof-ALinv}) for $k=2$ we use the above mentioned properties of $Q\mapsto Q|w_n$ 
and the identities 
\begin{gather}\label{eqn:class:proof_Qwnids}
	\frac{Q(w_n\tau',1)}{\Im(w_n\tau')^2}\frac{n}{(cm\overline{\tau'}+dn)^2}=\frac{(Q|w_n)(\tau',1)}{\Im(\tau')^2},\quad
	\frac{\widehat{Q}(w_n\tau')}{\Im(w_n\tau')}=\frac{\widehat{Q|w_n}(\tau')}{\Im(\tau')},
\end{gather}
to
deduce that $nC_{D,r}(D',r';\tau,w_n\tau')(cm\overline{\tau'}+dn)^{-2}=C_{D,ra(n)}(D',r';\tau,\tau')$ for $w_n$ as in (\ref{eqn:class:proof-Qwe}), where $C_{D,r}$ is as defined in (\ref{eqn:jac:szlifts-ThetaC}).
From this 
we see that
\begin{gather}\label{eqn:class:proof-ThetaWe}
	\lab \phi(\cdot\,,\cdot),\Theta_{D,r}(\cdot\,,\cdot\,,\tau')\rab|W_n
	=	\lab \phi(\cdot\,,\cdot),\Theta_{D,ra(n)}(\cdot\,,\cdot\,,\tau')\rab,
\end{gather}
so $(\mathcal{S}_{D,r}\phi)|W_n=\mathcal{S}_{D,ra(n)}\phi$ according to (\ref{eqn:jac:szlifts-SDrTheta}). Comparing (\ref{eqn:jac:ez-Wmndota}) with the definition (\ref{eqn:jac:szlifts-SDr}) of $\mathcal{S}_{D,r}$ we have $\mathcal{S}_{D,ra(n)}\phi=\mathcal{S}_{D,r}(\phi|\W_m(n))$. Thus we have verified (\ref{eqn:class:proof-ALinv}) for $k=2$.

The case that $k\geq 3$ can be handled similarly, except that an alternative kernel function 
\begin{gather}
	\begin{split}
	\Theta_{D,r}(\tau,z,\tau')&:=\sum_{\substack{D',r'\in \ZZ\\D'=(r')^2\xmod 4m}}C_{D,r}(D',r';\tau,\tau')\bar{q}^{D'/4m}q^{(r')^2/4m}y^{r'},
	\label{eqn:class:proof-Theta}
	\\
	C_{D,r}(D',r';\tau,\tau')&:=|DD'|^{k-3/2}
	\sum_{Q\in\mathcal{Q}(m,DD',rr')}\chi_D(Q)Q(\tau',1)^{1-k},
	\end{split}
\end{gather}
is more convenient. In fact, this
is essentially the kernel used to define the $\mathcal{S}_{D,r}$ for $D<0$ in \cite{MR909238}. The precise analogue of (\ref{eqn:jac:szlifts-SDrTheta}) that relates this $\Theta_{D,r}$ to $\mathcal{S}_{D,r}\phi$ is 
\begin{gather}
(\mathcal{S}_{D,r}\phi)(\tau') = c^+_{k,m}\lab \phi(\cdot,\cdot),\Theta_{D,r}(\cdot,\cdot,-\overline{\tau'})\rab,
\end{gather}
for a certain (non-zero) constant $c^+_{k,m}$,
and is established in Proposition 1 of \cite{MR1072974}. The proof then proceeds as before, using (\ref{eqn:class:proof_Qwnids}) and the Atkin--Lehner invariance of $\chi_D$.
\end{proof}

The next lemma relates the existence of optimal mock Jacobi theta functions to a vanishing condition on special values of $L$-functions for cusp forms of weight $2$. 
In order to formulate it, 
observe that if $m'|m$ then any $\a'\in \widehat{O}_{m'}$ is naturally an element of $\widehat{O}_m$. Indeed, we have the natural map $\Ex_m\to \Ex_{m'}$ given by $n\mapsto (n,m')$. So given $\a'\in O_{m'}$ it is natural to set 
\begin{gather}\label{eqn:class:proof-alphaprimelift}
\a'(a(n)):=\a'(a'((n,m')))
\end{gather} 
for $n\in \Ex_m$, where $a$ and $a'$ denote the natural isomorphisms $\Ex_m\to O_m$ and $\Ex_{m'}\to O_{m'}$, respectively (cf. \S\S\ref{sec:notn},\ref{sec:jac:ez}).

\begin{lem}\label{vanishing_period_maxsym}
Suppose that $\JJ^{\top,\a}_{1,m}\neq\{0\}$ for some positive integer $m$, and some $\a\in \widehat{O}_m$. Then $L(f',1)=0$ for every newform $f'$ in $S_2(\Gamma_0(m'))^{\a'}$, whenever $m'$ divides $m$, and $\a'\in \widehat{O}_{m'}$ coincides with $\a$ as an element of $\widehat{O}_m$.
\end{lem}

\begin{proof}
Let $m$ and $\a$ be as in the statement of the lemma, and suppose that $\phi$ is a non-zero element of $\JJ^{\top,\a}_{1,m}$. 
Note that, since $\phi$ is optimal, 
for 
$\varphi\in S^\sk_{2,m}$ we have
\begin{gather}\label{eqn:class-lamphioptcsp}
	\{\f,\varphi\}=
	\sum_{\substack{r\xmod 2m\\r^2=1\xmod 4m}} C_\phi(1,r)\overline{C_\varphi(1,r)}
\end{gather}
according to Proposition \ref{prop:jac:hol-lampol}. In particular, the sum in (\ref{eqn:class-lamphioptcsp}) is indexed by $O_m$.
Since $\phi$ has weight $1$ we have $C_\phi(D,r)=-C_\phi(D,-r)$ (cf. (\ref{eqn:jac:ez-Foucoeffsym})), and similarly with $\varphi$ in place of $\phi$ (cf. the text preceding (\ref{eqn:jac:ez-Foucoeffsym})), so the contributions of $r$ and $-r$ to (\ref{eqn:class-lamphioptcsp}) are the same. 
The elements of $\ker(\alpha)$ run over a set of representatives for the orbits of $\{\pm 1\}$ on $O_m$, so we can write $\{\phi,\varphi\}=2\sum_{a\in\ker(\alpha)}C_\phi(1,a)\overline{C_\varphi(1,a)}$. This in turn becomes
\begin{gather}\label{eqn:class-lamphioptcspred}
	\{\f,\varphi\}=
	2C_\phi(1,1)\sum_{a\in \ker(\alpha)} \overline{C_\varphi(1,a)}
\end{gather}
by virtue of the fact that $\phi\cdot a= \phi$ for $a\in \ker(\a)$, since $\phi\in \JJ^{\top,\a}_{1,m}$. 
Note that $C_\f(1,1)\neq 0$ according to Lemma \ref{lem:class:proof-Cphi11nonzero}.

Now suppose that $m'|m$ and $\a'\in \widehat{O}_{m'}$ coincides with $\a$ when lifted in the natural way to $\widehat{O}_m$. Explicitly, this is the requirement that 
\begin{gather}\label{eqn:class:proof-alphaprimeisalpha}
\a'(a'((n,m')))=\a(a(n))
\end{gather} 
for all $n\in \Ex_m$ (cf. (\ref{eqn:class:proof-alphaprimelift})). Let $f'\in S_2(\Gamma_0(m'))^{\a'}$ be a newform. 
Then Theorem 
\ref{thm:jac:szlifts-sz} attaches a skew-holomorphic Jacobi form $\varphi'\in P^\sk_{2,m'}$ to $f'$ having the same eigenvalues as $f'$ under the Hecke operators $T_p$ (cf. (\ref{eqn:jac:hol-phiTn})) for $(p,m')=1$. 
Applying Lemma \ref{lem:class:proof-ALinv} we see that $\varphi'$ and $f'$ have the same eigenvalues under the action of $\Ex_{m'}$. That is, $\varphi'|\W_{m'}(n)=\alpha'(a'(n'))\varphi'$ for $n'\in \Ex_{m'}$, so $\varphi'\in P^{\sk,\a'}_{2,m'}$. 

We now consider (\ref{eqn:class-lamphioptcspred}) for $\varphi=\varphi'|V_{m/m'}$ where $V_{\ell}$ is the Hecke-like operator defined in (\ref{eqn:jac:szlifts-Vell}). Upon verifying that 
$V_{m/m'}\W_{m}(n)=\W_{m'}((n,m'))V_{m/m'}$ 
as operators on $J_{2,m'}^\sk$ for all $n\in \Ex_m$, we conclude from (\ref{eqn:class:proof-alphaprimeisalpha}) that $\varphi\in S_{2,m}^{\sk,\a}$. 
Thus (\ref{eqn:class-lamphioptcspred}) is reduced to 
\begin{gather}\label{eqn:class-lamphioptcspredred}
	\{\f,\varphi\}=
\#\Ex_mC_\phi(1,1)\overline{C_\varphi(1,1)}.
\end{gather}
Using (\ref{eqn:jac:szlifts-Belld}) we obtain that the image of $\varphi$ under $\SZ$ is cuspidal, so actually $\varphi\in P_{2,m}^{\sk,\a}$.
But the shadow of $\phi$ belongs to $T^\sk_{2,m}$ by assumption, so is orthogonal to $P^\sk_{2,m}$ according to Proposition \ref{prop:class:tht-orthog}. So $\{\f,\varphi\}=\lab\xi(\f),\varphi\rab=0$. 
Lemma \ref{lem:class:proof-Cphi11nonzero} ensures that $C_\phi(1,1)$ is non-zero if $\phi$ is non-zero, so it must be the case that $C_\varphi(1,1)=0$. Using (\ref{eqn:jac:szlifts-Vell}) we verify that $C_{\varphi'}(1,1)=C_{\varphi}(1,1)$, so $C_{\varphi'}(1,1)=0$. We now apply Theorem \ref{thm:jac:szlifts-Lfns} to $\varphi'$ and $f'$, with $D=1$, and deduce that $L(f',1)=0$, as we required to show.
\end{proof}  

Our final lemma in this section is a simple non-vanishing result for central critical values of modular $L$-functions, which we use together with Lemma \ref{vanishing_period_maxsym} to establish a genus zero property for optimal mock Jacobi theta functions in Proposition \ref{prop:class:proof-genuszero}.
To prepare for this, recall (cf. \S\ref{sec:notn}) that for $K<O_m$ we write $\Gamma_0(m)+K$ for the subgroup of $\SL_2(\RR)$ 
generated by $\Gamma_0(m)$ and the $W_n$ such that $a(n)\in K$. 
The {\em genus} of the group $\Gamma=\Gamma_0(m)+K$ is, by definition, the genus of the Riemann surface $X_\Gamma$. 
The eigenspace $S_2(\Gamma_0(m))^\a$ is naturally isomorphic to the space of holomorphic $1$-forms on $X_\Gamma$, when $\Gamma=\Gamma_0(m)+\ker(\a)$. So the genus of $\Gamma_0(m)+\ker(\a)$ is the dimension of $S_2(\Gamma_0(m))^\a$.

\begin{lem}\label{lem:class:proof-S2maLf1}
Suppose that $m$ is a positive integer and $\a\in\widehat{O}_m$ satisfies $\a(-1)=-1$. If 
$\Gamma_0(m)+\ker(\a)$ is not genus zero then $L(f,1)\neq 0$ for some $f\in S_2(\Gamma_0(m))^\a$.
\end{lem}
\begin{proof}
Set $\Gamma=\Gamma_0(m)+\ker(\a)$ and suppose that $L(f,1)=0$ for every $f\in S_2(\Gamma_0(m))^\a$. 
By our hypothesis on $\a$ the Fricke involution $W_m$ does not belong to $\Gamma$, so the image of the divisor $E=(\infty)-(0)$ under the natural map $\widehat{\HH}\to X_\G$ does not vanish. 
We have
\begin{gather}\label{eqn:class:proof-Lf1}
L(f,1)=2\pi \int_0^{\infty}f(iy){\rm d}y,
\end{gather} 
so the Abel theorem implies that $E$ maps to a principal divisor of $X_\Gamma$. In other words, there is a meromorphic function on $X_\Gamma$ with a single simple pole. Such a function defines a holomorphic isomorphism from $X_\Gamma$ to the Riemann sphere, so $\Gamma$ is genus zero,
and the lemma is proved.
\end{proof}
\begin{rmk}
Kolyvagin--Logachev showed \cite{MR1036843} that if $f\in S_2(\Gamma_0(m))$ is a newform with $L(f,1)\neq 0$ then the abelian variety associated to $f$ by Shimura's construction \cite{Shi_IntThyAutFns,MR0318162} has finite Mordell--Weil group.
One may check (cf. \cite{Fer_Genus0prob}) that if $S_2(\Gamma_0(m))$ is not zero then there is an $\a\in \widehat{O}_m$ such that $\Gamma_0(m)+\ker(\a)$ is non-Fricke and not genus zero. 
So Lemma \ref{lem:class:proof-S2maLf1} shows that if $m$ is a positive integer other than
\begin{gather}
	1,\;
	2,\;
	3,\;
	4,\;
	5,\;
	6,\;
	7,\;
	8,\;
	9,\;
	10,\;
	12,\;
	13,\;
	16,\;
	18,\;
	25,
\end{gather}
then there exists a modular abelian variety with conductor $m$ that has finite Mordell--Weil group. In order words, if an integer
can occur as the conductor of a modular abelian variety, then it does occur as the conductor of a modular abelian variety with finitely many rational points.
\end{rmk}

We now use Lemmas \ref{vanishing_period_maxsym} and \ref{lem:class:proof-S2maLf1} to establish a genus zero property for optimal mock Jacobi theta functions of weight 1. To formulate it we say that $\Gamma=\Gamma_0(m)+K$ is {\em Fricke} or {\em non-Fricke} according as $W_m\subset \Gamma$ or not. 

\begin{prop}\label{prop:class:proof-genuszero}
Let $m$ be a positive integer and $\a\in\widehat{O}_m$. If $\mathbb{J}_{1,m}^{\top,\a}\neq \{0\}$ then $\Gamma_0(m)+\ker(\a)$ is non-Fricke and genus zero.
\end{prop}
\begin{proof}
Suppose that $m$ and $\a$ are as in the statement of the theorem and $\phi\in \mathbb{J}_{1,m}^{\top,\a}$ is not zero.
We have observed that $\JJ^{\top,\a}_{1,m}$ vanishes unless $\a(-1)=-1$ (cf. (\ref{eqn:class:defn-JJtopkmOmdec})), so the existence of $\phi$ implies that $\Gamma_0(m)+\ker(\a)$ is non-Fricke. 
We will use Lemma \ref{vanishing_period_maxsym} to show 
that $L(f,1)=0$ for every $f\in S_2(\Gamma_0(m))^\a$. Then Lemma \ref{lem:class:proof-S2maLf1} will imply that $\Gamma_0(m)+\ker(\a)$ is genus zero, and complete the proof.

To setup for the application of Lemma \ref{vanishing_period_maxsym}, we first recall the natural decomposition
\begin{gather}\label{eqn:class:proof-ALdec}
	S_2(\Gamma_0(m))= \bigoplus_{m'|m}\bigoplus_{t|\frac{m}{m'}}\iota_{m',t}^mS_2(\Gamma_0(m'))^\new
\end{gather}
arising from the theory \cite{MR0268123} of Atkin--Lehner, 
where $S_2(\Gamma_0(m'))^\new$ is the space of newforms of level $m'$ in $S_2(\Gamma_0(m'))$, and $\iota_{m',t}^m$ is the {\em degeneracy map} $S_2(\Gamma_0(m'))\to S_2(\Gamma_0(m))$ defined by $(\iota_{m',t}^m f)(\tau):=f(t\tau)$. Applying the projection operator 
\begin{gather}\label{eqn:class:proof-Palpha}
P^\alpha f
	:=\frac{1}{\#\Ex_m}\sum_{n\in \Ex_m}\a(a(n))f|W_n
\end{gather}
to (\ref{eqn:class:proof-ALdec}) we obtain
\begin{gather}\label{eqn:class:proof-ALalphadec}
	S_2(\Gamma_0(m))^{\a}= 
	\bigoplus_{m'|m}
	\bigoplus_{\a'\in \widehat{O}_{m'}}
	\bigoplus_{t|\frac{m}{m'}}
	P^\alpha\iota_{m',t}^mS_2(\Gamma_0(m'))^{\new,\a'},
\end{gather}
where $S_2(\Gamma_0(m'))^{\new,\a'}$ is the intersection of $S_2(\Gamma_0(m'))^{\a'}$ and $S_2(\Gamma_0(m'))^\new$. 
So to see that $L(f,1)=0$ for all $f\in S_2(\Gamma_0(m))^\a$, 
it suffices to verify the case that 
$f=tP^\a\iota_{m',t}^mf'$, for $f'$ a newform in $S_2(\Gamma_0(m'))^{\a'}$, 
for some $m'|m$ and $t|\tfrac{m}{m'}$, and some $\a'\in \widehat{O}_{m'}$. 

A routine calculation confirms that $t(\iota_{m',t}^m f')|W_n=t'\iota_{m',t'}^m(f'|W_{n'})$ for $n\in \Ex_m$, where $n'=(n,m')$ and $t'=t\ast (n/n')$ (and $\ast$ is as defined in \S\ref{sec:notn}). So with $f=tP^\a\iota_{m',t}^mf'$ and $f'$ as above we have
\begin{gather}
f=
\frac{1}{\#\Ex_m}\sum_{n\in \Ex_m}\a(a(n))\a'(a'(n')) t'\iota_{m',t'}^mf',
\end{gather}
where in each summand $n'=(n,m')$ and $t'=t\ast (n/n')$. We have $L(t'\iota_{m',t'}^mf',1)=L(f',1)$ (cf. (\ref{eqn:class:proof-Lf1})), so 
\begin{gather}\label{eqn:class:proof-LfalphasLfprime}
	L(f,1)=\left(\frac1{\#\Ex_m} \sum_{n\in \Ex_m}\a(a(n))\a'(a'(n'))\right)L(f',1),
\end{gather}
where again, $n'=(n,m')$ in each summand. Now the first factor on the right-hand side of (\ref{eqn:class:proof-LfalphasLfprime}) is just the inner product of $\alpha$ and $\alpha'$ in the ring of class functions on $O_m$, where $\a'\in \widehat{O}_{m'}$ is regarded as a character on $O_m$ in the natural way (cf. (\ref{eqn:class:proof-alphaprimelift})). So $L(f,1)=0$ unless $\a$ and $\a'$ coincide as elements of $\widehat{O}_m$. But if $\a$ and $\a'$ are the same element of $\widehat{O}_m$ then $L(f,1)=0$ according to (\ref{eqn:class:proof-LfalphasLfprime}) and Lemma \ref{vanishing_period_maxsym}. So $L(f,1)=0$ for every $f\in S_2(\Gamma_0(m))^\a$. We apply Lemma \ref{lem:class:proof-S2maLf1} to obtain that $\Gamma_0(m)+\ker(\a)$ is genus zero, and the proof of the proposition is complete.
\end{proof}

We now prove our first main result. 
\begin{proof}[Proof of Theorem \ref{thm:intro-firstmainthm}]
Suppose $m\in \ZZ^+$ and $\a\in \widehat{O}_m$. We first observe that the dimension of $\mathbb{J}_{1,m}^{\top,\a}$ is bounded above by $1$. Actually, the dimension of $\mathbb{J}_{1,m}^{\opt,\a}$ is bounded above by $1$, for given $\f,\f'\in\mathbb{J}_{1,m}^{\opt,\a}$ we may consider $\f''=C_{\f'}(1,1)\f-C_\f(1,1)\f'$. If $\f$ and $\f'$ are non-zero then $C_\f(1,1)$ and $C_{\f'}(1,1)$ are non-zero according to Lemma \ref{lem:class:proof-Cphi11nonzero}, but $C_{\f''}(1,1)=0$ by construction, so $\f''=0$ by another application of Lemma \ref{lem:class:proof-Cphi11nonzero}. Thus non-zero elements of $\mathbb{J}_{1,m}^{\opt,\a}$ are collinear. In particular, non-zero elements of $\mathbb{J}^{\top,\a}_{1,m}$ are collinear.

If $\dim\mathbb{J}_{1,m}^{\top,\a}\neq 0$ then 
by Proposition \ref{prop:class:proof-genuszero} we have that $\a(-1)=-1$ and the genus of $\Gamma_0(m)+\ker(\a)$ is zero. So it remains to demonstrate that $\JJ_{1,m}^{\top,\a}$ is non-zero for such $\a$. One way to do this is to consider suitable regularized Poincar\'e series or Rademacher sums, such as are discussed in \S B.4 of \cite{umrec}, for example. This produces a non-zero element $\f\in \JJ^{\opt,\a}_{1,m}$. By our assumptions on $\a$ we have $P^{\sk,\a}_{2,m}=\{0\}$ (cf. Lemmas \ref{lem:class:tht-PtoS} and \ref{lem:class:proof-ALinv}), so the shadow of $\f$ is a skew-holomorphic Jacobi form of theta type, and $\f\in \JJ^{\top,\a}_{1,m}$. So $\JJ_{1,m}^{\top,\a}$ is not zero.

Later we will require a more concrete understanding of the optimal mock Jacobi theta functions, so we also offer the following, more concrete verification that $\JJ_{1,m}^{\top,\a}$ is non-vanishing when $\a(-1)=-1$ and $\Gamma_0(m)+\ker(\a)$ has genus zero.
Inspecting Table 3.1 of \cite{Fer_Genus0prob} we see that the genus zero groups of the form $\Gamma_0(m)+\ker(\a)$ with $\a(-1)=-1$ are
\begin{gather}
\begin{split}\label{eqn:class:proof-gzgps-mum}
	2,\;
	3,\;
	&4,\;
	5,\;
	7,\;
	8,\;
	9,\;
	13,\;
	14+7,\;
	16,\;
	25,\;
	22+11,\;
	46+23,\;
	\\
	6&+3,\;
	10+5,\;
	12+4,\;
	18+9,\;
	30+6,10,15,\;
\end{split}\\
\begin{split}\label{eqn:class:proof-gzgps-emum}
	&6+2,\;
	10+2,\;
	12+3,\;
	18+2,\;
	30+3,5,15,\;
	\\
	15+5,\;
	20+&4,\;
	21+3,\;
	24+8,\;
	28+7,\;
	33+11,\;
	36+4,\;
	60+12,15,20,\;
	\\
	&42+6,14,21,\;
	70+10,14,35,\;
	78+6,26,39,
\end{split}
\end{gather}
where $m+n,n',\ldots$ is a shorthand for $\Gamma_0(m)+\{1,a(n),a(n'),\ldots\}$. In the first two lines (\ref{eqn:class:proof-gzgps-mum}) we list the 18 of these genus zero groups that appear in umbral moonshine \cite{UM,MUM}. To each such group $\Gamma$ is associated a Niemeier root system $X$ (cf. Table \ref{tab:intro-pm}). 
If $H^X=(H^X_r)$ is the mock modular form attached to $X$ in \cite{MUM} (cf. also \S B.3 of \cite{umrec}) then $\phi^X:=\sum_{r\xmod 2m}H^X_r\th_{m,r}$ is a non-zero element of $\JJ_{1,m}^{\top,\a}$. 

The 16 groups in (\ref{eqn:class:proof-gzgps-emum}) represent distinguished mock Jacobi forms that do not have Niemeier root systems attached. However, most of them can be constructed in umbral terms. For example, if $K$ is a subgroup of $\ker(\a)$ such that $\Gamma_0(m)+K$ appears in umbral moonshine, then the associated mock Jacobi form $\f^X$ is mapped to a non-zero element of $\JJ_{1,m}^{\top,\a}$ by the projection operator
\begin{gather}\label{eqn:class:proof-Palphaphi}
	P^\a\f:=\frac{1}{\#O_m}\sum_{a\in O_m}\a(a)\f\cdot a
\end{gather} 
(cf. (\ref{eqn:class:proof-Palpha})). In this way (cf. Table \ref{tab:intro-pm}) we see that $\JJ_{1,m}^{\top,\a}\neq \{0\}$ for the five groups in the first line of (\ref{eqn:class:proof-gzgps-emum}). The eight groups in the second line of (\ref{eqn:class:proof-gzgps-emum}) are handled in an analogous way, but using umbral McKay--Thompson series $H^X_{g}$ for $g$ certain non-trivial elements $g\in G^X$. For example, optimal theta type mock Jacobi forms corresponding to the groups $15+5$, $24+8$, $33+11$, and $36+4$ can be constructed from the mock modular forms $H^X_g$, where $X=A_2^{12}$ and $g$ is in the class $5A$, $8C$, $11A$ and $6C$ of $G^X\simeq 2.M_{12}$, respectively. The details of these constructions are given in \S\ref{sec:desc:mjt}. 

For the three groups in the last line of (\ref{eqn:class:proof-gzgps-emum}) we refer to \S9.5 of \cite{Dabholkar:2012nd}, where corresponding optimal mock Jacobi forms $\mathcal{Q}_{42}$, $\mathcal{Q}_{70}$ and $\mathcal{Q}_{78}$ have appeared already. 
One can check that these mock Jacobi forms lie in the required eigenspaces for the corresponding $O_m$. This inspection completes the proof of the theorem.
\end{proof}

We conclude this section with two corollaries to Theorem \ref{thm:intro-firstmainthm}. 
To formulate the first, let $\a\in \widehat{O}_m$ such that $\Gamma_0(m)+\ker(\a)$ is non-Fricke and genus zero, and let $\ell=m+n,n',\dots$ be the corresponding symbol in (\ref{eqn:class:proof-gzgps-mum}) or (\ref{eqn:class:proof-gzgps-emum}). Lemma \ref{lem:class:proof-Cphi11nonzero} guarantees that if $\f\in \JJ^{\top,\a}_{1,m}$ and $\f\neq 0$ then $C_\f(1,1)\neq 0$. So by Theorem \ref{thm:intro-firstmainthm} there is a unique $\f^{(\ell)}\in \JJ^{\top,\a}_{1,m}$ such that $C_{\f^{(\ell)}}(1,1)=-2$. By another application of Theorem \ref{thm:intro-firstmainthm}, the $\f^{(\ell)}$ obtained in this way furnish a basis for the space of optimal weight 1 mock Jacobi theta functions. 
\begin{cor}\label{cor:class:proof-dimJtop}
The $\f^{(\ell)}$ for $\ell$ in (\ref{eqn:class:proof-gzgps-mum}) and (\ref{eqn:class:proof-gzgps-emum}) furnish a basis for $\JJ^\top_{1,*}$. In particular, $\dim \JJ_{1,*}^{\top}=34$.
\end{cor}

According to Corollary \ref{cor:class:proof-dimJtop} 
we can canonically (up to scale) attach an optimal mock Jacobi form to each genus zero group of the form $\Gamma_0(m)+\ker(\a)$. What if $K$ is a proper subgroup of $\ker(\a)$ such that $\Gamma_0(m)+K$ also has genus zero? The prescription of \cite{MUM} (i.e. umbral moonshine) attaches optimal mock Jacobi forms to such groups. Our second corollary to Theorem \ref{thm:intro-firstmainthm} identifies these mock Jacobi forms in a general and uniform way. To formulate it, we 
define
\begin{gather}\label{eqn:class:proof-JJ1mtopK}
	\JJ_{1,m}^{\top|K}:=\left\{\f\in \JJ_{1,m}^\top\mid \f\cdot a=\f\Leftarrow a\in K\text{ and }C_\f(1,r)\neq 0\Rightarrow \pm r\in K\right\}
\end{gather}
for $K<O_m$. 
\begin{cor}\label{cor:class:proof-genK}
Let $K<O_m$. 
If $\Gamma_0(m)+K$ is non-Fricke and genus zero then we have $\dim\JJ_{1,m}^{\top|K}=1$. Otherwise $\dim \JJ_{1,m}^{\top|K}=0$.
\end{cor}
\begin{proof}
Let $K<O_m$ and set $\Gamma=\Gamma_0(m)+K$. Recall that $\Gamma$ is Fricke if and only if $-1\in K$. We have $C_\f(D,r)=-C_\f(D,-r)$ for every $\f\in \JJ^\wk_{1,m}$ according to (\ref{eqn:jac:ez-Foucoeffsym}), so if $-1\in K$ then $\JJ^{\top|K}_{1,m}=\{0\}$. If $-1\notin K$ and $K$ is a maximal subgroup of $O_m$ then $K=\ker(\a)$ for some $\a\in\widehat{O}_m$ with $\a(-1)=-1$. Then $\JJ^{\top|K}_{1,m}=\JJ^{\top,\a}_{1,m}$ by definition, and the claimed result follows from Theorem \ref{thm:intro-firstmainthm}. 

So assume that $K$ is a nonmaximal subgroup of $O_m$ that does not contain $-1$. Then the set
\begin{gather}
\widehat{K}:=\left\{\a\in \widehat{O}_m\mid \a(-1)=-1\text{ and }K<\ker(\a)\right\}
\end{gather} 
has cardinality at least $2$. If $\f$ is a nonzero element of $\JJ_{1,m}^{\top|K}$ then $P^\a\f$ 
(cf. (\ref{eqn:class:proof-Palphaphi}))
is a nonzero element of $\JJ_{1,m}^{\top,\a}$ for each $\a\in \widehat{K}$, so $\dim \JJ_{1,m}^{\top|K}\neq 0$ implies that $\Gamma_0(m)+\ker(\a)$ has genus zero for each $\a\in\widehat{K}$, according to Theorem \ref{thm:intro-firstmainthm}. Inspecting (\ref{eqn:class:proof-gzgps-mum}) and (\ref{eqn:class:proof-gzgps-emum}) we conclude that if $\dim \JJ_{1,m}^{\top|K}\neq 0$ then either $m\in \{6,10,12,18\}$ and $K=\{1\}$, or $m=30$ and $K=\{1,29\}$ (since $a(15)=29$ when $m=30$). In every case $\Gamma$ has genus zero, so it remains to show that $\dim \JJ_{1,m}^{\top|K}=1$ for such $K$.

So suppose that $K$ is as above. 
Then the cardinality of $\widehat{K}$ is exactly $2$. Say $\widehat{K}=\{\a_1,\a_2\}$. 
Invoking Theorem \ref{thm:intro-firstmainthm}, let 
$\f_i$ be the unique element of $\JJ_{1,m}^{\top,\a_i}$ such that 
$C_{\f_i}(1,1)=-2$, and set $\f=\frac{1}{2}\left(\f_1+\f_2\right)$. Then $C_{\f}(1,r)=\mp 2$ when $\pm r\in K$, and $C_{\f}(1,r)=0$ when $\pm r\notin K$. In particular, $\f$ is a non-zero element of $\JJ_{1,m}^{\top|K}$, so $\dim \JJ_{1,m}^{\top|K}\geq 1$. But any other element of $\JJ_{1,m}^{\top|K}$ is proportional to $\f$ by an application of Proposition \ref{prop_uniqueness1} (as in the beginning of the proof of Theorem \ref{thm:intro-firstmainthm}). We conclude that $\dim \JJ_{1,m}^{\top|K}=1$, and the proof of the corollary is complete.
\end{proof}

From the proof of Corollary \ref{cor:class:proof-genK} we see that, in addition to the 34 groups of (\ref{eqn:class:proof-gzgps-mum}) and (\ref{eqn:class:proof-gzgps-emum}), there are 5 further groups 
\begin{gather}\label{eqn:class:proof-gzgps-nonmax}
	6,\; 10,\; 12,\; 18,\; 30+15,
\end{gather}
such that $\JJ_{1,m}^{\top|K}$ is non-vanishing. 
Following \cite{UM,MUM} we refer to the 39 symbols of (\ref{eqn:class:proof-gzgps-mum}), (\ref{eqn:class:proof-gzgps-emum}) and (\ref{eqn:class:proof-gzgps-nonmax})---these are just the symbols that appear in Table \ref{tab:intro-pm}---as {\em lambencies}, and write $\gt{L}_1$ for the set they comprise. Corollary \ref{cor:class:proof-genK} ensures that, given a lambency $\ell=m+n,n',\dots$ in $\gt{L}_1$, we may write $\f^{(\ell)}$ for the unique element\footnote{Note that the symbol $\phi^{(\ell)}$ has a different meaning in \cite{MUM}, where it denotes a certain weak Jacobi form of weight $0$ and index $m-1$, and is only defined when $K=\{1\}$.} of $\JJ_{1,m}^{\top|K}$ such that $C_{\f^{(\ell)}}(1,1)=-2$, where $K=\{1,a(n),a(n'),\dots\}$. 

\subsection{Rationality} 
\label{sec:class:rat}

In this section we prove our second main result, and thereby establish a classification of the optimal mock Jacobi forms with rational coefficients. 

\begin{proof}[Proof of Theorem \ref{thm:intro-rat}]
Suppose that $\f\in\JJ^{\opt,\a}_{1,m}$ and $\f\neq 0$. 
After rescaling if necessary we may assume that $C_\f(1,1)=-2$ (cf. Lemma \ref{lem:class:proof-Cphi11nonzero}). Then the first statement is a corollary of our proof of Theorem \ref{thm:intro-firstmainthm}, for it tells us that if $\xi(\f)\in T^\sk_{2,m}$ then $\f$ is one of the mock Jacobi forms described explicitly in \S\ref{sec:desc:mjt}. By inspection of these descriptions, all the Fourier coefficients of $\f$ are integers. 

For the second statement we employ methods developed by Bruinier--Ono in \cite{MR2726107}. Particularly, Theorem 5.5 of loc. cit. implies that $C_\f(Dn^2,rn)$ is transcendental for some $n\in \ZZ$ if no multiple of $Z_{D,r}(\f)$ is the divisor of a rational function on $X_0(m)$. So assuming $\xi(\f)\notin T^\sk_{2,m}$ we require to find a negative fundamental discriminant $D$, admitting an $r\xmod 2m$ with $D=r^2\xmod 4m$, such that $Z_{D,r}(\f)$ does not vanish in 
$J_0(m)$.

So suppose that $\f\in\JJ^{\opt,\a}_{1,m}$ 
and 
$C_\f(1,1)=-2$ but $\xi(\f)\notin T^\sk_{2,m}$. 
Then there is a $\varphi\in P^\sk_{2,m}$ such that $\{\f,\varphi\}=\lab \xi(\f),\varphi\rab\neq 0$ (cf. Proposition \ref{prop:class:tht-orthog}), and we may assume that $\varphi\in P^{\sk,\a}_{2,m}$. 
As in the proof of Lemma \ref{vanishing_period_maxsym} we have $\{\f,\varphi\}=\#\Ex_m C_\f(1,1)\overline{C_\varphi(1,1)}$ (cf. (\ref{eqn:class-lamphioptcspredred})), so $C_\varphi(1,1)\neq 0$. On the strength of this we may assume that $\varphi=\varphi'|V_{m/m'}$ 
(cf. (\ref{eqn:jac:szlifts-Vell})) for some newform $\varphi'\in P^{\sk,\a'}_{2,m'}$, for some divisor $m'|m$, where $\a'\in\widehat{O}_{m'}$ coincides with $\a$ when lifted to $\widehat{O}_m$ in the natural way (cf. (\ref{eqn:class:proof-alphaprimelift})). This is because $P^\sk_{2,m}$ is spanned by the images of newforms under the Hecke-like operators $U_d\circ V_\ell$ (cf. \S\ref{sec:jac:szlifts}), the operators $U_d$ and $V_\ell$ commute, and $C_{\varphi'|U_d}(1,1)=0$ as soon as $d>1$ (cf. (\ref{eqn:jac:szlifts-Ud})). We have $C_{\varphi'}(1,1)=C_\varphi(1,1)$ by (\ref{eqn:jac:szlifts-Vell}), so if $f'$ is the newform corresponding to $\varphi'$ under the Skoruppa--Zagier map $\SZ$ (cf. \S\ref{sec:jac:szlifts}) then by applying Theorem \ref{thm:jac:szlifts-Lfns} with $D=1$ we obtain that $L(f',1)\neq 0$.

We now come to the choice of $D$ and $r$. Let $m_0$ be the product of the primes dividing $2m$. According to the proof of the main theorem in \cite{MR1074487} (see \S9 of loc. cit.) there exist infinitely many negative fundamental discriminants $D$ that are quadratic residues modulo $4m_0$, are coprime to $m_0$, and are such that $L(f'\otimes D,s)$ has a simple zero at $s=1$. So we may assume from now on that $D<-4$ is a fundamental discriminant coprime to $4m$, and $r$ is such that $D=r^2\xmod 4m$, and the derivative of $L(f'\otimes D,s)$ does not vanish at $s=1$. Then the Gross--Zagier formula (Theorem 6.3 in Ch. I of \cite{MR833192}) implies that $Z_{D,r}'(1,a')$ does not vanish in $J_0(m')$, for $a'\in O_{m'}$, where  
\begin{gather}\label{eqn:class:rpm-ZprimeDr1a}
	Z_{D,r}'(1,a'):=\sum_{Q\in \mathcal{Q}(m',D,ra')/\Gamma_0(m')}\frac{\chi'_D(Q)}{\#\Gamma_0(m')_Q}\a'_Q.
\end{gather}
Here $\a'_Q$ denotes the image in $X_0(m')$ of the unique root of $Q(x,1)$ in $\HH$, and $\chi'_D$ denotes the generalized genus character for quadratic forms of level $m'$ (cf. \S\ref{sec:notn}).

By our assumptions on $\f$ we have $Z_{D,r}(\f)=\sum_{a\in \ker(\a)}-4Z_{D,r}(1,a)$, where $Z_{D,r}(1,a)$ is as in (\ref{eqn:jac:szlifts-ZDrDprimerprime}). 
Define $Z'_{D,r}(\f):=\sum_{a\in \ker(\a)}-4Z_{D,r}'(1,a')$, where $a\mapsto a'$ denotes the natural map $O_m\to O_{m'}$. We claim that $Z'_{D,r}(\f)$ is the image of $Z_{D,r}(\f)$ under the natural map $J_0(m)\to J_0(m')$, and $Z'_{D,r}(\f)$ is not zero in $J_0(m')$. The latter claim holds because Theorem 7.7 of \cite{MR2726107} implies that $Z'_{D,r}(1,a')$ and $\a'(a')Z'_{D,r}(1,1)$ define the same point in the $f'$-isotypical component of $J_0(m')$. To verify the former claim we
apply the assumptions that $D$ is fundamental, odd and coprime to $4m$ to the proposition in \S I.1 of \cite{MR909238}. We obtain that the natural inclusion $\mathcal{Q}(m,D,ra)\to \mathcal{Q}(1,D,1)$ induces an isomorphism $\mathcal{Q}(m,D,ra)/\Gamma_0(m)\simeq \mathcal{Q}(1,D,1)/\Gamma_0(1)$, and similarly with $m'$ in place of $m$. So the summations in 
the definitions of $Z_{D,r}(1,a)$ and $Z_{D,r}'(1,a')$ can be taken over the same set of quadratic forms, and for each $Q$ in this set, $\a_Q'$ is the image of $\a_Q$ under $X_0(m)\to X_0(m')$, and $\chi'_D(Q)=\chi_D(Q)$. Also, $\#\Gamma_0(m)_Q=\#\Gamma_0(m')_Q=2$ since $D<-4$. 
So $Z'_{D,r}(1,a')$ is the image of $Z_{D,r}(1,a)$, and so $Z'_{D,r}(\f)$ is the image of $Z_{D,r}(\f)$. We have shown that $Z_{D,r}(\f)$ is not zero in $J_0(m)$, so the proof of the theorem is complete.
\end{proof}
We conclude this section by using Theorems \ref{thm:intro-firstmainthm} and \ref{thm:intro-rat} to show that an optimal mock Jacobi form of weight 1 with algebraic coefficients is a linear combination of the $\f^{(\ell)}$ for $\ell\in \gt{L}_1$.
To formulate the result precisely, set 
\begin{gather}
\JJ^\opt_{1,m}(R):=\left\{\f\in \JJ^\opt_{1,m}\mid C_\f(D,r)\in R \text{ for all } D,r\in\ZZ \right\}
\end{gather}
when $R$ is a subring of $\CC$, and define $\JJ^\top_{1,m}(R)$ similarly. According to \S\ref{sec:desc:mjt} we have $\f^{(\ell)}\in \JJ^\top_{1,m}(\ZZ)$ (for some $m$), for all $\ell\in \gt{L}_1$.
\begin{cor}\label{cor:class:rpm-theta}
If $F$ is 
an algebraic extension of $\QQ$ then 
$\JJ^\opt_{1,m}(F)=\JJ^\top_{1,m}(\ZZ)\otimes F$. In particular, any optimal mock Jacobi form of weight 1 with algebraic coefficients is a mock Jacobi theta function.
\end{cor}
\begin{proof}
Let $\f \in \JJ^\opt_{1,m}(F)$. Then for $\a\in \widehat{O}_m$ the projected function $P^\a\f$ (cf. (\ref{eqn:class:proof-Palpha})) is optimal, belongs to $\JJ^{\wk,\a}_{1,m}$, and has coefficients in $F$. So $P^\a\f\in \JJ^{\top,\a}_{1,m}(F)$ according to Theorem \ref{thm:intro-rat}. Let $\ell$ be the lambency of $\Gamma_0(m)+\ker(\a)$. Then $P^\a\f$ is an $F$-multiple of $\f^{(\ell)}$ according to Theorem \ref{thm:intro-firstmainthm}. 
So $P^\a\f\in \JJ^\top_{1,m}(\ZZ)\otimes F$.
We have $\f=\sum_\a P^\a\f$ so the corollary is proved.
\end{proof}

\subsection{Principal Moduli}\label{sec:class:pm}

We now develop some consequences of the genus zero classification of optimal mock Jacobi theta functions that has been obtained in \S\ref{sec:class:proof}. The first of these is a concrete construction of the shadow 
of $\f^{(\ell)}$, for each of the $39$ lambencies $\ell\in \gt{L}_1$. We achieve this by expressing $\xi(\f^{(\ell)})$ (cf. (\ref{eqn:jac:mock-xi})) in terms of a specific principal modulus (a.k.a. Hauptmodul) $T^{(\ell)}$ for the corresponding genus zero group. The theorem in this section relates the principal modulus $T^{(\ell)}$ directly to the mock modular form $\f^{(\ell)}$ via the generalized Borcherds product construction of Bruinier--Ono \cite{MR2726107}. This gives us---in principle---a construction of $\f^{(\ell)}$ in terms of the apparently simpler object $T^{(\ell)}$.

To explain the definition of the $T^{(\ell)}$, suppose that $\Gamma<\SL_2(\RR)$ is commensurable with $\SL_2(\ZZ)$, and recall that a $\Gamma$-invariant holomorphic function $f$ on $\HH$ such that $f(\gamma\tau)=O(e^{C\Im(\tau)})$ as $\Im(\tau)\to \infty$ for some $C>0$, for every $\gamma\in \SL_2(\ZZ)$, is called a {\em principal modulus} for $\Gamma$ if the induced function $X_\Gamma\to \CC\cup\{\infty\}$ is an isomorphism of Riemann surfaces.
Evidently, a group $\Gamma$ admits principal moduli if and only if it has genus zero.

The following lemma characterizes the functions $T^{(\ell)}$ abstractly.
They are specified explicitly in Table \ref{tab:intro-pm}, where a symbol of the form $n_1^{d_1}\cdots n_l^{d_l}$ is used as a shorthand for the eta product $\eta(n_1\tau)^{d_1}\cdots \eta(n_l\tau)^{d_l}$. 
\begin{lem}\label{lem:class:pm-chrTell}
Let $\ell\in \gt{L}_1$ and let $\Gamma=\Gamma_0(m)+K$ be the corresponding genus zero group. Then $T^{(\ell)}$ is the unique principal modulus for $\Gamma$ such that $T^{(\ell)}(\tau)=q^{-1}+O(1)$ as $\Im(\tau)\to \infty$, and such that the product $T^{(\ell)}(T^{(\ell)}|W_m)$ is constant. 
\end{lem}
\begin{proof}
It can be verified directly using the expressions in Table \ref{tab:intro-pm} that the $T^{(\ell)}$ satisfy the stated conditions. Uniqueness holds because the first condition determines $T^{(\ell)}$ up to an additive constant, and the second condition determines that constant.
\end{proof}

Let $\ell\in \gt{L}_1$ and let $m$ be the level $\ell$ (cf. (\ref{eqn:class:proof-JJ1mtopK})). 
Define a skew-holomorphic theta type Jacobi form $\s^{(\ell)}\in T^\sk_{2,m}$ by setting 
\begin{gather}\label{eqn:class:pm-sigmaell}
	\s^{(\ell)}(\tau,z):=
	-\frac1{\sqrt{2}}
	\sum_{r,r'\xmod 2m} \overline{\th^1_{m,r}(\tau)}\Omega^{(\ell)}_{r,r'}\th_{m,r'}(\tau,z)
\end{gather}
where $\Omega^{(\ell)}:=\sum_i d_i\Omega_m(n_i)$ in case $T^{(\ell)}=\prod_i \eta(n_i\tau)^{d_i}$.
We will see momentarily that $\s^{(\ell)}$ is the shadow of $\f^{(\ell)}$.
Thus the shadow of $\f^{(\ell)}$ is constructed explicitly in terms of the principal modulus $T^{(\ell)}$.

It can be checked that $\s^{(\ell)}$ corresponds to the logarithmic derivative of $T^{(\ell)}$ under the Skoruppa--Zagier map $\SZ$ (cf. \S\ref{sec:jac:szlifts}). More specifically, we have the following result, which may be checked directly. 
\begin{lem}\label{lem:class:pm-sigmaellprops}
Let $\ell\in \gt{L}_1$. Then
$\mathcal{S}_{1,1}(\s^{(\ell)})=\frac{1}{2\pi i}\frac{\rm d}{{\rm d}\tau}\log T^{(\ell)}(\tau)$. Further, if $\ell=m+n,n',\dots$ then $\s^{(\ell)}$ is invariant under the Eichler--Zagier operators $\W_m(n), \W_m(n'),\dots$. 
\end{lem}

We now compute the shadow of $\f^{(\ell)}$.
\begin{prop}\label{prop:class:pm-xiphisig}
Let $\ell\in \gt{L}_1$. Then $\xi(\f^{(\ell)})=\s^{(\ell)}$.
\end{prop}
\begin{proof}
We first consider the case that $\ell$ corresponds to $\Gamma=\Gamma_0(m)+K$ where $K=\ker(\a)$ for some $\a\in \widehat{O}_m$. 
Then 
$\xi(\f^{(\ell)})\in T^{\sk,\a}_{2,m}$ by (\ref{eqn:class:defn-xiJJtopa}), and $\s^{(\ell)}\in T^{\sk,\a}_{2,m}$ by Lemma \ref{lem:class:pm-sigmaellprops}.
From the proof of Proposition \ref{prop:class:tht-tht} 
we have an isomorphism
\begin{gather}\label{eqn:class:pm-Tskalpha2mId}
	T^{\sk,\a}_{2,m}\simeq 
		\bigoplus_{d^2|m}
		\bigoplus_{\substack{\a'\in \widehat{O}_{m/d^2}\\\a'=\a}}
		\CC \Id_{m^2/d}^{\new,\a'}		
\end{gather} 
where $\Id^{\new,\a}_m$ is the identity 
map $\Th_m^{\new,\a}\to \Th_m^{\new,\a}$, and the condition $\a'=\a$ in the second direct sum means equality as elements of $\widehat{O}_m$ (cf. (\ref{eqn:class:proof-alphaprimelift})). Observe that if $\varphi\in T^{\sk,\a}_{2,m}$ corresponds under (\ref{eqn:class:pm-Tskalpha2mId}) to an element of $\CC\Id^{\new,\a'}_{m/d^2}$ then $\varphi$ is in the image of $U_d$ (cf. (\ref{eqn:jac:szlifts-Ud})). If $d>1$ then we have $C_{\varphi}(1,r)=0$ for all $r$, so $\{\phi^{(\ell)},\varphi\}=0$ according to Proposition \ref{prop:jac:hol-lampol}. But $\{\phi^{(\ell)},\varphi\}=\lab \xi(\f^{(\ell)}),\varphi\rab$ by definition (cf. (\ref{eqn:jac:mock-BFpairing})), so it must be that $\xi(\f^{(\ell)})$ corresponds to an element of $\CC\Id^{\new,\a}_m$ under (\ref{eqn:class:pm-Tskalpha2mId}). The same is true of $\s^{(\ell)}$ by inspection, so $\xi(\f^{(\ell)})$ is proportional to $\s^{(\ell)}$. 

Suppose $\xi(\f^{(\ell)})=c\s^{(\ell)}$. Then 
we have
$\{\f^{(\ell)},\s^{(\ell)}\}=c\lab \s^{(\ell)},\s^{(\ell)}\rab$ by definition. Using the Rankin--Selberg formula (cf. \S3 of \cite{Dabholkar:2012nd}), we compute
\begin{gather}
	\int_{\mathcal{F}}\th^1_{m,r}(\t)\overline{\th^1_{m,r'}(\t)}v^{-\frac12}{\rm d}u{\rm d}v
	=\frac{\sqrt{m}}{12}(\delta_{r,r'\xmod 2m}-\delta_{r,-r'\xmod 2m}).
\end{gather}
Applying this to (\ref{eqn:jac:hol-PetIP}) 
we obtain $\lab \s^{(\ell)},\s^{(\ell)}\rab=\frac{\sqrt{2}}{24}\|\Omega^{(\ell)}\|^2$, where $\|\Omega^{(\ell)}\|^2=\sum_{r,r'}|\Omega^{(\ell)}_{r,r'}|^2$. So $\{\f^{(\ell)},\s^{(\ell)}\}={c}\frac{\sqrt{2}}{24}\|\O^{(\ell)}\|^2$. On the other hand, we find 
\begin{gather}
	\{\f^{(\ell)},\s^{(\ell)}\}=
	-4\sum_{r\in K}\overline{C_{\s^{(\ell)}}(1,r)}=4\sqrt{2}\sum_{r\in K}\O^{(\ell)}_{1,r},
\end{gather}
using Proposition \ref{prop:jac:hol-lampol}. Inspection reveals that $\|\Omega^{(\ell)}\|^2=96\sum_{r\in K}\O^{(\ell)}_{1,r}$, for all $\ell\in\gt{L}_1$, so $c=1$, as was claimed.

In the remaining cases we have $\f^{(\ell)}=\frac12(\f^{(\ell_1)}+\f^{(\ell_2)})$ for $\ell_i\in \gt{L}_1$ corresponding to $\Gamma_0(m)+\ker(\a_i)$ for some $\a_i\in \widehat{O}_m$. Inspecting Table \ref{tab:intro-pm} again we find $(T^{(\ell)})^2=T^{(\ell_1)}T^{(\ell_2)}$, so $\s^{(\ell)}=\frac12(\s^{(\ell_1)}+\s^{(\ell_2)})$. So $\xi(\f^{(\ell)})=\s^{(\ell)}$ since we have already checked that $\xi(\f^{(\ell_i)})=\s^{(\ell_i)}$. This completes the proof.
\end{proof}

At this point we note a beautiful counterpoint between the mock Jacobi theta functions $\f^{(\ell)}$ and the dual defining modular properties of the corresponding genus zero groups $\Gamma$. Namely, if $\ell\in \gt{L}_1$ and $\Gamma=\Gamma_0(m)+K$ is the corresponding group, then the existence of $\f^{(\ell)}$ implies the vanishing of $S_2(\Gamma)$, according to Corollary \ref{cor:class:proof-genK}. Conversely, the vanishing of $S_2(\Gamma)$ implies the existence of $T^{(\ell)}$, which in turn determines the shadow of $\f^{(\ell)}$, according to Proposition \ref{prop:class:pm-xiphisig} and the definition (\ref{eqn:class:pm-sigmaell}).

Our final goal in this section is to demonstrate a constructive relationship between the principal moduli $T^{(\ell)}$ and their corresponding mock Jacobi forms $\f^{(\ell)}$. More specifically, we will show that if $\ell\in \gt{L}_1$ then certain generalized Borcherds products, obtained from $\f^{(\ell)}$ by applying results \cite{MR2726107} of Bruinier--Ono, are explicitly computable rational functions in the principal modulus $T^{(\ell)}$. This result extends Theorem 1.1 of \cite{MR3357517}, and demonstrates that the Fourier coefficients of $\f^{(\ell)}$ may be expressed in terms of the singular moduli of $T^{(\ell)}$. Thus, in principle, the mock Jacobi form $\f^{(\ell)}$ may be constructed directly from its principal modulus $T^{(\ell)}$.

To formulate this precisely, let $\ell\in \gt{L}_1$ and let $\Gamma=\Gamma_0(m)+K$ be the corresponding genus zero group.
For $D$ a negative fundamental discriminant and $r\xmod 2m$ such that $D=r^2\xmod 4m$ define 
a formal product $\Psi^{(\ell)}_{D,r}$ 
by setting
\begin{gather}\label{eqn:class:pm-Psiell}
	\Psi^{(\ell)}_{D,r}(\tau):=\prod_{n>0}\prod_{b\xmod D}\left(1-\ex\left(\tfrac b D\right)q^n\right)^{\left(\frac D b\right)C^{(\ell)}(Dn^2,rn)},
\end{gather}
where $C^{(\ell)}(D,r):=C_{\phi^{(\ell)}}(D,r)$ and $\left(\frac D b\right)$ is the Kronecker symbol (cf. \S\ref{sec:notn}). 
Set $Z^{(\ell)}_{D,r}:=Z_{D,r}(\phi^{(\ell)})$ (cf. (\ref{eqn:jac:szlifts-ZDrDprimerprime})) and write 
$\check{Z}^{(\ell)}_{D,r}$ for the image of $Z^{(\ell)}_{D,r}$ under the natural map $X_0(m)\to X_\Gamma$.

\begin{thm}
\label{thm:hauptmodul_Psi}
Let $\ell\in \gt{L}_1$ and let $\Gamma=\Gamma_0(m)+K$ be the corresponding genus zero group. Let $D$ be a negative fundamental discriminant and suppose that $D=r^2\xmod 4m$. Assume also that $D\neq -3$ if $m\in \{7,13,21\}$. 
Then $\Psi^{(\ell)}_{D,r}(\tau)$ descends to a meromorphic function on $X_\Gamma$ whose divisor is $\check{Z}^{(\ell)}_{D,r}$. In particular, $\Psi^{(\ell)}_{D,r}$ is a rational function in $T^{(\ell)}$.
\end{thm}
\begin{proof}
Set $\Phi^{(\ell)}_{D,r}:=\mathcal{S}_{D,r}^\reg\f^{(\ell)}$ (cf. (\ref{eqn:jac:szlifts-SDrreg})), and write $\widehat{Z}^{(\ell)}_{D,r}$ for the pullback of $Z^{(\ell)}_{D,r}$ along the natural map $\HH\to X_0(m)$.
Then Theorem \ref{thm:jac:szlifts-SDrreg} tells us that $\Phi^{(\ell)}_{D,r}$ is a $\Gamma_0(m)$-invariant function on $\HH\setminus \widehat{Z}^{(\ell)}_{D,r}$ with a logarithmic singularity on $-4\widehat{Z}^{(\ell)}_{D,r}$, and 
\begin{gather}
\Phi^{(\ell)}_{D,r}(\tau)=-4\sum_{n>0}\sum_{b\xmod D}\left(\frac{D}{b}\right)C^{(\ell)}(Dn^2,rn)\log\left|1-\ex\left(n\tau+\frac b D\right)\right|
\end{gather}
for $\Im(\t)$ sufficiently large.
From this we conclude that the product defining $\Psi^{(\ell)}_{D,r}$ converges for $\Im(\t)$ sufficiently large, and admits an analytic continuation satisfying $\Phi^{(\ell)}_{D,r}(\tau)=-4\log\left|\Psi^{(\ell)}_{D,r}\right|$ on $\HH\setminus\widehat{Z}^{(\ell)}_{D,r}$. In particular, $\Psi^{(\ell)}_{D,r}$ inherits modular invariance from $\Phi^{(\ell)}_{D,r}$. Specifically, we have 
\begin{gather}\label{eqn:class:pm-Psiellxfm}
\Psi^{(\ell)}_{D,r}(\gamma \tau)=\sigma(\gamma)\Psi^{(\ell)}_{D,r}(\tau)
\end{gather}
for $\gamma\in \Gamma_0(m)$, for some unitary character $\sigma:\Gamma_0(m)\to \CC^\times$. 
Actually, the formula (\ref{eqn:class:pm-Psiellxfm}) holds for $\gamma\in \Gamma=\Gamma_0(m)+K$, for some extension of $\sigma$ to $\Gamma$. For arguing as in the proof of Lemma \ref{lem:class:proof-ALinv}, using the identities (\ref{eqn:class:proof_Qwnids}), and the $\mathcal{W}_m(n)$-invariance of $\phi^{(\ell)}$ for $n$ such that $a(n)\in K$, we conclude that $\Phi^{(\ell)}_{D,r}$ and $\widehat{Z}^{(\ell)}_{D,r}$ are also $W_n$-invariant, when $a(n)\in K$. 

We now show that the character in (\ref{eqn:class:pm-Psiellxfm}) is trivial. 
Note that the restriction on $D$ implies that either $\Psi^{(\ell)}_{D,r}$ is constant, and $Z^{(\ell)}_{D,r}$ vanishes, or the cardinality of $\Gamma_0(m)_Q$ is $2$ or $4$ for all $Q\in \mathcal{Q}(m,D,ra)$ (for all $a\in K$). We have 
\begin{gather}\label{eqn:class:pm-ZellDr}
	Z^{(\ell)}_{D,r}=\sum_{a\in K}\sum_{Q\in \mathcal{Q}(m,D,ra)/\Gamma_0(m)}
	\frac{-4}{\# \Gamma_0(m)_Q}
	\chi_D(Q)
	\a_Q
\end{gather}
since $C^{(\ell)}(1,1)=-2$, so the hypotheses imply that the divisor $Z^{(\ell)}_{D,r}$ is integral. Also, ${Z}^{(\ell)}_{D,r}$ has degree zero, so $\check{Z}^{(\ell)}_{D,r}$ has degree zero too. Since $\Gamma$ has genus zero $\check{Z}^{(\ell)}_{D,r}$ is a principal divisor on $X_\Gamma$, and we may consider a meromorphic function $f$ on $X_\Gamma$ whose associated divisor is $\check{Z}^{(\ell)}_{D,r}$. Then the expression $|\Psi^{(\ell)}_{D,r}/f|$ defines a harmonic function on $X_\Gamma$ with no singularities, which must therefore be constant. So $\Psi^{(\ell)}_{D,r}/f$ is a holomorphic function on $\HH$ with constant modulus, and must therefore also be constant. We conclude that the character $\sigma$ in (\ref{eqn:class:pm-Psiellxfm}) is trivial, as required. The claimed results follow from this. 
\end{proof}

\subsection{Positivity and Moonshine} 
\label{sec:class:rpm}

In this final section we 
identify two sign conditions on Fourier coefficients which, independently, single out the mock Jacobi forms that appear in umbral moonshine.

Let $\gt{L}_1^+$ be the subset of $\gt{L}_1$ composed of the 23 lambencies that have a root system $X$ attached to them by Table \ref{tab:intro-pm}. We call these the {\em Niemeier} lambencies because 
they are in correspondence with the Niemeier lattices, according to the analysis in \S2.3 of \cite{MUM}. 
As explained in \S\ref{sec:intro:moon}, it is the $\phi^{(\ell)}$ for $\ell\in \gt{L}_1^+$ that appear in umbral moonshine. 
But they may be characterized more abstractly in terms of positivity conditions on Fourier coefficients, as we now demonstrate. 

To formulate these results, define $\varepsilon_m:\ZZ\to \{0,\pm1\}$ by setting $\varepsilon_m(r):=\pm 1$ when $\pm r\xmod 2m$ is represented by an integer in the range $0<r'<m$, and $\varepsilon_m(r)=0$ when $r=0\xmod m$. For convenience we define $\sgn(0):=0$.

\begin{prop}\label{prop:class:rpm-possig}
Let $\ell\in \gt{L}_1$ and let $m$ be the level of $\ell$. Then 
\begin{gather}\label{eqn:class:rpm-possig}
\sgn(C_{\sigma^{(\ell)}}(k^2,r))=\varepsilon_m(k)\varepsilon_m(r)
\end{gather}
for all pairs $(k,r)$ such that $k\in \ZZ^+$ and $r\xmod 2m$ satisfies $k^2=r^2\xmod 4m$ if and only if $\ell\in\gt{L}_1^+$.
\end{prop}
\begin{proof}
Let $\ell$ and $m$ be as in the statement of the proposition. Inspecting the definition (\ref{eqn:class:pm-sigmaell}) of $\s^{(\ell)}$ we see that 
\begin{gather}
	\begin{split}
\sgn\left( C_{\sigma^{(\ell)}}(k^2,r) \right) &= \pm \sgn\left( C_{\sigma^{(\ell)}}(k^2,r') \right) \text{ if } r =\pm r'\xmod 2m,\\ 
\label{sign_equivalences}
\sgn\left( C_{\sigma^{(\ell)}}(k^2,r) \right) &= \pm \sgn\left( C_{\sigma^{(\ell)}}((k')^2,r) \right) \text{ if }  k,k'>0 \text{ and }k =\pm k'\xmod 2m.
	\end{split}
\end{gather}
Using this we can reduce the sign condition 
to checking that $C_{\sigma^{(\ell)}}(D,r)\geq 0$ for all $0<D<m^2$ and $0<r<m$ if and only if $\ell\in\gt{L}_1^+$.
This, in turn, can be verified by inspection, using the definition (\ref{eqn:class:pm-sigmaell}) of $\s^{(\ell)}$, and the descriptions in Table \ref{tab:intro-pm} of the principal moduli $T^{(\ell)}$. 
\end{proof}

The second positivity property that distinguishes the twenty-three lambencies in $\gt{L}_1^+$ 
is the positivity of the coefficients of the corresponding optimal mock Jacobi theta functions themselves. 
\begin{prop}\label{prop:class:rpm-posphi}
Let $\ell\in \gt{L}_1$ and let $m$ be the level of $\ell$. Then 
\begin{gather}\label{eqn:class:rpm-posphi}
\sgn(C_{\f^{(\ell)}}(D,r))=\varepsilon_m(r)
\end{gather}
for all $D<0$ if and only if $\ell\in\gt{L}_1^+$.
\end{prop}
\begin{proof}
The fact that (\ref{eqn:class:rpm-posphi}) holds for $\ell\in \gt{L}_1^+$ was conjectured in \cite{MUM} and proven in  \cite{MR3539377,mnstmlts}. That this sign condition is not satisfied for any $\ell\in \gt{L}_1$ that is not in $\gt{L}_1^+$ 
can be checked explicitly from the descriptions of the $\f^{(\ell)}$ given in \S\ref{sec:desc:mjt}. 
\end{proof}

Note that the positivity properties (\ref{eqn:class:rpm-possig}) and (\ref{eqn:class:rpm-posphi}) do not a priori imply one another and should be regarded as independent. 
In contrast to the condition (\ref{eqn:class:rpm-possig}) on the coefficients of the shadow functions, the condition (\ref{eqn:class:rpm-posphi})
on the Fourier coefficients of the mock theta functions is a condition on infinitely many terms, that cannot easily be reduced to a finite check (cf. \S4 in \cite{MR3539377}).

The positivity conditions (\ref{eqn:class:rpm-possig}) and (\ref{eqn:class:rpm-posphi}) for $\ell\in \gt{L}_1^+$ play important roles in the umbral moonshine conjectures. The former leads to an assignment of root systems to each $\ell\in \gt{L}_1^+$ (cf. \cite{MUM}), and thereby to finite groups $G^{(\ell)}$, whereas the the latter condition allows one to hypothesize that the coefficients $C_{\f^{(\ell)}}(D,r)$ are dimensions of representations of $G^{(\ell)}$. Although the coefficients $C_{\f^{(\ell)}}(D,r)$ fail to satisfy (\ref{eqn:class:rpm-posphi}) when $\ell\in \gt{L}_1$ does not belong to $\gt{L}_1^+$, the descriptions in \S\ref{sec:desc:mjt} suggest that an interpretation of these coefficients as dimensions of modules for some finite groups $G^{(\ell)}$ is still possible, subject to a suitable modification of (\ref{eqn:class:rpm-posphi}).

%---------------------------------------------------------------------------------------%
\section*{Acknowledgements}
%---------------------------------------------------------------------------------------%

We thank Barry Mazur, Ken Ono, Nils-Peter Skoruppa and Don Zagier for discussions on related topics. The work of M.C. was supported by ERC starting grant H2020 ERC StG 2014. J.D. acknowledges support from the U.S. National Science Foundation (DMS 1203162, DMS 1601306), and the Simons Foundation (\#316779), and thanks the University of Tokyo for hospitality during the middle stages of this work.

%------------------------------------------------------------------------%
\appendix
%------------------------------------------------------------------------%

%---------------------------------------------------------------------------------------%
\section{Descriptions}\label{sec:desc}
%---------------------------------------------------------------------------------------%

In \S\ref{sec:desc:mjt} we give explicit constructions for the optimal mock Jacobi theta functions that have not already appeared in umbral moonshine. In \S\ref{sec:mockthetafunctions} we present concrete expressions for the mock theta functions of Ramanujan, and also those identified by Andrews and Gordon--McIntosh, in terms of the theta-coefficients of optimal mock Jacobi theta functions.

\subsection{Optimal Mock Jacobi Theta Functions}\label{sec:desc:mjt}

In this section we describe the optimal mock Jacobi theta functions $\f^{(\ell)}$ that have not already been discussed in the context of umbral moonshine. Denote the theta-coefficients of $\f^{(\ell)}$ by $H^{(\ell)}_r$, so that $\f^{(\ell)}(\tau,z) = \sum_{r\xmod 2m}H^{(\ell)}_r(\tau)\th_{m,r}(\tau,z)$, where $m$ is the level of $\ell$. The $H^{(\ell)}_r$ are described explicitly for all $\ell\in\gt{L}_1^+$ in \S B.3 of \cite{umrec}, so it remains to treat $\f^{(\ell)}$ for $\ell\in \gt{L}_1\setminus\gt{L}_1^+$. These are the lambencies appearing in (\ref{eqn:class:proof-gzgps-emum}).

\begin{table}[ht]
\centering
\scalebox{.9}{
\begin{tabular}{CCCL}\toprule
\ell&\ell'&[g']&{\text{ Relations}} \\\midrule
15+5 & 3& 5A &   \sum_{n\in\ZZ/5} \ H^{(15+5)}_{r+6n}(5\t) = 2 H^{(3)}_{5A,r}(\t),\text{ $r\xmod 6$}  \\\midrule
\multirow{2}*{20+4}& \multirow{2}*{5} & \multirow{2}*{2C} 
&(H^{(20+4)}_{r}+H^{(20+4)}_{r+10} )(4\t) =  H^{(5)}_{g,r}(\t),\text{ $r=\pm1,\pm 3\xmod 10$} 
\\
&&&   (H^{(20+4)}_{2r})(\t) = \ex({\frac{r^2}{160}}) H^{(5)}_{g,r}(\t+\tfrac{1}{2}), \text{ $r=\pm2,\pm 4\xmod 10$}
 \\\midrule
 21+3& 7& 3AB& \sum_{n\in\ZZ/3} H^{(21+3)}_{r+14n}(3\t) = 2 H^{(7)}_{3AB,r}(\t),\text{ $r\xmod 14$}\\\midrule
 24+8 & 3& 8CD & \sum_{n\in\ZZ/4} H^{(24+8)}_{r+12n}(8\t) =  H^{(3)}_{8CD,r}(\t),\text{ $r\xmod 6$}\\\midrule
 28+7 & 4 & 7AB&  \sum_{n\in\ZZ/7} H^{(28+7)}_{r+8n}(7\t) = 2 H^{(4)}_{7AB,r}(\t),\text{ $r\xmod 8$}\\\midrule
 33+11& 3&11AB&  \sum_{n\in\ZZ/11} H^{(33+11)}_{r+6n}(11\t) = 2 H^{(3)}_{11AB,r}(\t),\text{ $r\xmod 6$}\\\midrule 
 \multirow{4}*{36+4} & \multirow{2}*{3} &\multirow{2}*{6C}&  \sum_{n\in\ZZ/6} H^{(36+4)}_{r+12n}(12\t) = H^{(3)}_{6C,r}(\t),\text{ $r=\pm1\xmod 6$}\\
 &&& \ex({\frac16}) \sum_{n=0}^{2}H^{(36+4)}_{2+12n}(3\t+\tfrac{3}{2}) = H^{(3)}_{6C,2}(\t)\\  \cmidrule(l){2-4}
 &\multirow{2}*{9}&\multirow{2}*{2B}&\sum_{n\in\ZZ/4} H^{(36+4)}_{r+18n}(4\t) = H^{(9)}_{2B,r}(\t),\text{ $r=\pm3\xmod 18$}\\
 &&&H^{(36+4)}_{12}(\t) = -H^{(9)}_{2B,6}(\t+\tfrac{1}{2})\\\midrule
 60+12,15,20& 30+6,10,15& 2A &  \sum_{n\in\ZZ/2}H^{(60+12,15,20)}_{r+60n}(2\t) = H^{(30+6,10,15)}_{2A,r}(\t),\text{ $r\xmod 60$} \\
 \bottomrule
\end{tabular}
}
\caption{Multiplicative relations.}
\label{tab:desc:mjt-mltrel}
\end{table}

First consider the five lambencies appearing in the first line of (\ref{eqn:class:proof-gzgps-emum}). For each such $\ell$ there is an $\ell'$ in $\gt{L}_1^+$ such that $\f^{(\ell)}=\frac12\left(\f^{(\ell')}+\f^{(\ell')}|\W_m(n)\right)$ for some $n\in \Ex_m$, where $m$ is the level of $\ell$. Specifically, 
\begin{gather}\label{eqn:desc:mjt:firstfive}
\begin{split}
\f^{(6+2)} &= \frac12\left(\f^{(6)} + \f^{(6)}\lvert{\cal W}_6(2)\right),\\
\f^{(10+2)} &= \frac12\left(\f^{(10)} + \f^{(10)}\lvert{\cal W}_{10}(2)\right),\\
\f^{(12+3)} &= \frac12\left(\f^{(12)} + \f^{(12)}\lvert{\cal W}_{12}(3)\right),\\
\f^{(18+2)} &= \frac12\left(\f^{(18)} + \f^{(18)}\lvert{\cal W}_{18}(2)\right),\\
\f^{(30+3,5,15)} &= \frac12\left(\f^{(30+15)} + \f^{(30+15)}\lvert{\cal W}_{30}(3)\right).
\end{split}
\end{gather}
We can also write $\f^{(\ell)}=P^\a\f^{(\ell')}$ where $\a$ is the character of $O_m$ such that $\Gamma_0(m)+\ker(\a)$ is the genus zero group corresponding to $\ell$.
The theta-coefficients $H^{(\ell')}_r$ of these $\f^{(\ell')}$ are described explicitly in \S B.3 of \cite{umrec}, so the $\f^{(\ell)}$ for $\ell$ in the first line of (\ref{eqn:class:proof-gzgps-emum}) are determined by (\ref{eqn:desc:mjt:firstfive}). Some low order Fourier coefficients of these $\f^{(\ell)}$ are given in Tables \ref{tab:desc:mjt-6+2} through \ref{tab:desc:mjt-30+3,5,15}.

Next consider the $\ell$ appearing in the second line of (\ref{eqn:class:proof-gzgps-emum}). Once again the $H^{(\ell)}_r$ are determined by mock modular forms attached to a lambency $\ell'\in\gt{L}_1^+$, but with level greater than one. Specific relations are given in Table \ref{tab:desc:mjt-mltrel} which are sufficient to determine the $H^{(\ell)}_r$ in question completely, in terms of corresponding functions $H^{(\ell')}_{g',r}$, which are in turn described explicitly in \cite{umrec}. Low order Fourier coefficients of these mock modular forms are given in Tables \ref{tab:desc:mjt-15+5} through \ref{tab:desc:mjt-60+12,15,20}.

\begin{table}[ht]
\begin{small}
\centering
\begin{tabular}{c|rrrrrrrrrrrrrrrr}\toprule
$r, n$&0&1&2&3&4&5&6&7&8&9&10&11&12&13&14&15\\\midrule
1 & -2& -2& 4& -6& 6& -6& 10& -14& 12& -12&20&-24&22&-26&34&-40\\ 
2& &16& 32& 64& 112& 176& 288& 448& 656& 976& 1408&1984&2800&3872&5280&7168\\
4&&8& 24& 48& 80& 144& 232& 352& 544& 808& 1168&1680&2368&3280&4528&6176\\\bottomrule
\end{tabular}
\caption{Fourier coefficients $C^{(6+2)}(r^2-24n,r)$.}
\label{tab:desc:mjt-6+2}
\end{small}
\end{table}

\begin{table}[ht]
\begin{small}
\centering
\begin{tabular}{c|rrrrrrrrrrrrrrrr}\toprule
$r,n$&0&1&2&3&4&5&6&7&8&9&10&11&12&13&14&15\\\midrule
1 & -2 & 2 & -2 & 0 & -2 & 4 & -2 & 2 & -2 & 4 & -6 & 2 & -4 & 8 & -6 & 4 \\
 2 &  & 4 & 8 & 16 & 16 & 28 & 40 & 48 & 72 & 96 & 120 & 160 & 208 & 256 & 328 & 416 \\
 3 &  & -2 & 2 & -2 & 4 & -4 & 2 & -4 & 6 & -6 & 6 & -6 & 8 & -8 & 8 & -10 \\
 4 &  & 8 & 12 & 16 & 28 & 40 & 56 & 80 & 104 & 136 & 184 & 240 & 304 & 392 & 496 & 624 \\
 6 &  & 4 & 8 & 16 & 24 & 32 & 48 & 64 & 88 & 124 & 160 & 208 & 272 & 348 & 440 & 560 \\
 8 &  &  & 4 & 8 & 8 & 16 & 24 & 32 & 44 & 64 & 80 & 104 & 144 & 176 & 232 & 296 \\
\bottomrule
\end{tabular}
\caption{Fourier coefficients $C^{(10+2)}(r^2-40n,r)$.}
\label{tab:desc:mjt-10+2}
\end{small}
\end{table}

\begin{table}[ht]
\begin{small}
\centering
\begin{tabular}{c|rrrrrrrrrrrrrrrr}\toprule
$r,n$&0&1&2&3&4&5&6&7&8&9&10&11&12&13&14&15\\\midrule
 1 & -2 & -4 & -2 & -6 & -8 & -12 & -10 & -22 & -22 & -34 & -38 & -52 & -58 & -84 & -92 & -120 \\
 2 &   & 4 & 4 & 8 & 12 & 12 & 20 & 28 & 32 & 44 & 56 & 68 & 88 & 112 & 132 & 164 \\
 3 &   & 4 & 8 & 8 & 16 & 20 & 28 & 32 & 48 & 60 & 80 & 92 & 124 & 148 & 188 & 224 \\
 5 &   & 2 & 2 & 6 & 6 & 8 & 12 & 18 & 20 & 30 & 32 & 44 & 58 & 72 & 84 & 110 \\
 6 &   & 4 & 8 & 12 & 16 & 24 & 32 & 44 & 56 & 72 & 96 & 120 & 152 & 188 & 232 & 288 \\
 9 &   &  & 4 & 4 & 12 & 8 & 20 & 20 & 36 & 36 & 56 & 64 & 92 & 100 & 144 & 160 \\
\bottomrule
\end{tabular}
\caption{Fourier coefficients $C^{(12+3)}(r^2-48n,r)$.}
\label{tab:desc:mjt-12+3}
\end{small}
\end{table}

\begin{table}[ht]
\begin{small}
\centering
\begin{tabular}{c|rrrrrrrrrrrrrrrr}\toprule
$r,n$&0&1&2&3&4&5&6&7&8&9&10&11&12&13&14&15\\\midrule
1 & -2 & 0 & -2 & 0 & -2 & 2 & 0 & 0 & -2 & 0 & -2 & 2 & -2 & 2 & -2 & 2 \\
 2 &      & 4 & 0 & 4 & 4 & 4 & 4 & 8 & 8 & 8 & 12 & 12 & 16 & 20 & 20 & 28 \\
 3 &      & -2 & 2 & 0 & 0 & -2 & 2 & -2 & 2 & -2 & 2 & 0 & 2 & -4 & 2 & -2 \\
 4 &      & 0 & 4 & 4 & 4 & 8 & 8 & 8 & 16 & 16 & 20 & 24 & 28 & 32 & 40 & 48 \\
 5 &      & 2 & 0 & 0 & 0 & 2 & 0 & 0 & 0 & 2 & -2 & 0 & 0 & 2 & 0 & 0 \\
 6 &      & 4 & 4 & 4 & 8 & 8 & 12 & 16 & 16 & 20 & 24 & 32 & 36 & 44 & 52 & 60 \\
 7 &      & -2 & 0 & -2 & 2 & -2 & 0 & -2 & 2 & -2 & 2 & -2 & 2 & -4 & 2 & -2 \\
 8 &      & 4 & 4 & 8 & 8 & 8 & 12 & 16 & 16 & 24 & 28 & 32 & 40 & 48 & 52 & 64 \\
 10 &      &  & 4 & 4 & 4 & 8 & 12 & 12 & 16 & 20 & 24 & 28 & 36 & 40 & 52 & 60 \\
 12 &      &  & 2 & 4 & 4 & 8 & 8 & 8 & 12 & 16 & 20 & 24 & 28 & 32 & 40 & 48 \\
 14 &      &  &  & 4 & 4 & 4 & 4 & 8 & 8 & 12 & 12 & 12 & 20 & 24 & 24 & 32 \\
 16 &      &  &  &  & 0 & 0 & 4 & 0 & 4 & 4 & 4 & 8 & 8 & 8 & 12 & 16 \\\bottomrule
\end{tabular}
\caption{Fourier coefficients $C^{(18+2)}(r^2-72n,r)$.}
\label{tab:desc:mjt-18+2}
\end{small}
\end{table}

\begin{table}[ht]
\begin{small}
\centering
\begin{tabular}{c|rrrrrrrrrrrrrrrrrr}\toprule
$r,n$&0&1&2&3&4&5&6&7&8&9&10&11&12&13&14&15\\\midrule
  1 & -2 & -2 & -2 & -4 & -2 & -2 & -4 & -6 & -6 & -6 & -6 & -8 & -10 & -10 & -10 & -14 \\
 3 & & 0 & 4 & 0 & 4 & 0 & 4 & 4 & 4 & 4 & 8 & 4 & 8 & 4 & 12 & 8 \\
 5 & & 4 & 0 & 4 & 4 & 4 & 4 & 8 & 4 & 8 & 8 & 8 & 12 & 16 & 12 & 16 \\
 7 & & -2 & 0 & 0 & -2 & -2 & 0 & -4 & 0 & -2 & -2 & -4 & -2 & -6 & -4 & -4 \\
 9 & & 4 & 4 & 4 & 4 & 8 & 8 & 8 & 8 & 12 & 12 & 16 & 16 & 20 & 20 & 24 \\
 15 & & 0 & 4 & 4 & 8 & 4 & 8 & 8 & 12 & 8 & 16 & 16 & 20 & 16 & 24 & 24 \\\bottomrule
\end{tabular}
\caption{Fourier coefficients $C^{(30+3,5,15)}(r^2-120n,r)$.}
\label{tab:desc:mjt-30+3,5,15}
\end{small}
\end{table}

\begin{table}[ht]
\begin{small}
\centering
\begin{tabular}{c|rrrrrrrrrrrrrrrr}\toprule
$r,n$&0&1&2&3&4&5&6&7&8&9&10&11&12&13&14&15\\\midrule
 1 & -2 & -2 & -2 & -2 & -4 & -4 & -6 & -10 & -8 & -10 & -16 & -18 & -22 & -28 & -32 & -36 \\
 2 &     & 4 & 4 & 8 & 8 & 12 & 16 & 20 & 24 & 32 & 36 & 48 & 56 & 72 & 80 & 100 \\
 4 &     & 0 & 4 & 0 & 4 & 4 & 8 & 4 & 12 & 8 & 16 & 16 & 24 & 20 & 36 & 32 \\
 5 &     & 4 & 4 & 8 & 12 & 12 & 16 & 24 & 28 & 36 & 44 & 52 & 68 & 80 & 92 & 116 \\
 7 &     & 2 & 4 & 6 & 6 & 10 & 14 & 16 & 20 & 26 & 34 & 40 & 48 & 60 & 72 & 88 \\
 10 &     &  & 4 & 4 & 8 & 8 & 16 & 12 & 24 & 24 & 36 & 40 & 56 & 56 & 80 & 88 \\
\bottomrule
\end{tabular}
\caption{Fourier coefficients $C^{(15+5)}(r^2-60n,r)$.}
\label{tab:desc:mjt-15+5}
\end{small}
\end{table}

\begin{table}[ht]
\begin{small}
\centering
\begin{tabular}{c|rrrrrrrrrrrrrrrr}\toprule
$r,n$&0&1&2&3&4&5&6&7&8&9&10&11&12&13&14&15\\\midrule
 1 & -2 & -2 & -2 & -2 & -2 & -6 & -4 & -6 & -6 & -8 & -10 & -14 & -12 & -18 & -18 & -22 \\
 3 &    & 2 & 4 & 2 & 6 & 6 & 8 & 8 & 12 & 12 & 18 & 16 & 24 & 26 & 32 & 32 \\
 4 &    & 2 & 2 & 4 & 4 & 4 & 8 & 8 & 8 & 12 & 14 & 16 & 20 & 22 & 24 & 32 \\
 7 &    & 2 & 2 & 4 & 4 & 6 & 6 & 8 & 10 & 14 & 14 & 16 & 20 & 26 & 26 & 34 \\
 8 &    & 2 & 4 & 4 & 6 & 8 & 8 & 12 & 14 & 16 & 20 & 24 & 28 & 32 & 40 & 44 \\
 11 &    &  & 2 & 0 & 4 & 2 & 4 & 2 & 8 & 4 & 10 & 8 & 12 & 10 & 18 & 14 \\
\bottomrule
\end{tabular}
\caption{Fourier coefficients $C^{(20+4)}(r^2-80n,r)$.}
\label{tab:desc:mjt-20+4}
\end{small}
\end{table}

\begin{table}[ht]
\begin{small}
\centering
\begin{tabular}{c|rrrrrrrrrrrrrrrr}\toprule
$r,n$&0&1&2&3&4&5&6&7&8&9&10&11&12&13&14&15\\\midrule
 1 & -2 & -2 & 0 & -2 & -4 & -2 & -4 & -6 & -6 & -8 & -8 & -8 & -12 & -14 & -14 & -18 \\
 2 &    & 2 & 2 & 4 & 2 & 4 & 4 & 8 & 6 & 10 & 8 & 12 & 14 & 18 & 18 & 24 \\
 3 &    & 2 & 0 & 2 & 4 & 4 & 4 & 4 & 8 & 8 & 8 & 12 & 12 & 14 & 16 & 20 \\
 4 &    & -2 & 0 & -2 & -2 & -2 & 0 & -4 & -2 & -4 & -4 & -6 & -4 & -10 & -8 & -10 \\
 5 &    & 2 & 2 & 2 & 4 & 4 & 4 & 6 & 6 & 8 & 10 & 12 & 14 & 16 & 18 & 20 \\
 6 &    & 2 & 4 & 4 & 6 & 4 & 8 & 8 & 12 & 12 & 16 & 16 & 22 & 24 & 32 & 32 \\
 8 &    & 2 & 0 & 2 & 2 & 2 & 4 & 6 & 4 & 6 & 8 & 8 & 10 & 14 & 12 & 16 \\
 9 &    & 2 & 4 & 4 & 4 & 8 & 8 & 10 & 12 & 14 & 20 & 20 & 24 & 28 & 32 & 40 \\
 11 &    &  & 0 & 2 & 2 & 0 & 2 & 2 & 2 & 4 & 2 & 4 & 6 & 6 & 6 & 8 \\
 12 &    &  & 4 & 2 & 6 & 4 & 8 & 8 & 12 & 12 & 16 & 16 & 24 & 24 & 32 & 32 \\
 15 &    &  &  & 2 & 4 & 4 & 4 & 6 & 8 & 8 & 8 & 12 & 16 & 18 & 20 & 24 \\
 18 &    &  &  & & 2 & 0 & 4 & 0 & 4 & 2 & 8 & 4 & 8 & 6 & 12 & 8 \\
\bottomrule
\end{tabular}
\caption{Fourier coefficients $C^{(21+3)}(r^2-84n,r)$.}
\label{tab:desc:mjt-21+3}
\end{small}
\end{table}

\begin{table}[ht]
\begin{small}
\centering
\begin{tabular}{c|rrrrrrrrrrrrrrrrrr}\toprule
$r,n$&0&1&2&3&4&5&6&7&8&9&10&11&12&13&14&15\\\midrule
 1 & -2 & -2 & 0 & -2 & -2 & -2 & -2 & -4 & -2 & -6 & -4 & -6 & -6 & -10 & -6 & -10 \\
 2 & & 2 & 2 & 2 & 4 & 4 & 4 & 6 & 6 & 8 & 10 & 10 & 12 & 14 & 16 & 18 \\
 5 & & 2 & 2 & 4 & 4 & 4 & 6 & 8 & 8 & 10 & 10 & 14 & 16 & 18 & 20 & 24 \\
 7 & & 0 & 2 & 0 & 2 & 2 & 4 & 0 & 4 & 4 & 6 & 4 & 8 & 4 & 10 & 8 \\
 8 & & 2 & 2 & 4 & 4 & 4 & 6 & 8 & 8 & 10 & 12 & 12 & 16 & 20 & 20 & 24 \\
 13 & & & 2 & 2 & 4 & 2 & 6 & 4 & 8 & 6 & 10 & 10 & 14 & 12 & 18 & 18 \\
\bottomrule
\end{tabular}
\caption{Fourier coefficients $C^{(24+8)}(r^2-96n,r)$.}
\label{tab:desc:mjt-24+8}
\end{small}
\end{table}

\begin{table}[ht]
\begin{small}
\centering
\begin{tabular}{c|rrrrrrrrrrrrrrrrrr}\toprule
$r,n$&0&1&2&3&4&5&6&7&8&9&10&11&12&13&14&15\\\midrule
 1 & -2 & 0 & -2 & -2 & -2 & -2 & -2 & -2 & -4 & -4 & -6 & -4 & -6 & -6 & -8 & -8 \\
 2 &  & 0 & 0 & 2 & 0 & 2 & 2 & 0 & 2 & 2 & 2 & 2 & 4 & 2 & 4 & 6 \\
 3 &  & 2 & 2 & 2 & 2 & 4 & 2 & 4 & 6 & 6 & 6 & 8 & 8 & 10 & 10 & 12 \\
 5 &  & -2 & 0 & 0 & 0 & -2 & 0 & -2 & 0 & -2 & -2 & -2 & 0 & -4 & -2 & -4 \\
 6 &  & 2 & 2 & 2 & 2 & 2 & 4 & 4 & 4 & 6 & 6 & 8 & 8 & 8 & 10 & 12 \\
 7 &  & 2 & 2 & 0 & 4 & 4 & 4 & 4 & 6 & 4 & 8 & 6 & 8 & 12 & 12 & 12 \\
 9 &  & 0 & 0 & 2 & 0 & 0 & 0 & 2 & 0 & 2 & 2 & 2 & 2 & 2 & 2 & 4 \\
 10 &  & 2 & 2 & 2 & 4 & 4 & 4 & 6 & 6 & 6 & 8 & 10 & 10 & 12 & 14 & 16 \\
 13 &  &  & 2 & 0 & 2 & 0 & 4 & 2 & 4 & 2 & 4 & 2 & 6 & 4 & 8 & 6 \\
 14 &  &  & 2 & 4 & 2 & 4 & 4 & 4 & 6 & 8 & 8 & 8 & 12 & 12 & 14 & 16 \\
 17 &  &  &  & 2 & 2 & 2 & 2 & 2 & 4 & 4 & 4 & 6 & 6 & 6 & 8 & 10 \\
 21 &  &  &  &  & 2 & 0 & 4 & 2 & 4 & 2 & 4 & 4 & 8 & 4 & 8 & 8 \\
\bottomrule
\end{tabular}
\caption{Fourier coefficients $C^{(28+7)}(r^2-112n,r)$.}
\label{tab:desc:mjt-28+7}
\end{small}
\end{table}

\begin{table}[ht]
\begin{small}
\centering
\begin{tabular}{c|rrrrrrrrrrrrrrrrrr}\toprule
$r,n$&0&1&2&3&4&5&6&7&8&9&10&11&12&13&14&15\\\midrule
   1 & -2 & 0 & 0 & -2 & -2 & -2 & 0 & -2 & -2 & -2 & -2 & -2 & -4 & -4 & -4 & -4 \\
 2 &   & 2 & 0 & 2 & 2 & 2 & 2 & 4 & 2 & 4 & 4 & 4 & 4 & 6 & 6 & 8 \\
 4 &   & -2 & 0 & 0 & 0 & 0 & 0 & -2 & 0 & -2 & 0 & -2 & 0 & -2 & 0 & -2 \\
 5 &   & 2 & 2 & 2 & 2 & 2 & 2 & 4 & 4 & 4 & 6 & 6 & 6 & 8 & 8 & 8 \\
 7 &   & 0 & 0 & 0 & 0 & 2 & 0 & 0 & 2 & 0 & 0 & 2 & 0 & 0 & 2 & 2 \\
 8 &   & 2 & 2 & 2 & 2 & 2 & 4 & 4 & 4 & 6 & 4 & 6 & 8 & 8 & 8 & 10 \\
 10 &   & 0 & 2 & 0 & 2 & 0 & 2 & 0 & 2 & 2 & 4 & 2 & 4 & 2 & 4 & 4 \\
 11 &   & 2 & 0 & 2 & 4 & 2 & 4 & 4 & 4 & 4 & 4 & 6 & 8 & 8 & 8 & 10 \\
 13 &   & & 2 & 2 & 0 & 2 & 2 & 2 & 2 & 4 & 4 & 4 & 4 & 4 & 6 & 6 \\
 16 &   &  & 2 & 0 & 2 & 2 & 4 & 2 & 4 & 2 & 4 & 4 & 6 & 4 & 8 & 6 \\
 19 &   &  &  & 2 & 2 & 2 & 2 & 2 & 4 & 4 & 4 & 4 & 6 & 6 & 6 & 8 \\
 22 &   &  &  & & 2 & 2 & 4 & 0 & 4 & 2 & 4 & 4 & 6 & 4 & 8 & 8  \\\bottomrule
\end{tabular}
\caption{Fourier coefficients $C^{(33+11)}(r^2-132n,r)$.}
\label{tab:desc:mjt-33+11}
\end{small}
\end{table}

\begin{table}[ht]
\begin{small}
\centering
\begin{tabular}{c|rrrrrrrrrrrrrrrrrr}\toprule
$r,n$&0&1&2&3&4&5&6&7&8&9&10&11&12&13&14&15\\\midrule
 1 & -2 & -2 & -2 & 0 & -2 & -2 & -2 & -2 & -2 & -2 & -4 & -2 & -4 & -4 & -4 & -4 \\
 3 & & 2 & 0 & 2 & 2 & 2 & 2 & 2 & 2 & 4 & 4 & 4 & 4 & 6 & 4 & 6 \\
 4 & & 0 & 2 & 0 & 0 & 2 & 2 & 0 & 2 & 2 & 2 & 2 & 2 & 2 & 4 & 4 \\
 5 & & 0 & 0 & 0 & 0 & -2 & 0 & 0 & 0 & -2 & -2 & -2 & 0 & -2 & -2 & -2 \\
 7 & & 0 & 2 & 0 & 2 & 2 & 2 & 2 & 4 & 2 & 4 & 4 & 6 & 4 & 6 & 6 \\
 8 & & 2 & 0 & 2 & 2 & 2 & 2 & 2 & 2 & 4 & 4 & 4 & 4 & 6 & 4 & 6 \\
 11 & & 2 & 2 & 2 & 2 & 2 & 2 & 4 & 2 & 4 & 4 & 4 & 4 & 6 & 6 & 6 \\
 12 & & 1 & 2 & 2 & 2 & 2 & 2 & 4 & 4 & 4 & 4 & 4 & 6 & 6 & 8 & 8 \\
 15 & &  & 2 & 0 & 2 & 2 & 2 & 0 & 4 & 2 & 4 & 2 & 4 & 4 & 6 & 4 \\
 16 & &  & 2 & 2 & 2 & 2 & 4 & 2 & 4 & 4 & 4 & 6 & 6 & 6 & 8 & 8 \\
 19 & &  &  & 0 & 0 & 2 & 0 & 0 & 2 & 2 & 2 & 2 & 2 & 2 & 2 & 4 \\
 23 & &  &  &  & 2 & 0 & 2 & 0 & 2 & 0 & 2 & 0 & 2 & 2 & 2 & 0 \\\bottomrule
\end{tabular}
\caption{Fourier coefficients $C^{(36+4)}(r^2-144n,r)$.}
\label{tab:desc:mjt-36+4}
\end{small}
\end{table}

\begin{table}[ht]
\begin{small}
\centering
\begin{tabular}{c|rrrrrrrrrrrrrrrrrr}\toprule
$r,n$&0&1&2&3&4&5&6&7&8&9&10&11&12&13&14&15\\\midrule
 1 & -2 & -2 & -2 & 0 & -2 & -2 & -2 & -2 & -2 & -2 & -2 & -2 & -4 & -4 & -4 & -2 \\
 2 &   & 2 & 0 & 2 & 2 & 0 & 2 & 2 & 2 & 2 & 2 & 2 & 2 & 4 & 2 & 4 \\
 7 &   & 0 & 2 & 0 & 0 & 2 & 2 & 0 & 2 & 0 & 2 & 2 & 2 & 2 & 2 & 2 \\
 11 &   & 2 & 0 & 2 & 2 & 2 & 0 & 2 & 2 & 4 & 2 & 2 & 2 & 4 & 2 & 4 \\
 13 &   & -2 & 0 & -2 & 0 & -2 & 0 & -2 & 0 & -2 & -2 & -2 & 0 & -4 & 0 & -2 \\
 14 &   & 2 & 2 & 2 & 2 & 2 & 2 & 4 & 2 & 2 & 4 & 4 & 4 & 4 & 4 & 4 \\\bottomrule
\end{tabular}
\caption{Fourier coefficients $C^{(60+12,15,20)}(r^2-240n,r)$.}
\label{tab:desc:mjt-60+12,15,20}
\end{small}
\end{table}

\clearpage

It remains to discuss the three lambencies in the last line of (\ref{eqn:class:proof-gzgps-emum}). For each of these we have $\f^{(\ell)}=2\mathcal{Q}_m$ where $m$ is the level of $\ell$, and $\mathcal{Q}_m$ is as discussed in \cite{Dabholkar:2012nd}. Low order Fourier coefficients of these mock modular forms are given in Tables \ref{tab:desc:mjt-42+6,14,21} through \ref{tab:desc:mjt-78+6,26,39}.

\begin{table}[ht]
\begin{small}
\centering
\begin{tabular}{c|rrrrrrrrrrrrrrrrrr}\toprule
$r,n$&0&1&2&3&4&5&6&7&8&9&10&11&12&13&14&15\\\midrule
 1 & -2 & -2 & 0 & -2 & -2 & -2 & 0 & -4 & -2 & -4 & -2 & -4 & -2 & -6 & -4 & -6 \\
 5 & & 2 & 2 & 2 & 4 & 2 & 4 & 4 & 4 & 6 & 6 & 4 & 8 & 8 & 8 & 8 \\
 11 &  & 2 & 2 & 4 & 2 & 4 & 4 & 6 & 4 & 6 & 6 & 8 & 8 & 10 & 8 & 12 \\\bottomrule
\end{tabular}
\caption{Fourier coefficients $C^{(42+6,14,21)}(r^2-168n,r)$.}
\label{tab:desc:mjt-42+6,14,21}
\end{small}
\end{table}

\begin{table}[ht]
\begin{small}
\centering
\begin{tabular}{c|rrrrrrrrrrrrrrrrrr}\toprule
$r,n$&0&1&2&3&4&5&6&7&8&9&10&11&12&13&14&15\\\midrule
 1 & -2 & 0 & -2 & -2 & 0 & 0 & -2 & 0 & -2 & -2 & -2 & -2 & -2 & 0 & -2 & -2 \\
 3 &  & 0 & 2 & 0 & 2 & 0 & 0 & 2 & 2 & 0 & 2 & 2 & 2 & 0 & 2 & 2 \\
 9 &  & 2 & 0 & 2 & 0 & 2 & 2 & 2 & 0 & 2 & 2 & 2 & 2 & 2 & 2 & 4 \\
 11 &  & -2 & 0 & 0 & 0 & -2 & 0 & -2 & 0 & 0 & 0 & -2 & 0 & -2 & 0 & -2 \\
 13 &  & 2 & 2 & 0 & 2 & 2 & 2 & 2 & 2 & 2 & 2 & 2 & 2 & 4 & 4 & 2 \\
 23 &  &  & 2 & 2 & 2 & 2 & 2 & 2 & 2 & 2 & 4 & 2 & 4 & 2 & 4 & 4 \\\bottomrule
\end{tabular}
\caption{Fourier coefficients $C^{(70+10,14,35)}(r^2-280n,r)$.}
\label{tab:desc:mjt-70+10,14,35}
\end{small}
\end{table}

\begin{table}[ht]
\begin{small}
\centering
\begin{tabular}{c|rrrrrrrrrrrrrrrrrr}\toprule
$r,n$&0&1&2&3&4&5&6&7&8&9&10&11&12&13&14&15\\\midrule
 1 & -2 & 0 & 0 & -2 & -2 & 0 & 0 & -2 & 0 & 0 & -2 & 0 & -2 & -2 & -2 & -2 \\
 5 &  & 2 & 0 & 2 & 0 & 2 & 0 & 2 & 2 & 2 & 0 & 2 & 2 & 2 & 2 & 2 \\
 7 &  & -2 & 0 & 0 & 0 & 0 & 0 & -2 & 0 & -2 & 0 & 0 & 0 & -2 & 0 & 0 \\
 11 &  & 2 & 2 & 0 & 2 & 0 & 2 & 2 & 2 & 2 & 2 & 2 & 2 & 2 & 2 & 2 \\
 17 &  & 2 & 0 & 2 & 2 & 2 & 2 & 2 & 0 & 2 & 2 & 2 & 2 & 4 & 2 & 4 \\
 23 &  & 0 & 2 & 2 & 2 & 0 & 2 & 2 & 2 & 2 & 2 & 2 & 4 & 2 & 2 & 2 \\\bottomrule
\end{tabular}
\caption{Fourier coefficients $C^{(78+6,26,39)}(r^2-312n,r)$.}
\label{tab:desc:mjt-78+6,26,39}
\end{small}
\end{table}

\subsection{Mock Theta Functions of Ramanujan}
\label{sec:mockthetafunctions} 

Essentially all of the mock theta functions of Ramanujan \cite{MR947735,MR2280843} can be recovered  from the theta-coefficients of optimal mock Jacobi forms with rational coefficients. In this section we give explicit expressions for every such function, in terms of the $H^{(\ell)}_r$ for $\ell\in \gt{L}_1$. We also consider the even order mock theta functions introduced by Andrews \cite{MR654518} and Gordon--McIntosh \cite{Gordon_Mcintosh}.

Most of the expressions are given in Tables \ref{tbl:mock theta ramunujan} and \ref{tbl:mock theta A GM}. To explain the notation, Recall the {\em $q$-Pochhamer symbol}
\begin{gather}
(a;x)_n := \prod_{k=0}^{n-1}(1-ax^k),\quad(a;x)_\infty:=\prod_{k\geq 0}(1-ax^k),
\end{gather}
and note that we omit the summation symbol $\sum_{n\geq 0}$ from the entries in the ``Eulerian Series'' columns. So, for instance, the first row of Table \ref{tbl:mock theta ramunujan} indicates that Ramanujan's order 3 mock theta function $\psi(q)$ may be defined by setting
\begin{gather}
	\psi(q):=\sum_{n\geq 0}\frac{q^{(n+1)^2}}{(q;q^2)_{n+1}}.
\end{gather}

In the ``Mock Jacobi Theta Functions'' columns we use operators $f\mapsto f|[a;b]$, for $a,b\in \QQ$. Assuming that $f$ is a holomorphic function on $\HH$ satisfying $f(\tau+N)=f(\tau)$ for some $N\in\ZZ$ we define 
\be
f\lvert[a;b](\t):=   \sum_{n={a}\xmod{b\ZZ}}  c_f(n) \, q^{n-a}
\ee 
when $f(\tau)=\sum_{n}c_f(n)q^n$.
To avoid clutter we set $f|[a]:=f|[a;1]$. For $\ell\in \gt{L}_1$ we write $H^{(\ell)}_r$ for the theta-coefficients of $\phi^{(\ell)}$, as in \S\ref{sec:desc:mjt}, so that $\phi^{(\ell)}=\sum_r H^{(\ell)}_r\theta_{m,r}$.

\vspace{-4pt}

\begin{center}
\begin{small}
\begin{longtable}{LLLL}
\toprule 
\rm{Order} & \rm{Name} & \text{Eulerian Series} & \text{Mock Jacobi Theta Functions} \\\midrule
\multirow{11}*{3} &\psi(q)&\displaystyle\frac{q^{(n+1)^2}}{(q;q^2)_{n+1}}& {\frac12}\pwrm{H_2^{(24+8)}}{1}{24}(\t)\\[5pt]
& \nu(q) &\displaystyle\frac{q^{n(n+1)}}{(-q;q^2)_{n+1}} & {\frac12 }\pwr{H_8^{(24+8)}}{1}{3} (\t-\tfrac{1}{2})
 \\[5pt] 
&  f(q)& \displaystyle\frac{q^{n^2}}{(-q;q)^2_{n}}  &  {\frac12 }\pwrm{(H_{5}^{(6)}-H_{1}^{(6)})}{1}{24}(\t)
\\[5pt] 
&\phi(q) &    \displaystyle\frac{q^{n^2}}{(-q^2;q^2)_{n}} &{\frac12 } \sum_{k\in\ZZ/4}(-1)^{k+1}H_{1+12k}^{(24+8)}\lvert[-\tfrac{1}{96},\tfrac{1}{4}] (4\t) \\[5pt]
&\chi(q) &   q^{n^2}\displaystyle\frac{(-q;q)_{n}}{(-q^3;q^3)_{n}} &-{\frac12 }\sum_{k\in\ZZ/3}(H_{1+12k}^{(18)}+H_{7+12k}^{(18)})\lvert[-\tfrac{1}{72};\tfrac{1}{3}] (3\t)
\\[5pt]
&\omega(q) &  \displaystyle\frac{q^{2n(n+1)}}{(q;q^2)^2_{n+1}} &{\frac14} (H^{(6)}_2+H^{(6)}_4)\lvert[\tfrac{1}{3};\tfrac{1}{2}](2\t)
 \\[5pt]
&\rho(q) &   q^{2n(n+1)}\displaystyle\frac{(q;q^2)_{n+1}}{(q^3;q^6)_{n+1}}& -{\frac12 }\sum_{k\in\ZZ/3}(H_{2+12k}^{(18)}+H_{4+12k}^{(18)})\lvert[\tfrac{1}{9};\tfrac{1}{6}] (6\t)
\\[7pt]\midrule
\multirow{11}*{5} &\psi_0(q) & (-q;q)_n\, q^{(n+1)(n+2)/2}  & {\frac12}\pwrm{ H^{(60+12,15,20)}_2}{1}{60}(\t)\\[5pt]
&\psi_1(q) & (-q;q)_n\, q^{n(n+1)/2} & {\frac12}\pwr{H^{(60+12,15,20)}_{14}}{11}{60}(\t)
   \\[5pt]
  &\chi_0(q)&  \displaystyle\frac{q^{n} }{(q^{n+1};q)_n}& {\frac12}\pwrm{H^{(30+6,10,15)}_{1}}{1}{120}(\t)+2 \\[5pt]
&\chi_1(q)& \displaystyle\frac{q^{n} }{(q^{n+1};q)_{n+1}} &{\frac12} \pwr{H^{(30+6,10,15)}_{7}}{71}{120}(\t)\\[5pt]
&\phi_0(q)&   q^{n^2} {(-q;q^2)_n}&
 -{\frac12}(H^{(60+12,15,20)}_1-H^{(60+12,15,20)}_{11})\lvert[-\tfrac{1}{240};\tfrac{1}{2}](2\t)  \\[5pt]
&\phi_1(q)& q^{(n+1)^2} {(-q;q^2)_n}&
{\frac12} (H^{(60+12,15,20)}_7-H^{(60+12,15,20)}_{13})\lvert[-\tfrac{49}{240};\tfrac{1}{2}](2\t)\\[5pt]
&F_0(q)&\displaystyle\frac{q^{2n^2}}{(q;q^2)_n}& 1+{\frac12} H^{(60+12,15,20)}_{2}\lvert[-\tfrac{1}{60};2](\tfrac{\t}{2})\\[5pt]
&F_1(q)& \displaystyle\frac{q^{2n^2+2n}}{(q;q^2)_{n+1}}&{\frac12} H^{(60+12,15,20)}_{14}\lvert[\tfrac{71}{60};2](\tfrac{\t}{2})\\[7pt]
\midrule
\multirow{6}*{6}& \sigma(q) &  \displaystyle\frac{q^{(n+1)(n+2)/2} (-q;q)_n}{(q;q^2)_{n+1}} & {\frac12 }\pwrm{ H_{2}^{(12)}}{1}{12}(\t)\\[5pt]
&\psi(q)& \displaystyle\frac{(-1)^n q^{(n+1)^2}(q;q^2)_n}{(-q;q)_{2n+1}} & -{\frac12 }(H_{3}^{(12)}-H_{9}^{(12)})\lvert[-\tfrac{3}{8};\tfrac{1}{2}](2\t)\\[5pt]
&\phi(q)& \displaystyle\frac{(-1)^n q^{n^2}(q;q^2)_n}{(-q;q)_{2n}} & -{\frac12 }\sum_{k\in \ZZ/2} (H_{1+12k}^{(12)}+H_{5+12k}^{(12)})\lvert[-\tfrac{1}{48};\tfrac{1}{2}](2\t)
 \\[5pt]
&\gamma(q)&\displaystyle\frac{ q^{n^2}(q;q)_n}{(q^3;q^3)_n} & -{\frac12 }\sum_{k\in \ZZ/3} (H_{1+12k}^{(18)}+H_{5+12k}^{(18)})\lvert[-\tfrac{1}{72};\tfrac{1}{3}]  (3\t)
\\[7pt]\midrule
\multirow{5}*{7}&F_0(q) &\displaystyle\frac{q^{n^2}}{(q^{n+1};q)_n}  &-{\frac12}\pwrm{H^{(42+6,14,21)}_{1}}{1}{168}(\t)
  \\[5pt]
&F_1(q) & \displaystyle\frac{q^{(n+1)^2}}{(q^{n+1};q)_{n+1}}&{\frac12}\pwrm{ H^{(42+6,14,21)}_{5}}{25}{168}(\t)\\[5pt]
&F_2(q) &\displaystyle\frac{q^{n^2+n}}{(q^{n+1};q)_{n+1}} &{\frac12} \pwr{H^{(42+6,14,21)}_{11}}{47}{168}(\t) \\[7pt]\midrule
\multirow{6}*{10}&
\phi(q) & \displaystyle\frac{q^{n(n+1)/2}}{(q;q^2)_{n+1}} &{\frac12}  \sum_{k\in \ZZ/2} (-1)^{k}H^{(10)}_{4+10k}\lvert[\tfrac{1}{10};\tfrac{1}{2}] (2\t)
\\[5pt]
&\psi(q) & \displaystyle\frac{q^{(n+1)(n+2)/2}}{(q;q^2)_{n+1}}&{\frac12} \sum_{k\in \ZZ/2}  (-1)^{k} H^{(10)}_{2+10k}\lvert[-\tfrac{1}{10};\tfrac{1}{2}]  (2\t) 
\\[5pt]
&X(q) & \displaystyle\frac{(-1)^n q^{n^2}}{(-q;q)_{2n}}  & -{\frac12}  \sum_{k\in \ZZ/2} \pwrm{H^{(10)}_{1+10k}}{1}{40} (\t)\\[5pt]
&\chi(q) & \displaystyle\frac{(-1)^n q^{(n+1)^2}}{(-q;q)_{2n+1}} &-{\frac12}  \sum_{k\in \ZZ/2} \pwrm{H^{(10)}_{3+10k}}{9}{40} (\t) \\[7pt]\bottomrule\\[-22pt]
\caption{Mock theta functions of Ramanujan in terms of optimal mock Jacobi theta functions.\label{tbl:mock theta ramunujan}}
\end{longtable}
\end{small}
\end{center}

\vspace{-31pt}

In addition to the mock theta functions given in Table \ref{tbl:mock theta ramunujan}, Ramanujan also considered the series
\begin{gather}
	\begin{split}\label{eqn:desc:ram-fi}
	f_0(q)&:=\sum_{n\geq 0}\frac{q^{n^2}}{(-q;q)_n},\\
	f_1(q)&:=\sum_{n\geq 0}\frac{q^{n(n+1)}}{(-q;q)_n},
	\end{split}
\end{gather}
which he included amongst his mock theta functions of order 5 (cf. \cite{MR2280843}). In \cite{MR947735} the series 
\begin{gather}
\begin{split}
\rho(q)&:= \sum_{n\geq 0}  \frac{q^{n(n+1)/2} (-q;q)_n}{(q;q^2)_{n+1}},\\
\lambda(q)&:= \sum_{n\geq 0}  \frac{(-1)^n q^{n} (q;q^2)_n}{(-q;q)_{n}},\\\label{order6mock}
2\mu(q)&:=1+ \sum_{n\geq 0}  \frac{(-1)^n q^{n+1} (1+q^n)(q;q^2)_n}{(-q;q)_{n+1}}, 
\end{split}
\end{gather}
appear, which have since been assigned the order 6.
For the $f_i$ in (\ref{eqn:desc:ram-fi}) we obtain mock Jacobi theta function expressions by noting that 
\begin{gather}
	\begin{split}\label{eqn:desc:rmt-f0f1id}
f_0(q) &= -\psi_0(-q) + \phi_0(-q^2),\\
f_1(q) &= \psi_1(-q) -q^{-1} \phi_1(-q^2), 
	\end{split}
\end{gather}
according to identities\footnote{These identities are summarized in a single matrix identity in \S1 of \cite{zagier_mock}. Note that the second and sixth columns in the matrix on the left-hand side of that identity should be multiplied by $-1$.} found by Ramanujan \cite{MR2280843} and proven by Watson \cite{MR1577032}, 
where the $\psi_i$ and $\phi_i$ in (\ref{eqn:desc:rmt-f0f1id}) are as in Table \ref{tbl:mock theta ramunujan}. For the order 6 functions in (\ref{order6mock}) we have Ramanujan's identities (cf. \cite{MR947735})
\begin{gather}
\begin{split}
q^{-1}\psi(q^2)+\rho(q)=2q^{-1} \psi(q^2)+\lambda(-q)&= 
	(-q;q^2)_\infty^2(-q;q^6)_\infty(-q^5;q^6)_\infty(q^6;q^6)_\infty
\\  \label{ord6}
\phi(q^2)+2\s(q) = 2 \phi(q^2)-2\mu(-q) &= 
	(-q;q^2)_\infty^2(-q^3;q^6)_\infty^2(q^6;q^6)_\infty
\end{split}
\end{gather}
which were proven by Andrews--Hickerson \cite{MR1123099}. These identities allow us to write $\mu$ in terms of the order 6 functions $\phi$ and $\sigma$, which are treated in Table \ref{tbl:mock theta ramunujan}, and show that $\rho$ and $\lambda$ each differ from the order 6 function $\psi$ (also appearing in Table \ref{tbl:mock theta ramunujan}) only by a theta function.

\vspace{-8pt}

\begin{center}

\begin{small}
\begin{longtable}{LLLL}
\toprule 
\rm{Order} & \rm{Name} & \text{Eulerian Series} & \text{Mock Jacobi Theta Functions} \\\midrule
\multirow{4}*{2} & \m(q)&\displaystyle\frac{(-1)^n\,q^{n^2}(q;q^2)_n}{(-q^2;q^2)_n^2}& -{\frac12 } \sum_{k\in\ZZ/4}{H_{1+4k}^{(8)}}[-\tfrac{1}{32};\tfrac{1}{4}](4\t)\\[5pt]
&A(q) &\displaystyle\frac{q^{(n+1)} (-q^2;q^2)_n}{(q;q^2)_{n+1}} & {\frac14 }\pwrm{H_{2}^{(8)}}{1}{8}(\t)
\\[5pt]
&B(q) & \displaystyle\frac{q^{n} (-q;q^2)_n}{(q;q^2)_{n+1}}&{\frac14 }\pwr{ H_{4}^{(8)}}{1}{2}(\t) \\[7pt]\midrule
\multirow{11}*{8}&S_0(q) &  \displaystyle\frac{q^{n^2} (-q;q^2)_n}{(-q^2;q^2)_n}&-{\frac12}  \sum_{k\in \ZZ/4} H^{(16)}_{1+8k}\lvert[-\tfrac{1}{64};\tfrac{1}{4}](4\t)
 \\[5pt]
&S_1(q) &  \displaystyle\frac{q^{n(n+2)} (-q;q^2)_n}{(-q^2;q^2)_n} &
{\frac12} \sum_{k\in \ZZ/4} H^{(16)}_{3+8k} \lvert[-\tfrac{7}{64};\tfrac{1}{4}](4\t)
\\[5pt]
&T_0(q) &  \displaystyle\frac{q^{(n+1)(n+2)} (-q^2;q^2)_n}{(-q;q^2)_{n+1}}&{\frac12}   \pwrm{H^{(16)}_{2}}{1}{16}(\t+\tfrac{1}{2})  \\[5pt]
&T_1(q) &  \displaystyle\frac{q^{n(n+1)} (-q^2;q^2)_n}{(-q;q^2)_{n+1}} & {\frac12}\pwr{H^{(16)}_{10}}{7}{16}(\t+\tfrac{1}{2})
\\[5pt]
&U_0 (q)&  \displaystyle\frac{q^{n^2}(-q;q^2)_n}{(-q^4;q^4)_n}&-{\frac12} \sum_{k\in \ZZ/8} H^{(16)}_{1+4k}\lvert[-\tfrac{1}{64};\tfrac{1}{8}](8\t) \\[5pt]
&U_1 (q)&  \displaystyle\frac{q^{(n+1)^2}(-q;q^2)_n}{(-q^2;q^4)_n}&{\frac12}\left(H^{(16)}_{2}+\ex(-\tfrac{1}{4})H^{(16)}_{10} \right)\lvert[-\tfrac{1}{16};\tfrac{1}{2}](2\t+\tfrac{1}{2})
\\[5pt]
&1+V_0 (q) &  2q^{n^2} \displaystyle\frac{ (-q;q^2)_n}{(q;q^2)_{n}}&1+ H^{(16)}_{8}(\t)\\[5pt]
&V_1 (q) &   \displaystyle\frac{q^{(n+1)^2} (-q;q^2)_n}{(q;q^2)_{n+1}}&{\frac12} \pwrm{H^{(16)}_{4}}{1}{4}(\t)
\\[7pt]\bottomrule\\[-18pt]
\caption{ Mock theta functions of Andrews and Gordon--McIntosh in terms of optimal mock Jacobi theta functions. \label{tbl:mock theta A GM}}
\end{longtable}
\end{small}
\end{center}

\vspace{-26pt}

The order 2 mock theta functions in Table \ref{tbl:mock theta A GM} were introduced by Andrews in \cite{MR654518}, and subsequently studied by McIntosh \cite{MR2317449}. The particular function $\mu$ appears a number of times in Ramanujan's work \cite{MR947735}. The order 8 mock theta functions were introduced by Gordon--McIntosh \cite{Gordon_Mcintosh}. Subsequently it was noticed that $U_0$ and $V_1$ also appear in \cite{MR947735}. We refer to \cite{GorMcI_SvyMckTht} for more background on mock theta functions.

%------------------------------------------------------------------%

%------------------------------------------------------------------%
\addcontentsline{toc}{section}{References}
%\bibliographystyle{alpha}
%\bibliography{../../bib/master}

\begin{thebibliography}{DGO15b}

\bibitem[AH91]{MR1123099}
George~E. Andrews and Dean Hickerson.
\newblock Ramanujan's ``lost'' notebook. {VII}. {T}he sixth order mock theta
  functions.
\newblock {\em Adv. Math.}, 89(1):60--105, 1991.

\bibitem[AL70]{MR0268123}
A.~O.~L. Atkin and J.~Lehner.
\newblock Hecke operators on {$\Gamma _{0}(m)$}.
\newblock {\em Math. Ann.}, 185:134--160, 1970.

\bibitem[And81]{MR654518}
George~E. Andrews.
\newblock Mordell integrals and {R}amanujan's ``lost'' notebook.
\newblock In {\em Analytic number theory ({P}hiladelphia, {P}a., 1980)}, volume
  899 of {\em Lecture Notes in Math.}, pages 10--18. Springer, Berlin-New York,
  1981.

\bibitem[BF04]{BruFun_TwoGmtThtLfts}
Jan~Hendrik Bruinier and Jens Funke.
\newblock On two geometric theta lifts.
\newblock {\em Duke Math. J.}, 125(1):45--90, 2004.

\bibitem[BFH90]{MR1074487}
Daniel Bump, Solomon Friedberg, and Jeffrey Hoffstein.
\newblock Nonvanishing theorems for {$L$}-functions of modular forms and their
  derivatives.
\newblock {\em Invent. Math.}, 102(3):543--618, 1990.

\bibitem[BO10]{MR2726107}
Jan Bruinier and Ken Ono.
\newblock Heegner divisors, {$L$}-functions and harmonic weak {M}aass forms.
\newblock {\em Ann. of Math. (2)}, 172(3):2135--2181, 2010.

\bibitem[Bor98]{Bor_AutFmsSngGrs}
Richard~E. Borcherds.
\newblock Automorphic forms with singularities on {G}rassmannians.
\newblock {\em Invent. Math.}, 132(3):491--562, 1998.

\bibitem[BR10]{MR2680205}
Kathrin Bringmann and Olav~K. Richter.
\newblock Zagier-type dualities and lifting maps for harmonic {M}aass-{J}acobi
  forms.
\newblock {\em Adv. Math.}, 225(4):2298--2315, 2010.

\bibitem[BR11]{MR2805582}
Kathrin Bringmann and Olav~K. Richter.
\newblock Exact formulas for coefficients of {J}acobi forms.
\newblock {\em Int. J. Number Theory}, 7(3):825--833, 2011.

\bibitem[Bru98]{MR1658385}
J.~H. Bruinier.
\newblock On a theorem of {V}ign\'eras.
\newblock {\em Abh. Math. Sem. Univ. Hamburg}, 68:163--168, 1998.

\bibitem[Bru02]{MR1903920}
Jan~H. Bruinier.
\newblock {\em Borcherds products on {O}(2, {$l$}) and {C}hern classes of
  {H}eegner divisors}, volume 1780 of {\em Lecture Notes in Mathematics}.
\newblock Springer-Verlag, Berlin, 2002.

\bibitem[BS17]{Bruinier201738}
Jan~Hendrik Bruinier and Markus Schwagenscheidt.
\newblock Algebraic formulas for the coefficients of mock theta functions and
  weyl vectors of borcherds products.
\newblock {\em Journal of Algebra}, 478:38 -- 57, 2017.

\bibitem[CDH14a]{UM}
Miranda C.~N. Cheng, John F.~R. Duncan, and Jeffrey~A. Harvey.
\newblock {Umbral Moonshine}.
\newblock {\em Commun. Number Theory Phys.}, 8(2):101--242, 2014.

\bibitem[CDH14b]{MUM}
Miranda C.~N. Cheng, John F.~R. Duncan, and Jeffrey~A. Harvey.
\newblock {Umbral Moonshine and the Niemeier Lattices}.
\newblock {\em Research in the Mathematical Sciences}, 1(3):1--81, 2014.

\bibitem[CN79]{MR554399}
J.~H. Conway and S.~P. Norton.
\newblock Monstrous moonshine.
\newblock {\em Bull. London Math. Soc.}, 11(3):308--339, 1979.

\bibitem[DGO15a]{mnstmlts}
J.~F.~R. {Duncan}, M.~J. {Griffin}, and K.~{Ono}.
\newblock {Moonshine}.
\newblock {\em Research in the Mathematical Sciences}, 2(11), 2015.

\bibitem[DGO15b]{umrec}
J.~F.~R. {Duncan}, M.~J. {Griffin}, and K.~{Ono}.
\newblock {Proof of the Umbral Moonshine Conjecture}.
\newblock {\em Research in the Mathematical Sciences}, 2(26), March 2015.

\bibitem[DMZ12]{Dabholkar:2012nd}
Atish Dabholkar, Sameer Murthy, and Don Zagier.
\newblock {Quantum Black Holes, Wall Crossing, and Mock Modular Forms}.
\newblock 2012.

\bibitem[Duk14]{MR3242661}
W.~Duke.
\newblock Almost a century of answering the question: what is a mock theta
  function?
\newblock {\em Notices Amer. Math. Soc.}, 61(11):1314--1320, 2014.

\bibitem[EZ85]{eichler_zagier}
Martin Eichler and Don Zagier.
\newblock {\em {The theory of Jacobi forms}}.
\newblock Birkh{\"a}user, 1985.

\bibitem[Fer93]{Fer_Genus0prob}
Charles~R. Ferenbaugh.
\newblock The genus-zero problem for {$n\vert h$}-type groups.
\newblock {\em Duke Math. J.}, 72(1):31--63, 1993.

\bibitem[Fol10]{Folsom_what}
Amanda Folsom.
\newblock What is {$\dots$} a mock modular form?
\newblock {\em Notices Amer. Math. Soc.}, 57(11):1441--1443, 2010.

\bibitem[Gan16]{MR3539377}
Terry Gannon.
\newblock Much ado about {M}athieu.
\newblock {\em Adv. Math.}, 301:322--358, 2016.

\bibitem[GKZ87]{MR909238}
B.~Gross, W.~Kohnen, and D.~Zagier.
\newblock Heegner points and derivatives of {$L$}-series. {II}.
\newblock {\em Math. Ann.}, 278(1-4):497--562, 1987.

\bibitem[GM00]{Gordon_Mcintosh}
Basil Gordon and Richard~J. McIntosh.
\newblock Some eighth order mock theta functions.
\newblock {\em J. London Math. Soc. (2)}, 62(2):321--335, 2000.

\bibitem[GM12]{GorMcI_SvyMckTht}
Basil Gordon and Richard~J. McIntosh.
\newblock A survey of classical mock theta functions.
\newblock In Krishnaswami Alladi and Frank Garvan, editors, {\em Partitions,
  q-Series, and Modular Forms}, volume~23 of {\em Developments in Mathematics},
  pages 95--144. Springer New York, 2012.
\newblock 10.1007/978-1-4614-0028-8$_9$.

\bibitem[GZ86]{MR833192}
Benedict~H. Gross and Don~B. Zagier.
\newblock Heegner points and derivatives of {$L$}-series.
\newblock {\em Invent. Math.}, 84(2):225--320, 1986.

\bibitem[KL89]{MR1036843}
V.~A. Kolyvagin and D.~Yu. Logach{\"e}v.
\newblock Finiteness of the {S}hafarevich-{T}ate group and the group of
  rational points for some modular abelian varieties.
\newblock {\em Algebra i Analiz}, 1(5):171--196, 1989.

\bibitem[Koh80]{MR575942}
Winfried Kohnen.
\newblock Modular forms of half-integral weight on {$\Gamma _{0}(4)$}.
\newblock {\em Math. Ann.}, 248(3):249--266, 1980.

\bibitem[Koh82]{MR660784}
Winfried Kohnen.
\newblock Newforms of half-integral weight.
\newblock {\em J. Reine Angew. Math.}, 333:32--72, 1982.

\bibitem[Koh87]{MR927162}
W.~Kohnen.
\newblock A simple remark on eigenvalues of {H}ecke operators on {S}iegel
  modular forms.
\newblock {\em Abh. Math. Sem. Univ. Hamburg}, 57:33--35, 1987.

\bibitem[McI07]{MR2317449}
Richard~J. McIntosh.
\newblock Second order mock theta functions.
\newblock {\em Canad. Math. Bull.}, 50(2):284--290, 2007.

\bibitem[Niw75]{MR0364106}
Shinji Niwa.
\newblock Modular forms of half integral weight and the integral of certain
  theta-functions.
\newblock {\em Nagoya Math. J.}, 56:147--161, 1975.

\bibitem[Niw77]{MR0562506}
Shinji Niwa.
\newblock On {S}himura's trace formula.
\newblock {\em Nagoya Math. J.}, 66:183--202, 1977.

\bibitem[Ono09]{Ono_unearthing}
Ken Ono.
\newblock Unearthing the visions of a master: harmonic {M}aass forms and number
  theory.
\newblock In {\em Current developments in mathematics, 2008}, pages 347--454.
  Int. Press, Somerville, MA, 2009.

\bibitem[ORTL15]{MR3357517}
Ken Ono, Larry Rolen, and Sarah Trebat-Leder.
\newblock Classical and umbral moonshine: connections and {$p$}-adic
  properties.
\newblock {\em J. Ramanujan Math. Soc.}, 30(2):135--159, 2015.

\bibitem[Ram88]{MR947735}
Srinivasa Ramanujan.
\newblock {\em The lost notebook and other unpublished papers}.
\newblock Springer-Verlag, Berlin, 1988.
\newblock With an introduction by George E. Andrews.

\bibitem[Ram00]{MR2280843}
Srinivasa Ramanujan.
\newblock {\em Collected papers of {S}rinivasa {R}amanujan}.
\newblock AMS Chelsea Publishing, Providence, RI, 2000.
\newblock Edited by G. H. Hardy, P. V. Seshu Aiyar and B. M. Wilson, Third
  printing of the 1927 original, With a new preface and commentary by Bruce C.
  Berndt.

\bibitem[Shi71]{Shi_IntThyAutFns}
Goro Shimura.
\newblock {\em Introduction to the arithmetic theory of automorphic functions}.
\newblock Publications of the Mathematical Society of Japan, No. 11. Iwanami
  Shoten, Publishers, Tokyo, 1971.
\newblock Kan{\^o} Memorial Lectures, No. 1.

\bibitem[Shi73a]{MR0332663}
Goro Shimura.
\newblock On modular forms of half integral weight.
\newblock {\em Ann. of Math. (2)}, 97:440--481, 1973.

\bibitem[Shi73b]{MR0318162}
Goro Shimura.
\newblock On the factors of the jacobian variety of a modular function field.
\newblock {\em J. Math. Soc. Japan}, 25:523--544, 1973.

\bibitem[Sko85]{Sko_Thesis}
Nils-Peter Skoruppa.
\newblock {\em {\"U}ber den {Z}usammenhang zwischen {J}acobiformen und
  {M}odulformen halbganzen {G}ewichts}.
\newblock Bonner Mathematische Schriften [Bonn Mathematical Publications], 159.
  Universit{\"a}t Bonn Mathematisches Institut, Bonn, 1985.
\newblock Dissertation, Rheinische Friedrich-Wilhelms-Universit{\"a}t, Bonn,
  1984.

\bibitem[Sko90a]{MR1072974}
Nils-Peter Skoruppa.
\newblock Binary quadratic forms and the {F}ourier coefficients of elliptic and
  {J}acobi modular forms.
\newblock {\em J. Reine Angew. Math.}, 411:66--95, 1990.

\bibitem[Sko90b]{MR1096975}
Nils-Peter Skoruppa.
\newblock Developments in the theory of {J}acobi forms.
\newblock In {\em Automorphic functions and their applications ({K}habarovsk,
  1988)}, pages 167--185. Acad. Sci. USSR, Inst. Appl. Math., Khabarovsk, 1990.

\bibitem[Sko90c]{MR1074485}
Nils-Peter Skoruppa.
\newblock Explicit formulas for the {F}ourier coefficients of {J}acobi and
  elliptic modular forms.
\newblock {\em Invent. Math.}, 102(3):501--520, 1990.

\bibitem[Sko91]{MR1116103}
Nils-Peter Skoruppa.
\newblock Heegner cycles, modular forms and {J}acobi forms.
\newblock {\em S\'em. Th\'eor. Nombres Bordeaux (2)}, 3(1):93--116, 1991.

\bibitem[Sko08]{MR2512363}
Nils-Peter Skoruppa.
\newblock Jacobi forms of critical weight and {W}eil representations.
\newblock In {\em Modular forms on {S}chiermonnikoog}, pages 239--266.
  Cambridge Univ. Press, Cambridge, 2008.

\bibitem[SS77]{MR0472707}
J.-P. Serre and H.~M. Stark.
\newblock Modular forms of weight {$1/2$}.
\newblock In {\em Modular functions of one variable, {VI} ({P}roc. {S}econd
  {I}nternat. {C}onf., {U}niv. {B}onn, {B}onn, 1976)}, pages 27--67. Lecture
  Notes in Math., Vol. 627. Springer, Berlin, 1977.

\bibitem[Stu82]{MR660380}
Jacob Sturm.
\newblock Theta series of weight {$3/2$}.
\newblock {\em J. Number Theory}, 14(3):353--361, 1982.

\bibitem[SZ88]{MR958592}
Nils-Peter Skoruppa and Don Zagier.
\newblock Jacobi forms and a certain space of modular forms.
\newblock {\em Invent. Math.}, 94(1):113--146, 1988.

\bibitem[Vig77]{MR0485739}
M.-F. Vign{\'e}ras.
\newblock Facteurs gamma et \'equations fonctionnelles.
\newblock In {\em Modular functions of one variable, {VI} ({P}roc. {S}econd
  {I}nternat. {C}onf., {U}niv. {B}onn, {B}onn, 1976)}, pages 79--103. Lecture
  Notes in Math., Vol. 627. Springer, Berlin, 1977.

\bibitem[Wal80]{MR577010}
J.-L. Waldspurger.
\newblock Correspondance de {S}himura.
\newblock {\em J. Math. Pures Appl. (9)}, 59(1):1--132, 1980.

\bibitem[Wal81]{MR646366}
J.-L. Waldspurger.
\newblock Sur les coefficients de {F}ourier des formes modulaires de poids
  demi-entier.
\newblock {\em J. Math. Pures Appl. (9)}, 60(4):375--484, 1981.

\bibitem[Wal91]{MR1103429}
Jean-Loup Waldspurger.
\newblock Correspondances de {S}himura et quaternions.
\newblock {\em Forum Math.}, 3(3):219--307, 1991.

\bibitem[Wat37]{MR1577032}
G.~N. Watson.
\newblock The {M}ock {T}heta {F}unctions (2).
\newblock {\em Proc. London Math. Soc.}, S2-42(1):274, 1937.

\bibitem[Zag09]{zagier_mock}
Don Zagier.
\newblock Ramanujan's mock theta functions and their applications (after
  {Z}wegers and {O}no-{B}ringmann).
\newblock {\em Ast\'erisque}, (326):Exp. No. 986, vii--viii, 143--164 (2010),
  2009.
\newblock S{\'e}minaire Bourbaki. Vol. 2007/2008.

\bibitem[Zwe02]{zwegers}
Sander Zwegers.
\newblock {\em {Mock Theta Functions}}.
\newblock PhD thesis, Utrecht University, 2002.

\end{thebibliography}

\end{document}